\documentclass[a4paper]{article}
\usepackage{graphicx}        
\usepackage{multicol} 
\usepackage[bottom]{footmisc}
\usepackage{color,amssymb,amsmath,amsfonts,epsfig,bbm,ifthen,graphicx,twoopt,tabularx,xargs,a4wide,answers}
\usepackage[utf8]{inputenc}

\usepackage{amsthm}
 \usepackage[numbers]{natbib}

\usepackage{hyperref}
 \hypersetup{colorlinks=true,citecolor=blue, urlcolor=blue, linkcolor=black, filecolor=black}

\usepackage{enumitem,verbatim}

\usepackage[textwidth=2cm,textsize=footnotesize]{todonotes}    

\usepackage{cleveref}
\newtheorem{theorem}{Theorem}[section]
\crefname{theorem}{theorem}{theorems}
\Crefname{Theorem}{Theorem}{Definitions}

\newtheorem{definition}{Definition}[section]
\crefname{definition}{definition}{definitions}
\Crefname{Definition}{Definition}{Definitions}

\newtheorem{corollary}[theorem]{Corollary}
\crefname{corollary}{corollary}{corollaries}
\Crefname{Corollary}{Corollary}{Corollaries}

\newtheorem{proposition}[theorem]{Proposition}
\crefname{proposition}{proposition}{propositions}
\Crefname{Proposition}{Proposition}{Propositions}

\newtheorem{lemma}[theorem]{Lemma}
\crefname{lemma}{lemma}{lemmas}
\Crefname{Lemma}{Lemma}{Lemmas}

\newtheorem{example}{Example}[section]
\crefname{example}{example}{examples}
\Crefname{example}{Example}{Examples}

\newtheorem{remark}{Remark}[section]
\crefname{remark}{remark}{remarks}
\Crefname{remark}{Remark}{Remarks}

\crefname{assumption}{assumption}{assumptions}
\Crefname{assumption}{Assumption}{Assumptions}

\crefname{enumi}{point}{point}
\Crefname{enumi}{Point}{Point}

\theoremstyle{remark}

\def\lcag{\mathbbm{G}}
\def\adjoint{\mathsf{H}}

\def\povm{p.o.v.m.}
\def\caos{o.s.}
\def\cagos{g.o.s.}

\def\wot{w.o.t.}
\def\lca{l.c.a.}

\newcommand{\fundef}[1]{{
\begin{array}{lcl} #1 \end{array} }} 
\newcommand{\norm}[1]{{
    \left\| #1 \right\|}} 
\newcommand{\iseg}[1]{{\left\llbracket #1 \right\rrbracket}} 

\newcommandx\cspanarg[2][1=]{\ensuremath{\overline{\mathrm{Span}}^{#1}\left(#2\right)}}

\newcommand{\gramian}[2]{{\left[ #1, #2 \right]}} 
\newcommand{\rank}{{\rm rank}}
\newcommand{\spec}{{\rm \spec}}

\newcommand{\range}{\mathrm{Im}}

\newcommand{\domain}{\mathcal{D}}

\newcommand{\gramianisometricallyembedded}{\; \mathrel{\substack{
      \textstyle\subseteq \\ \textstyle\sim}} \;}
\newcommand{\gramianisomorphic}{\; \mathrel{\substack{
      \textstyle\sim\\[-0.5ex]\textstyle=}} \;}
\newcommand{\quotient}[2]{{\raisebox{.2em}{$#1$}\left/\raisebox{-.2em}{$#2$}\right.}}

\newcommand{\tnorm}[1]{{\left\vert\kern-0.25ex\left\vert\kern-0.25ex\left\vert #1 
    \right\vert\kern-0.25ex\right\vert\kern-0.25ex\right\vert}} 

\newcommand{\filterspecrep}[1]{\hat{F}_{#1}}

\newcommand{\filterprocess}[1]{F_{#1}}
\newcommand{\specreptransfer}[1]{\hat{\cS}_{#1}}
\newcommand{\processtransfer}[1]{\cS_{#1}}

\newcommand{\essuparg}[1]{\mathop{#1\text{-}\mathrm{essup}}}

\def\unitcircle{\mathbb{U}}

\def\rset{\mathbb{R}}
\def\cset{\mathbb{C}}
\def\zset{\mathbb{Z}}
\def\nset{\mathbb{N}}

  {\begin{enumerate}}%
  {\end{enumerate}}

\def\rmi{\mathrm{i}}
\def\rme{\mathrm{e}}
\def\rmd{\mathrm{d}}

\newcommandx{\aslim}[1]{\ensuremath{\stackrel{#1\text{a.s.}}{\longrightarrow}}}  

\newcommand{\1}{\mathbbm{1}}
\newcommand{\indi}[1]{\mathbbm{1}_{\left\{#1\right\}}}

\def\bX{\mathbf{X}}

\def\eqsp{\;}

\newcommand{\aop}{\mathrm{Q}}
\newcommand{\bop}{\mathrm{T}}
\newcommand{\fop}{\mathrm{F}}

\def\calA{\mathcal{A}}

\def\cF{\mathcal{F}}
\def\cE{\mathcal{E}}

\def\cH{\mathcal{H}}
\def\cG{\mathcal{G}}
\def\cI{\mathcal{I}}

\def\borel{\mathcal{B}}
\newcommand{\pscal}[2]{\left\langle #1, #2 \right\rangle}

\def\b0{{\bf 0}}

\def\bX{{\bf X}}

\def\cA{\mathcal{A}}

\def\cC{\mathcal{C}}

\def\cE{\mathcal{E}}
\def\cF{\mathcal{F}}
\def\cG{\mathcal{G}}
\def\cH{\mathcal{H}}
\def\cI{\mathcal{I}}

\def\cK{\mathcal{K}}
\def\cL{\mathcal{L}}
\def\cM{\mathcal{M}}

\def\cO{\mathcal{O}}

\def\cS{\mathcal{S}}

\def\tr{\mathrm{Tr}}

\newcommand\lspan[1]{\mathrm{Span}\left(#1\right)}
\newcommand\cspan[1]{\overline{\mathrm{Span}}\left(#1\right)}
\newcommand\algo[1]%
{
    \begin{center} %
    \begin{tabular} {||p{10 cm}l ||}%
               #1 &  \\
    \hline
    \end{tabular}
    \end{center}
}

\newcommand{\chunk}[4][]%
{\ifthenelse{\equal{#1}{}}{\ensuremath{{#2}_{#3:#4}}}{\ensuremath{#2^#1}_{#3:#4}}
}

\def\esp{\mathbb{E}}
\newcommandx\prob[2][1=,2=]{\ensuremath{{\mathbb P}_{#1}^{#2}}}
\newcommand{\PP}[1][]{\ifthenelse{\equal{#1}{}}{\ensuremath{\mathbb{P}}}{\ensuremath{\mathbb{P}\left( #1 \right)}}}
\newcommandx{\PParg}[2][1=]{\PP_{#1}\left(#2\right)}
\newcommand{\PE}[1][]{\ifthenelse{\equal{#1}{}}{\esp}{\ensuremath{{\mathbb E}\left[ #1 \right]}}}
\newcommandx{\PEarg}[2][1=]{\PE_{#1}\left[#2\right]}
\newcommand{\PVar}{\ensuremath{\operatorname{Var}}}
\newcommandx\var[2][1=]{\ensuremath{\PVar_{#1}\left( #2\right)}}
\newcommandx\cvar[3][1=]{\ensuremath{\PVar_{#1}\left( \left. #2 \right| #3 \right)}}

\newcommand{\PCov}{\ensuremath{\operatorname{Cov}}}
\newcommand\Cov[2]{\PCov\left(#1,#2\right)}
\newcommandx\cov[3][1=]{\ensuremath{\mathrm{Cov}_{#1}\left( #2,#3 \right)}}
\newcommandx\ccov[3][1=]{\ensuremath{\mathrm{Cov}_{#1}\left( \left. #2 \right| #3 \right)}}

\newcommand{\CPP}[3][]
{\ifthenelse{\equal{#1}{}}{\PP\left[\left. #2 \, \right| #3 \right]}{\mathbb{P}_{#1}\left(\left. #2 \, \right | #3 \right)}}
\newcommand{\CPE}[3][]
{\ifthenelse{\equal{#1}{}}{\PE\left[ \left. #2 \right| #3
    \right]}{\mathbb{E}_{#1} \left[ \left. #2 \right| #3 \right]}}
\newcommandx\cprob[4][1=,2=]{\ensuremath{\PP_{#1}^{#2}\left[ \left. #3 \right|
      #4 \right]}}
\newcommandx\HMCP[2][1=]{
\ifthenelse{\equal{#1}{}}{\PP_{#2}}{\PP_{#1,#2}}
}
\newcommandx\HMCE[2][1=]{
\ifthenelse{\equal{#1}{}}{\PE_{#2}}{\PE_{#1,#2}}
}
\newcommandx\HMCEarg[3][1=]{
\ifthenelse{\equal{#1}{}}{\PE_{#2}\left[#3\right]}{\PE_{#1,#2}\left[#3\right]}
}

\def\loikhi2{\mathbf{\chi^2}}

\newcommandx{\proj}[2]{\ensuremath{\operatorname{proj}\left( \left. #1\right|#2\right)}}

\newcommand{\g}{\ensuremath{g}}

\newcommand{\URoot}{\ensuremath{R}}
\newcommand{\UCov}[1][]%
{%
\ifthenelse{\equal{#1}{}}{\URoot \URoot^t}{\URoot_{#1} \URoot^t_{#1}}%
}

\newcommand{\VRoot}{\ensuremath{S}}
\newcommand{\VCov}[1][]%
{%
\ifthenelse{\equal{#1}{}}{\VRoot \VRoot^t}{\VRoot_{#1} \VRoot^t_{#1}}%
}

\newcommand{\LDX}[2]{\ensuremath{L}}

\newcommand{\postdx}[3][]%
{%
\ifthenelse{\equal{#1}{}}{\ensuremath{\psi_{#2|#3}}}{\ensuremath{\psi_{#1,#2|#3}}}%
}

\newcommand{\epostdx}[3][]%
{%
\ifthenelse{\equal{#1}{}}{\ensuremath{\hat{\psi}_{#2|#3}}}{\ensuremath{\hat{\psi}_{#1,#2|#3}}}%
}

\newcommandx{\predx}[3][1=\bX]{#1_{#2|#3}}   
\newcommand{\predpx}[3][]%
{%
\ifthenelse{\equal{#1}{}}{\ensuremath{\varphi_{#2|#3}}}{\ensuremath{\varphi_{#1,#2|#3}}}%
}
\newcommandx\cesp[4][1=,2=]{\ensuremath{{\mathbb E}_{#1}^{#2}\left[ \left. #3 \right| #4 \right]}}

\newcommand{\filt}[2][]%
{%
\ifthenelse{\equal{#1}{}}{\ensuremath{\phi_{#2}}}{\ensuremath{\phi_{#1,#2}}}%
}
\newcommand{\pred}[3][]%
{%
\ifthenelse{\equal{#1}{}}{\ensuremath{\phi_{#2|#3}}}{\ensuremath{\phi_{#1,#2|#3}}}%
}
\newcommand{\post}[3][]%
{%
\ifthenelse{\equal{#1}{}}{\ensuremath{\phi_{#2|#3}}}{\ensuremath{\phi_{#1,#2|#3}}}%
}
\newcommand{\logl}[2][]%
{%
\ifthenelse{\equal{#1}{}}{\ensuremath{\ell_{#2}}}{\ensuremath{\ell_{#1,#2}}}%
}
\newcommand{\lhood}[2][]%
{%
\ifthenelse{\equal{#1}{}}{\ensuremath{\mathrm{L}_{#2}}}{\ensuremath{\mathrm{L}_{#1,#2}}}%
}
\newcommand{\cc}[2][]%
{%
\ifthenelse{\equal{#1}{}}{\ensuremath{c_{#2}}}{\ensuremath{c_{#1,#2}}}%
}
\newcommand{\forvar}[2][]%
{%
\ifthenelse{\equal{#1}{}}{\ensuremath{\alpha_{#2}}}{\ensuremath{\alpha_{#1,#2}}}%
}
\newcommand{\nforvar}[2][]%
{%
\ifthenelse{\equal{#1}{}}{\ensuremath{\bar{\alpha}_{#2}}}{\ensuremath{\bar{\alpha}_{#1,#2}}}%
}

\newcommand{\BK}[2][]%
{%
\ifthenelse{\equal{#1}{}}{\ensuremath{\mathrm{\mathrm{B}}_{#2}}}{\ensuremath{\mathrm{B}_{#1,#2}}}%
}
\newcommand{\filtfunc}[2][]%
{%
\ifthenelse{\equal{#1}{}}{\ensuremath{\tau_{#2}}}{\ensuremath{\tau_{#1,#2}}}%
}

\newcommand{\filtmean}[2][]
{\ifthenelse{\equal{#1}{}}{{\ensuremath{\hat{X}_{#2|#2}}}}{\ensuremath{\hat{X}_{#1,#2|#2}}}
}
\newcommand{\filtcov}[2][]
{\ifthenelse{\equal{#1}{}}{\ensuremath{\Sigma_{#2|#2}}}{\ensuremath{\Sigma_{#1,#2|#2}}}}
\newcommand{\postmean}[3][]
{\ifthenelse{\equal{#1}{}}{\ensuremath{\hat{X}_{#2|#3}}}{\ensuremath{\hat{X}_{#1,#2|#3}}}
}
\newcommand{\postcov}[3][]
{\ifthenelse{\equal{#1}{}}{\ensuremath{\Sigma_{#2|#3}}}{\ensuremath{\Sigma_{#1,#2|#3}}}}

\newcommandx{\QEM}[4][1=,4=]{\ensuremath{\mathcal{Q}_{#1}(#4;#2 \, ; #3)}}

\newcommand{\argmax}{\ensuremath{\mathop{\mathrm{argmax}}}}

\def\tore{\mathbb{T}}
\def\btore{\mathcal{B}(\tore)}

\newcommandx\sequence[3][2=t,3=\zset]{\ensuremath{\left(#1_{#2}\right)_{#2 \in #3 }}}
\newcommandx\dsequence[4][3=t,4=\zset]{\ensuremath{\left( (#1_{#3}, #2_{#3})\right)_{#3 \in #4}}}
\newcommandx{\sequencen}[2][2=n\in\nset]{\ensuremath{\left(#1\right)_{#2}}}

\def\bpm{\left[\begin{matrix}}
\def\epm{\end{matrix}\right]}
\def\bma{\begin{matrix}}
\def\ema{\end{matrix}}

\newcommand{\be}{\begin{equation}}     
\newcommand{\ee}{\end{equation}}        

\newcommand{\eg}{e.g.}

\newcommandx\lnorm[3][1=]{\left\lVert #2 \right\rVert^{#1}_{#3}}

\newcommand{\Id}{\mathrm{Id}}

\newcommandx\supnorm[2][1=]{| #2 |^{#1}_\infty}

\newcommandx\ball[3][1=]{\mathrm{B}_{#1} (#2,#3)}

\newcommandx{\prohosym}[1][1=]{{\boldsymbol\rho}_{#1}}
\newcommandx{\proho}[3][1=]{\prohosym{#1}\left(#2,#3\right)}

\newcommandx{\pp}[1][1=\mu]{\ensuremath{#1\-\mathrm{a.e.}}}

\renewcommand{\-}{\mbox{-}}

\newcommandx{\as}[1][1=\PP]{\ensuremath{#1\-\mathrm{a.s.}}}

\newcommandx{\oscnorm}[3][1=,3=]{\operatorname{osc}^{#1}_{#3}\left(#2\right)}

\newcommandx{\tvdist}[3][1=]{\ensuremath{d^{#1}_{\mathrm{TV}}}(#2,#3)}

\newcommandx{\VnormFunc}[3][1=]{\ensuremath{\left|#2\right|_{\mathrm{#3}}^{#1}}}

\newcommand{\continuousfunctionset}[1]{\mathrm{C}_b(#1)}

\newcommand{\lipschitzfunctionset}[1]{\mathrm{Lip}(#1)}
\newcommand{\boundedlipschitzfunctionset}[1]{\mathrm{Lip}_b(#1)}

\newcommand{\simplemeasfunctionsetarg}[1]{\mathbb{F}_s\left(#1\right)}
\newcommand{\weakmeasfunctionsetarg}[1]{\mathbb{F}_{\cO}\left(#1\right)}
\newcommandx\functionsetarg[2][1=]{
\ifthenelse{\equal{#1}{c}}{\continuousfunctionset{\mathsf{#2}}}
{\ifthenelse{\equal{#1}{bc}}{\mathrm{C}_b(#2)}
{\ifthenelse{\equal{#1}{u}}{\mathrm{U}(#2)}
{\ifthenelse{\equal{#1}{bu}}{\mathrm{U}_b(#2)}
{\ifthenelse{\equal{#1}{l}}{\lipschitzfunctionset{#2}}
{\ifthenelse{\equal{#1}{bl}}{\boundedlipschitzfunctionset{#2}}
{\mathbb{F}_{#1}(#2)}
}}}}}}

\newcommandx\functionsetspec[1][1=]{
\ifthenelse{\equal{#1}{c}}{\mathrm{C}}
{\ifthenelse{\equal{#1}{bc}}{\mathrm{C}_b}
{\ifthenelse{\equal{#1}{u}}{\mathrm{U}}
{\ifthenelse{\equal{#1}{bu}}{\mathrm{U}_b}
{\ifthenelse{\equal{#1}{l}}{\mathrm{Lip}}
{\ifthenelse{\equal{#1}{bl}}{\mathrm{Lip}_b}
{\mathbb{F}_{#1}}
}}}}}}

\newcommandx{\taboo}[3][1=,3=]{\left(\leftidx{_#1}{#2}{}\right){^{#3}}}

\newcommandx\vectornorm[2][1=]{\left| #2 \right|^{#1}}  

\newcommand{\ensemble}[2]{\left\{#1\,:\eqsp #2\right\}}
\newcommand{\set}[2]{\ensemble{#1}{#2}}

\newcommandx{\plim}[1]{\ensuremath{\stackrel{#1\-\text{prob}}{\longrightarrow}}}
\newcommandx{\dlim}[1]{\ensuremath{\stackrel{#1}{\Longrightarrow}}}

\newcommandx\measureset[3][1=\mathrm{s},3=]{\mathbb{M}^{#3}_{#1}(#2)}
\newcommandx\measuresetmetric[2][1=1]{\mathbb{M}_{#1}(\mathcal{B}(\mathsf{#2}))}  
\newcommandx\measuresetspec[1][1=\mathrm{s}]{\mathbb{M}_{#1}}

\newcommand{\abs}[1]{\left\vert #1 \right\vert}

\newcommandx\canonicalkernel[1][1=P]{\mathbb{K}_{#1}}

\usepackage{mathrsfs}
\usepackage{stmaryrd}

\numberwithin{equation}{section}

\title{Weakly stationary stochastic processes valued in a separable Hilbert
  space: Gramian-Cramér representations and applications}

\author{Amaury Durand\,\footnote{LTCI, Telecom Paris, Institut Polytechnique de
    Paris.}\;\,\footnote{EDF R\&D, TREE, E36, Lab Les Renardieres, Ecuelles, 77818 Moret sur Loing, France. } \and François Roueff\,$^*$
   \addtocounter{footnote}{-3}
   \footnote{Math Subject Classification. Primary: 60G12; Secondary: 47A56, 46G10} \addtocounter{footnote}{-1}
     \addtocounter{footnote}{-1}
 }
 \date{}
\begin{document}
\sloppy
\maketitle

\abstract{The spectral theory for weakly stationary processes valued
  in a separable Hilbert space has known renewed interest in the past
  decade. Here we follow earlier approaches which fully exploit the
  \emph{normal Hilbert module} property of the time domain. The key
  point is to build the Gramian-Cramér representation as an isomorphic
  mapping from the \emph{modular spectral domain} to the
  \emph{modular time domain}. We also discuss the general Bochner
  theorem and provide useful results on the composition and inversion
  of lag-invariant linear filters. Finally, we derive the
  Cramér-Karhunen-Loève decomposition and harmonic functional
  principal component analysis, which are established without relying
  on additional assumptions.}

\section{Introduction}
Spectral theory for weakly stationary time series has been originally developed in a very general fashion, starting from the seminal works by
\citeauthor{kolmogorov_stationary_1941},
\cite{kolmogorov_stationary_1941}, and spanning over several decades,
see \cite{holmes-signalprocessing} and the references therein.  These
foundations include time domain and frequency domain analyses,
Cramér (or spectral) representations, the Herglotz theorem and linear filters.  In
\cite{kolmogorov_stationary_1941, holmes-signalprocessing} the adopted
framework is that of a bi-sequence $X=(X_t)_{t\in\zset}\in\cH^\zset$
valued in a Hilbert space $(\cH,\pscal{\cdot}\cdot_{\cH})$ and weakly
stationary in the sense that $\pscal{X_s}{X_t}_{\cH}$ only depends on
the lag $s-t$. In this framework, a linear filter is a
linear operator on $H_X$ onto $H_X$ which commutes with the lag
operator $U^X$, where $H_X$ is the closure in $\cH$ of the linear span of
$(X_t)_{t\in\zset}$ and $U^X$ is the operator defined on $H_X$ by
mapping $X_t$ to $X_{t+1}$ for all $t\in\zset$. As explained in
\cite[Section~3]{holmes-signalprocessing}, a complete description of such a filter is
given in the spectral domain by its transfer function. Let us
recall the essential formulas which summarize what this means.
In \cite{holmes-signalprocessing}, the spectral
theory follows from and start with
the \emph{canonical representation} of the lag operator $U^X$
above, namely
\begin{equation}
  \label{eq:homes-U-spectral}
U^X=\int_{\tore} \rme^{\rmi\lambda}\;\xi(\rmd\lambda)\;,
\end{equation}
where $\tore=\rset / (2 \pi \zset)$ and $\xi$ is the spectral measure
of $U^X$ (which is a measure valued in the space of operators on $H_X$
onto itself). This corresponds to
\cite[Eq.~(8)]{holmes-signalprocessing} with a slightly different
notation. Then defining $\hat X$ as $\xi(\cdot)X_0$ (thus a measure valued
in $H_X$), one gets the celebrated Cramér representation (see
\cite[Eq.~(13a)]{holmes-signalprocessing} again with a slightly
different notation)
\begin{equation}
  \label{eq:homes-x-spectral}
X_t=\int_{\tore} \rme^{\rmi\lambda\,t}\;\hat X(\rmd\lambda)\;,\quad t\in\zset\;.
\end{equation}
An other consequence of~(\ref{eq:homes-U-spectral}) is what is called
the Herglotz theorem in \cite[Eq.~(9)]{holmes-signalprocessing},
summarized by the formula
\begin{equation}
  \label{eq:homes-x-herglotz}
\pscal{X_s}{X_t}_{\cH}=\int_{\tore} \rme^{\rmi\lambda\,(s-t)}\;\mu(\rmd\lambda)\;,\quad s,t\in\zset\;,
\end{equation}
where $\mu=\pscal{\xi(\cdot)X_0}{X_0}_{\cH}$ is a non-negative measure
on $(\tore,\borel(\tore))$. Interpreting the right-hand side
of~(\ref{eq:homes-x-herglotz}) as the scalar product of the two
functions $\rme_s:\lambda\mapsto\rme^{\rmi\lambda s}$ and
$\rme_t:\lambda\mapsto\rme^{\rmi\lambda t}$ in
$L^2(\tore,\borel(\tore),\mu)$, Relation~(\ref{eq:homes-x-herglotz}) is simply saying that the
Cramér representation~(\ref{eq:homes-x-spectral}) mapping $\rme_t$ to $X_t$ is
isometric. Following this interpretation, one can extend this
isometric mapping to a unitary operator between the two isomorphic Hilbert spaces
$L^2(\tore,\borel(\tore),\mu)$ and $H_X$, respectively refered to as
the \emph{spectral domain} and the \emph{time domain}. In particular the output of
a linear filter with transfer function
$\Phi\in L^2(\tore,\borel(\tore),\mu)$ is given by
\begin{equation}
  \label{eq:homes-filter}
Y_t=\int \rme^{\rmi\lambda t}\,\Phi(\lambda)\;\hat{X}(\rmd\lambda)\;,\quad t\in\zset\;,
\end{equation}
or in other words, $Y_t$ is the image of the function $\rme_t\Phi$ by the extended
unitary operator that maps the spectral domain to the time domain.

The spectral
theory~(\ref{eq:homes-U-spectral})--(\ref{eq:homes-filter}) applies to
univariate times series by letting $\cH$ be the space
$L^2(\Omega,\cF,\PP)$ of $\cset$-valued random variables on
$(\Omega,\cF)$ with finite second order moment. It also applies to
multivariate time series by taking
$\cH = L^2(\Omega, \cF, \cset^q, \PP) = (L^2(\Omega, \cF, \PP))^q$ and
to functional time series by letting $\cH$ be the Bochner space
$L^2(\Omega,\cF,L^2(0,1),\PP)$ of measurable mappings
$V:\Omega\to L^2(0,1)$ such that
$\PEarg{\left\|V\right\|_{L^2(0,1)}^2}<\infty\;,$ where
$\left\|\cdot\right\|_{L^2(0,1)}$ here denotes the norm endowing the
Hilbert space $L^2(0,1)$. However, in
\cite[Section~7]{holmes-signalprocessing},
\citeauthor{holmes-signalprocessing} argues that important
generalizations are needed for multivariate time series. This claim
applies even more to functional time series. In this paper, we address such
a generalization valid in the functional context. Related issued have
been recently considered in \cite{panaretos13, panaretos13Cramer,
  these-tavakoli-2015} where, in particular, the authors derive a
functional version of the Cramér representation which relies on a
spectral density operator defined under strong assumptions on the
covariance structure of the time series. Under the same assumption,
\cite{panaretos13Cramer} introduced filters whose transfer functions
are valued in a restricted set of operators and this was latter
generalized to bounded-operator-valued transfer functions in
\cite[Section~2.5]{these-tavakoli-2015} (see also
\cite[Appendix~B.2.3]{vandelft2018}). An application of this spectral
theory to dimension reduction is proposed by the means of a harmonic
functional principal components analysis (see
\cite{panaretos13Cramer,hoermann15}).  A more general approach is
adopted in \cite{van_delft_note_herglots_2020} where the authors
provide a definition of operator-valued measures from which they
derive a functional version of the Herglotz theorem, the functional
Cramér representation, the definition of linear filters with
bounded-operator-valued transfer functions and a harmonic functional
principal component analysis in the case where the spectral measure
has finitely many discontinuities.

To complement these recent works, we here focus on the Gramian structure of the space
$\cH = L^2(\Omega,\cF,L^2(0,1),\PP)$. This approach extends naturally
the results gathered in \cite{masani_recent_1966} for the multivariate
case and where the Gramian nature of the covariance matrix plays a key
role. In particular, in the multivariate case, the lag operator $U^X$
is not only (scalar product) isometric on $H_X$ but also 
Gramian-isometric on the larger space
$\cspan{\aop X_t\,,\, t \in \zset, \aop \in \cset^{q\times q}}$. In
the functional case, we exhibit the Gramian structure of
$\cH = L^2(\Omega,\cF,L^2(0,1),\PP)$ by making it a \emph{normal
  Hilbert module}. As a result, the time domain space $H_X$ of
\cite{holmes-signalprocessing} is replaced by the
\emph{modular time domain}
\begin{equation}
  \label{eq:ts-cHX}
\cH^X=\cspan{\aop X_t\,,\,t \in \zset, \,\aop\in\cL_b(L^2(0,1))}\;,  
\end{equation}
where $\cL_b(L^2(0,1))$ denotes the space of bounded operators on
$L^2(0,1)$ onto itself. In comparison, in the definition of $H_X$ used
in \cite{holmes-signalprocessing}, $\aop$ is restricted to be a scalar
operator. Thus, while $H_X$ is a subspace of $\cH$ seen as a Hilbert
space, $\cH^X$ is a submodule of $\cH$ seen as a normal Hilbert
module. Based on this simple fact, a natural path for achieving
and fully exploiting a Cramér representation on $\cH^X$ is:
\begin{enumerate}[label=Step~\arabic*)]
\item\label{item:U-step} Interpret the
  representation~(\ref{eq:homes-U-spectral}) as the one of a Gramian-isometric
  operator on $\cH^X$ (and not only an scalar product
isometric operator on $H_X$). 
\item\label{item:Cramer-step} Deduce that the Cramér
representation~(\ref{eq:homes-x-spectral}) can effectively be extended
as a Gramian-isometric operator mapping $L^2(0,1)\to L^2(0,1)$-operator-valued functions on
$(\tore,\borel(\tore))$ to an element of $\cH^X$.
\item\label{item:herglotz-step} As a first consequence, the
scalar product isometric relation~(\ref{eq:homes-x-herglotz}) is
extended to
\begin{equation}
  \label{eq:fts-herglotz}
\gramian{X_s}{X_t}_{\cH}=\int_{\tore} \rme^{\rmi\lambda\,(s-t)}\;\nu(\rmd\lambda)\;,\quad s,t\in\zset\;,
\end{equation}
where, here, $\nu$ is an operator-valued measure on
$(\tore,\borel(\tore))$ and $\gramian{X_s}{X_t}_\cH = \Cov{X_s}{X_t}$. This Gramian-isometric relationship
corresponds to what is called the Herglotz theorem in the functional
time series case.
\item\label{item:filter-general} As a second consequence, the Cramér
  representation~(\ref{eq:homes-filter}) of a linear filter is
  extended to the case where the transfer function $\Phi$ is now an
  $L^2(0,1)\to L^2(0,1)$-operator-valued functions on
  $(\tore,\borel(\tore))$ (and not only a scalar valued functions on
  $(\tore,\borel(\tore))$). This raises the question, in particular,
  of the precise condition required on the transfer function to replace
  the condition  $\Phi\in L^2(\tore,\borel(\tore),\mu)$ of the scalar
  case. 
\item\label{item:composition-inversion-filter} An interesting
  consequence of~\ref{item:filter-general} is to
  study the composition of linear filters and deduce when and how it
  is possible to invert them.
\item\label{item:harmonic-pca} An other interesting consequence
  of~\ref{item:Cramer-step} is to derive the Cramér-Karhunen-Loève
  decomposition and the harmonic principal component analysis for any
  weakly stationary functional time series valued in a separable
  Hilbert space.
\end{enumerate}
In this contribution, we basically follow this path, up to the
following slight modifications.
\begin{enumerate}
\item We treat the more general case of a stochastic process
  $(X_t)_{t \in \lcag}$, where $(\lcag,+)$ is a locally compact
  Abelian (\lca) group set of indices and for each $t \in \lcag$,
  $X_t$ is a random variable defined on a probability space
  $(\Omega, \cF,\PP)$ and valued in a separable Hilbert space $\cH_0$
  (endowed with its Borel $\sigma$-field). Typical examples for
  $\lcag$ and $\cH_0$ are the ones of functional time series, namely
  $\lcag=\zset$ and $\cH_0=L^2(0,1)$ but, as far as spectral theory is
  concerned, the presentation of the results is not only more general
  (one can \textit{e.g.} take $\lcag=\rset$) but also more elegant in
  this general setting. Of
  course, in the discrete time case $\lcag=\zset$, any continuity
  condition imposed on a function defined on $\lcag$ is trivially
  satisfied. Such continuity conditions constitute a small price to
  pay (and the only one) in order to be able to treat the case of a
  general \lca\ group $\lcag$ rather than focusing on the 
  discrete time case alone.
\item For obvious practical reasons, it is usual to treat the mean of
  a stochastic process separately. Therefore we will assume that the
  process $(X_t)_{t \in \lcag}$ is centered.
\item We will consider the case where the separable Hilbert space $\cG_0$ in which
  the output of the filter is valued is different
  from $\cH_0$, the one of the input, that is, we replace $\aop\in\cL_b(\cH_0)$
  in~(\ref{eq:ts-cHX}) by $\aop\in\cL_b(\cH_0,\cG_0)$, the space of
  bounded operators from $\cH_0$ to $\cG_0$. This makes the results
  directly applicable in the case of different input and output
  spaces, especially in the case where they have different dimensions
  (so that they are not isomorphic). 
\end{enumerate}
The approach to derive a spectral theory
following~\ref{item:U-step}--~\ref{item:filter-general} is essentially
contained in \cite{kallianpur1971spectral,
  mandrekar-square-integrability,yuichiro-multidim}. Our main
contribution concerning these steps is to introduce all the
preliminary definitions required to understand them, to select the
most important results, to provide detailed proofs of the key points
and to bring forward this approach which offers an interesting
alternative to the ones recently proposed in \cite{panaretos13,
  panaretos13Cramer, these-tavakoli-2015, vandelft2018,
  van_delft_note_herglots_2020}.  A first benefit of the
Gramian-isometric approach is that it allows a concrete description of
the spectral domain rather than relying on the completion of a
pre-Hilbert space or on the compactification of a pointed convex cone
as used in \cite[Section~2.5]{these-tavakoli-2015} and
\cite{van_delft_note_herglots_2020}, respectively. A second benefit is
to make the Cramér representation much easier to exploit for deriving
useful general results.  This will be made apparent when establishing
the composition and inversion of filters
of~\ref{item:composition-inversion-filter}, which to our best
knowledge, appear to be novel in this degree of generality.
Similarly, our versions of the Cramér-Karhunen-Loève decomposition and
harmonic functional principal component analysis are not restricted to
the case where the spectral density operator is continuous or the
spectral measure has finitely many discontinuities as in
\cite{these-tavakoli-2015,van_delft_note_herglots_2020}. However, it
is important to note that, contrary to \cite{panaretos13,
  panaretos13Cramer, these-tavakoli-2015, hoermann15}, we do not
address the question of statistical estimation in the spectral
domain. The spectral theory we present applies to all weakly
stationary processes whereas statistical inference results require
specific assumptions. As in the univariate setting, the spectral
analysis of long memory processes necessitates assumptions and
technical developments beyond the ones used for
the spectral analysis of short memory processes. In the functional
setting, up to our knowledge, long memory processes have been mostly
studied in the time domain (see \eg\
\cite{rackauskas2011,characiejus2013central,CHARACIEJUS20142605,duker_limit_2018,
  Li-long-memory-2020}). Clarifying the general spectral theory that
applies to all functional weakly stationary processes is a first step
towards studying functional long memory processes in the spectral
domain as classically done for univariate long memory processes (see
\eg\ \cite[Section~2.4]{pipiras_taqqu_2017} about the celebrated
FARIMA processes). Such a study, however, is out of the scope of the
present paper.

The paper is organized as follows.  Basic definitions of
operator-valued measures, operator-valued functions (and the various
notions of measurability related to them) and Gramian-isometric
operators on normal Hilbert modules are assembled in
\Cref{sec:preliminaries}. \Cref{sec:more-preliminaries} contains some
preliminaries paving the way for describing the \emph{modular spectral
  domain}. In particular, we explain how to use normal Hilbert modules
for defining Gramian-orthogonally scattered
measures. \Cref{sec:spectral-analysis} contains the main results: 1)
we offer a synthesis of the results of \cite{kallianpur1971spectral,
  mandrekar-square-integrability,yuichiro-multidim} providing a
natural and complete spectral theory for weakly stationary processes
valued in a separable Hilbert space; 2) then, this approach is
exploited to address \ref{item:composition-inversion-filter}
and~\ref{item:harmonic-pca} above, successively.  All the proofs are
postponed in \Cref{sec:proofs} along with additional useful results.

\tableofcontents

\section{Basic definitions and notation}\label{sec:preliminaries}
\subsection{Operators, measurability and integrals}\label{sec:notations}
Basic definitions on linear operators can be found, for example, in \cite{Weidmann-operators-hilbert} and we refer the reader to  \cite[Chapter~1]{dinculeanu2011vector} for a nice overview of measurability and integration on Banach spaces. Throughout this paper, we will denote by $\cO(\cH, \cG)$ the set of linear operators $\aop$ from the (complex) Hilbert space $\cH$ to the (complex) Hilbert space $\cG$ whose domains, denoted by $\domain(\aop)$, are linear subspaces of $\cH$. We then denote by $\cL_b(\cH, \cG)$ its subset of
continuous operators, by $\cK(\cH, \cG)$ its subset of compact
continuous operators and, for all $p \in[1,\infty)$, by
$\cS_p(\cH, \cG)$ the Schatten-$p$ class of compact operators with
$\ell^p$ singular values. Schatten-$1$ and
Schatten-$2$ operators are usually referred to as \emph{trace-class}
and \emph{Hilbert-Schmidt} operators respectively. If $\cG = \cH$, we omit $\cG$ in the
notation of these operator sets. For $\aop \in \cL_b(\cH, \cG)$ we denote its adjoint by $\aop^\adjoint$. We denote by $\cL_b^+(\cH)$ the set of \emph{positive} operators i.e. the set of $\aop \in \cL_b(\cH)$ such that $\pscal{\aop x}{x}_{\cH} \geq 0$ for all $x\in\cH$. Similarly, $\cK^+(\cH)$ and $\cS_p^+(\cH)$ denote respectively the sets of positive compact and positive Schatten-$p$ operators. If $\aop \in \cK^+(\cH)$, $\aop^{1/2}$ denote the unique operator of $\cK^+(\cH)$ which satisfies
$\aop = \left( \aop^{1/2} \right)^2$.  The notation $\norm{\cdot}$ is used for the operator norm on $\cL_b(\cH, \cG)$ and $\norm{\cdot}_p$ is used for the Schatten-$p$ norm on $\cS_p(\cH, \cG)$. For $E \subset \cH$, we will used the notation $\cspanarg[\cH]{E}$  for the smallest linear subspace of $\cH$ which contains $E$ and is closed for the norm topology in $\cH$.

For a measurable space $(\Lambda, \calA)$ and a Banach space $(E,\norm{\cdot}_E)$, we denote by $\functionsetarg{\Lambda, \calA, E}$ the space of measurable functions from $(\Lambda,\cA)$ to $(E,\borel(E))$, where $\borel(E)$ is the Borel $\sigma$-field on $E$. For a non-negative measure $\mu$ on
$(\Lambda, \calA)$ and $p\in[1,\infty]$ , we denote by
$\cL^p(\Lambda, \calA, E, \mu)$ the space of functions
$f \in \functionsetarg{\Lambda, \calA, E}$ such that
$\int \norm{f}_E^p \, \rmd \mu$ (or $\essuparg{\mu} \norm{f}_E$ for
$p=\infty$) is finite and by $L^p(\Lambda, \calA, E, \mu)$ its
quotient space with respect to $\mu\text{-a.e.}$ equality. The
corresponding norm is denoted by
$\norm{\cdot}_{L^p(\Lambda, \calA, E, \mu)}$. The
Bochner integral is defined on $L^1(\Lambda, \calA, E, \mu)$ by linear
and continuous extension of the mapping $\1_Ax\to\mu(A)\,x$ defined for
$x\in E$ and $A\in\cA$ such that $\mu(A)<\infty$. In the particular case where $E$ is a space of linear operators between two Hilbert spaces $\cH$ and $\cG$, we use
the following weaker notion of measurability.
\begin{definition}[Simple measurability]\label{def:Fs-fiarma-fiarma}
  A function
$\Phi : \Lambda \to \cL_b(\cH,\cG)$ is said to be \emph{simply measurable} if
for all $x \in \cH$, $\lambda \mapsto \Phi(\lambda)x$ is measurable as a $\cG$-valued
function. The set of such functions is denoted by
$\simplemeasfunctionsetarg{\Lambda, \calA, \cH, \cG}$  or simply
$\simplemeasfunctionsetarg{\Lambda, \calA, \cH}$ if $\cG=\cH$. 
\end{definition}
Simple measurability is weaker than Borel measurability in the sense
that for all Banach spaces $\cE$ which are continuously embedded in
$\cL_b(\cH, \cG)$, a function in $\functionsetarg{\Lambda,
  \calA,\cE}$ is simply measurable from $\Lambda$ to
$\cL_b(\cH,\cG)$. Moreover, if $\cH$ and $\cG$ are separable and $\cE = \cK(\cH, \cG)$ or
$\cE = \cS_p(\cH, \cG)$ with $p \in \{1,2\}$, the class of simple measurable
functions valued in $\cE$ coincides with $\functionsetarg{\Lambda, \calA,\cE}$, see \Cref{lem:meas-schatten}. 

\subsection{Vector-valued and Positive Operator-Valued Measures}\label{sec:povm}
Measures valued in a Banach space, and in particular Positive
Operator-Valued Measures, are key in the spectral theory of functional
times series. This section gathers results on such measures. Details
can be found in
\cite{dinculeanu1967vector_measures,Berberian1966NotesOS}. First, we
recall that a measure $\mu$ defined on the measurable space
$(\Lambda, \calA)$ and valued in the Banach space $(E,\norm{\cdot}_E)$
is an $\calA \to E$ mapping such that, for any sequence
$(A_n)_{n \in \nset} \in \calA^\nset$ of pairwise disjoint sets,
$\mu \left(\bigcup_{n \in \nset} A_n \right) = \sum_{n \in \nset}
\mu(A_n)$, where the series converges in $E$, that is,
\begin{equation}
  \label{eq:banach-valued-measure-sigma-additive}
\lim_{N \to +\infty} \norm{\mu \left(\bigcup_{n \in \nset} A_n \right) - \sum_{n = 0}^N \mu(A_n)}_E = 0\; .  
\end{equation}
For such a measure $\mu$, the mapping
$$
\norm{\mu}_E : A \mapsto \sup \set{\sum_{i \in \nset} \norm{\mu(A_i)}_E}{(A_i)_{i \in \nset} \in \calA^\nset \text{ is a countable partition of } A}
$$
defines a non-negative measure on $(\Lambda, \calA)$ called the
\emph{variation measure} of $\mu$. For instance, if $E=\cS_1(\cH)$, we
write $\norm{\mu}_1$ since we use $\norm{\cdot}_1$ to denote the
Schatten-1 norm. Integrals of functions in
$L^1(\Lambda, \calA, \norm{\mu}_E)$ with respect to $\mu$ are
discussed in \cite[P.~120]{dinculeanu1967vector_measures}.  When
$\Lambda$ is a locally compact topological space and $\calA$ is the
Borel $\sigma$-field, an $E$-valued measure $\mu$ is said to be
\emph{regular} if for all $A \in \calA$ and $\epsilon > 0$, there
exist a compact set $K \in \calA$ and an open set $U \in \calA$ with
$K \subset A \subset U$ such that
$\norm{\mu(U \setminus K)}_E \leq \epsilon$.  The special case of
operator-valued measures is of particular interest to us and,
specifically, Positive Operator-Valued Measures (\povm's). We recall
that a sequence $(\aop_n)_{n \in \nset} \in \cL_b(\cH)^\nset$
converges to an operator $\aop \in \cL_b(\cH)$ in \emph{weak operator
  topology} (\wot) if for all $x \in \cH$,
$\displaystyle\lim_{n \to \infty} \pscal{\aop_n x}{x}_{\cH} =
\pscal{\aop x}{x}_{\cH}$.

\begin{definition}[Positive Operator-Valued Measures (\povm)]\label{def:povm}
  Let $(\Lambda, \calA)$ be a measurable space and $\cH$ be a
  Hilbert space. A Positive Operator-Valued Measure (\povm)
  on $(\Lambda, \calA, \cH)$ is a mapping
  $\nu : \calA \to \cL_b^+(\cH)$ such that for all sequences of
  disjoint sets $(A_n)_{n \in \nset} \in \calA^\nset $,
  \begin{equation}\label{eq:sigma-add}
  \nu \left(\bigcup_{n\in \nset} A_n \right) = \sum_{n \in \nset} \nu(A_n)
  \end{equation}
  where the series converges in $\cL_b^+(\cH)$ in \wot
\end{definition}
Note that the series in \eqref{eq:sigma-add} does not necessarily
converge in operator norm which implies that, in this definition, a
\povm\ does not need to be an $\cL_b(\cH)$-valued measure in the sense
of~(\ref{eq:banach-valued-measure-sigma-additive}). Therefore the
above definitions of integrals and regularity cannot be applied. This
is circumvented by noting that a \povm\ is entirely characterized by
the family of non-negative measures
$\set{\nu_{x} : A \mapsto x^\adjoint \nu(A) x}{x \in \cH}$. We refer
to Definition~14 and Theorem~20 in \cite{Berberian1966NotesOS} for
details about regular \povm's and to Theorem~9 in
\cite{Berberian1966NotesOS} for details about integration of bounded
scalar functions with respect to a \povm

When dealing with spectral operator
measures of weakly stationary processes valued in a separable Hilbert space, we
can rely on the additional trace-class property, which makes all the
previous definitions easier to handle and extend. 
\begin{definition}[Trace-class \povm]\label{def:tcpovm}
  Let $(\Lambda, \calA)$ be a measurable space, $\cH_0$ be a \emph{separable}
  Hilbert space and $\nu$ be a \povm\ on $(\Lambda, \calA, \cH_0)$. We say that
  $\nu$ is a \emph{trace-class}-\povm\ if it is
  $\cS_1^+(\cH_0)$-valued.
\end{definition}
The first advantage of a trace-class \povm\ is that it fits the
framework of vector-valued measures, namely, we have the following
result, whose proof can be found in \Cref{sec:proofs-prel-results}.
\begin{lemma}\label{lem:traceclasspovm}
  Let $(\Lambda, \calA)$ be a measurable space and $\cH_0$ be a \emph{separable}
  Hilbert space. Then a \povm\ $\nu$ on $(\Lambda, \calA, \cH_0)$ is trace-class if and
  only if $\nu(\Lambda)\in\cS_1(\cH_0)$. In this case, $\nu$ is an
  $\cS_1(\cH_0)$-valued measure (in the sense that
  \eqref{eq:sigma-add} holds in $\norm{\cdot}_1$-norm) with finite
  variation measure $\norm{\nu}_1 : A \mapsto
  \norm{\nu(A)}_1$. Moreover, $\nu$  is regular if and only if  $\norm{\nu}_1$ is regular.
\end{lemma}
Another advantage of trace-class \povm's is that they satisfy the Radon-Nikodym property. Namely, if  $\nu$ a trace-class \povm\ on $(\Lambda, \calA, \cH_0)$ and $\mu$ is a $\sigma$-finite non-negative measure on $(\Lambda, \calA)$, then $\norm{\nu}_1 \ll \mu$  (\textit{i.e.} for all $A \in \calA$, $\mu(A) = 0 \Rightarrow \norm{\nu}_1(A) = 0$),  if and only if there exists $g \in L^1(\Lambda, \calA, \cS_1(\cH_0), \mu)$ 
  such that $\rmd \nu = g \, \rmd \mu$, \textit{i.e.} for all
  $A \in \calA$,
  \begin{equation}\label{eq:density-povm}
    \nu(A) = \int_A g \, \rmd \mu \; . 
  \end{equation}
  In this case, $g$ is unique and is called the density of $\nu$ with respect
  to $\mu$ and we write $g = \frac{\rmd \nu}{\rmd \mu}$. This result is a consequence of  Theorem 1 in \cite[Chapter~III, Section~3]{diestel1977vector} because $\cS_1(\cH_0)$ is the dual of the separable space $\cK(\cH_0)$. 

\subsection{Normal Hilbert modules}\label{sec:hilbert-modules}
Modules extend the notion of vector spaces to the case where scalar
multiplication is replaced by a multiplicative operation with elements
of a ring. The case where the ring is $\cL_b(\cH_0)$ for a separable
Hilbert space $\cH_0$ is of particular interest for $\cH_0$-valued
random variables. In short, a normal Hilbert $\cL_b(\cH_0)$-module is
a Hilbert space endowed with a \emph{module action} and a
\emph{Gramian}. A Gramian $\gramian{\cdot}{\cdot}$ is similar to a
scalar product but is valued in the space $\cS_1(\cH_0)$ and is
related to scalar product by the relation
$\pscal{\cdot}{\cdot} = \tr(\gramian{\cdot}{\cdot})$. Notions such as
sub-modules, Gramian-orthogonality, Gramian-isometric operators are
natural extensions of their counterparts in the Hilbert framework. We
give such useful definitions hereafter and refer to \cite[Chapter~2]{yuichiro-multidim} for details.

\begin{definition}[$\cL_b(\cH_0)$-module]
  Let $\cH_0$ be a separable Hilbert space. An $\cL_b(\cH_0)$-module is a commutative group $(\cH, +)$ such that
  there exists a multiplicative operation (called the \emph{module action})
  $$
  \fundef{\cL_b(\cH_0) \times \cH & \to & \cH \\ (\aop,x) & \mapsto & \aop \bullet x}
  $$
  which satisfies the usual distributive properties : for all $\aop, \bop \in \cL_b(\cH_0)$, and $x,y \in \cH$,
  \begin{align*}
    \aop \bullet (x + y) &= \aop \bullet x + \aop \bullet y, \\
    (\aop + \bop) \bullet x &= \aop \bullet x + \bop \bullet x, \\
    (\aop \bop) \bullet x &= \aop \bullet (\bop \bullet x), \\
    \Id_{\cH_0} \bullet x &= x. 
  \end{align*}
\end{definition}
Next, we endow an $\cL_b(\cH_0)$-module with an $\cL_b(\cH_0)$-valued product.
\begin{definition}[(Normal) pre-Hilbert $\cL_b(\cH_0)$-module]\label{def:prehilbertmodule}
  Let $\cH_0$ be a separable Hilbert space. We say that $(\cH, \gramian{\cdot}{\cdot}_\cH)$ is a pre-Hilbert $\cL_b(\cH_0)$-module if  $\cH$ is an $\cL_b(\cH_0)$-module and $\gramian{\cdot}{\cdot}_\cH : \cH \times \cH \to \cL_b(\cH_0)$ satisfies, for all $x,y,z \in \cH$, and $\aop \in \cL_b(\cH_0)$,
  \begin{enumerate}[label=(\roman*)]
  \item $\gramian{x}{x}_\cH \in \cL_b^+(\cH_0)$,
  \item\label{itm:gramian-definite} $\gramian{x}{x}_\cH = 0$ if and only if $x = 0$,
  \item $\gramian{x + \aop \bullet y}{z}_\cH = \gramian{x}{z}_\cH + \aop \gramian{y}{z}_\cH$,
  \item $\gramian{y}{x}_\cH = \gramian{x}{y}_\cH^\adjoint$.
  \end{enumerate}
  If moreover, for all $x,y \in \cH$, $\gramian{x}{y}_\cH \in \cS_1(\cH_0)$, we say that $\gramian{\cdot}{\cdot}_\cH$ is a \emph{Gramian} and that $\cH$ is a \emph{normal} pre-Hilbert $\cL_b(\cH_0)$-module. 
\end{definition}
Note that an $\cL_b(\cH_0)$-module is a vector space if we define the
scalar-vector multiplication by
$\alpha x = (\alpha \Id_{\cH_0}) \bullet x$ for all
$\alpha \in \cset$, $x \in \cH$ and that, in the particular case where
$\gramian{\cdot}{\cdot}_\cH$ is a Gramian, then
$\pscal{\cdot}{\cdot}_\cH := \tr \gramian{\cdot}{\cdot}_\cH$ is a
scalar product. Hence a normal pre-Hilbert $\cL_b(\cH_0)$-module is
also a pre-Hilbert space. A normal pre-Hilbert $\cL_b(\cH_0)$-module
is said to be a \emph{normal Hilbert $\cL_b(\cH_0)$-module} if it is
complete (for the norm defined by
$\norm{x}_\cH^2 = \pscal{x}{x}_\cH = \norm{\gramian{x}{x}_\cH}_1$). A
subset of $\cH$ is called a \emph{submodule} if it is an
$\cL_b(\cH_0)$-module. An operator $\fop \in \cL_b(\cH, \cG)$, where
$\cH$ and $\cG$ are two $\cL_b(\cH_0)$-module, is said to be
\emph{$\cL_b(\cH_0)$-linear} if for all $\aop \in \cL_b(\cH_0)$ and
$x \in \cH$, $\fop (\aop \bullet x) = \aop \bullet (\fop x)$. An
$\cL_b(\cH_0)$-linear operator $U$ between two pre-Hilbert
$\cL_b(\cH_0)$-modules $\cH$ and $\cG$ is said to be
\emph{Gramian-isometric} if for all $x,y \in \cH$,
$\gramian{Ux}{Uy}_\cG = \gramian{x}{y}_\cH$ and \emph{Gramian-unitary}
if it is bijective Gramian-isometric. The space $\cH$ is said to be
\emph{Gramian-isometrically embedded} in $\cG$ (denoted by
$\cH \gramianisometricallyembedded \cG$) if there exists a
Gramian-isometric operator from $\cH$ to $\cG$. The spaces $\cH$ and
$\cG$ are said to be \emph{Gramian-isometrically isomorphic} (denoted
by $\cH \gramianisomorphic \cG$) if there exists a Gramian-unitary
operator from $\cH$ to $\cG$.  The well known isometric extension
theorem can be straightforwardly generalized to the case of
Gramian-isometric operators and is stated in the following proposition
for latter reference.
\begin{proposition}[Gramian-isometric extension]\label{prop:gramian-isometric-extension}
  Let $\cH_0$ be a separable Hilbert space, $\cH$ be a normal
  pre-Hilbert $\cL_b(\cH_0)$-module, and $\cG$ be a
  normal Hilbert $\cL_b(\cH_0)$-module. Let $(v_j)_{j \in J}$ and
  $(w_j)_{j \in J}$ be two collections of vectors in $\cH$ and $\cG$
  respectively with $J$ an arbitrary index set. If for all
  $i,j \in J$, $\gramian{v_i}{v_j}_{\cH} = \gramian{w_i}{w_j}_{\cG}$
  then there exists a unique Gramian-isometric operator
  $$
  S : \cspanarg[\cH]{\aop \bullet v_j, \aop \in \cL_b(\cH_0), j \in J} \to \cG
  $$
  such that for all $j \in J$, $S v_j = w_j$. If moreover $\cH$ is complete then
  $$
  S \left(\cspanarg[\cH]{\aop \bullet v_j, \aop \in \cL_b(\cH_0), j \in J}\right) = \cspanarg[\cG]{\aop \bullet w_j, \aop \in \cL_b(\cH_0), j \in J}
  $$
\end{proposition}

As stated in the introduction, the spectral theory for functional time series relies on the Gramian structure of the space of functional random variables with finite second order moment. The following example exhibits this structure.

\begin{example}[Normal Hilbert module $\cM(\Omega, \cF, \cH_0, \PP)$]\label{exple:functional-rv}
  Let $(\Omega, \cF, \PP)$ be a probability space and $\cH_0$ be a
  separable Hilbert space. The Bochner space
  $L^2(\Omega, \cF, \cH_0, \PP)$ is the space of $\cH_0$-valued random
  variables $Y$ such that $\PE[\norm{Y}_{\cH_0}^2] < +\infty$. Then the
  expectation of $Y$ is the unique
  vector $\PE[Y] \in \cH_0$ satisfying
$$
\pscal{\PE[Y]}{x}_{\cH_0} = \PE[\pscal{Y}{x}_{\cH_0}], \quad \text{for all } x \in \cH_0\;,
$$
and the covariance operator between $Y,Z \in L^2(\Omega, \cF, \cH_0, \PP)$ is the unique linear operator $\Cov{Y}{Z} \in \cL_b(\cH_0)$, satisfying
$$
\pscal{\Cov{Y}{Z} y}{x}_{\cH_0} = \Cov{\pscal{Y}{x}_{\cH_0}}{\pscal{Z}{y}_{\cH_0}}, \quad \text{for all } x,y \in \cH_0 \; .
$$
The space $\cM(\Omega, \cF, \cH_0, \PP)$ of all centered
  random variables in $L^2(\Omega, \cF, \cH_0, \PP)$ is a normal
  Hilbert $\cL_b(\cH_0)$-module for the module action defined for all
  $\aop\in\cL_b(\cH_0)$ and $X\in\cM(\Omega, \cF, \cH_0, \PP)$ by  $\aop \bullet X=\aop X$, and the Gramian
$$
\gramian{X}{Y}_{\cM(\Omega, \cF, \cH_0, \PP)}=\cov XY \; . 
$$
\end{example}

\section{Towards the stochastic integral}
\label{sec:more-preliminaries}

\subsection{Gramian-orthogonally scattered (\cagos) measures}\label{sec:cagos}
In this section, we introduce the notion of random \cagos\ measures
which will have an important role in the construction provided by
\cite{kallianpur1971spectral,
  mandrekar-square-integrability,yuichiro-multidim}. The terminologies
\caos\ and \cagos\ are borrowed from Definition~3 in \cite[Section~3.1]{yuichiro-multidim}

\begin{definition}[(Random) \caos\  measures]
  \label{def:caos}
  Let $\cH$ be a Hilbert space and $(\Lambda, \calA)$ be a measurable
  space. We say that $W : \calA \to \cH$ is a \emph{countably additive
    orthogonally scattered} (\caos) measure on $(\Lambda, \calA, \cH)$
  if it is an $\cH$-valued measure on $(\Lambda, \calA)$ such that for
  all $A,B \in \calA$,
  $$
  A \cap B = \emptyset \Rightarrow \pscal{W(A)}{W(B)}_{\cH} = 0\;.
  $$
  In this case, the mapping
  $$
  \nu_W : A \mapsto \pscal{W(A)}{W(A)}_\cH 
  $$
  is a finite non-negative measure on $(\Lambda, \calA)$
  called the \emph{intensity measure} of $W$ and we have that, for all
  $A,B \in \calA$,
  \begin{equation}
    \label{eq:caos-basic-isometric-identity}
  \nu_W(A \cap B) = \pscal{W(A)}{W(B)}_{\cH}\;.    
  \end{equation}
  We say that $W$ is regular if $\nu_W$ is regular. When 
  $\cH$ is the space $\cM(\Omega, \cF, \cH_0, \PP)$ of~\Cref{exple:functional-rv}, we
  say that $W$ is an $\cH_0$-valued \emph{random \caos\ measure} on
  $(\Lambda, \cA, \Omega, \cF, \PP)$.
\end{definition}
The generalization to a normal Hilbert module is straightforward. 
\begin{definition}[(Random) \cagos\ measures]
  \label{def:cagos}
  Let $\cH_0$ be a separable Hilbert space, $\cH$ be a normal Hilbert
  $\cL_b(\cH_0)$-module and $(\Lambda, \calA)$ be a measurable space. We
  say that $W : \calA \to \cH$ is a \emph{countably additive
    Gramian-orthogonally scattered} (\cagos) measure on
  $(\Lambda, \calA, \cH)$ if it is an $\cH$-valued measure on
  $(\Lambda, \calA)$ such that for all $A,B \in \calA$,
  $$
  A \cap B = \emptyset \Rightarrow \gramian{W(A)}{W(B)}_{\cH} = 0\;.
  $$
  In this case, the mapping
  $$
  \nu_W : A \mapsto \gramian{W(A)}{W(A)}_\cH 
  $$
  is a trace-class \povm\ on $(\Lambda, \calA, \cH_0)$
  called the \emph{intensity operator measure} of $W$ and we have that, for all
  $A,B \in \calA$,
   \begin{equation}
    \label{eq:cagos-basic-isometric-identity}
  \nu_W(A \cap B) = \gramian{W(A)}{W(B)}_{\cH}\;.
\end{equation}
We say that $W$ is regular if $\norm{\nu_W}_1$ is regular. When
  $\cH=\cM(\Omega, \cF, \cH_0, \PP)$ of
  \Cref{exple:functional-rv}, we say that $W$ is an $\cH_0$-valued
  \emph{random \cagos\ measure} on $(\Lambda, \cA, \Omega, \cF, \PP)$.
\end{definition}

It is easy to show that a \caos\ measure $W$ as in \Cref{def:caos} can be
equivalently seen as the restriction of an isometric operator $I$ from
$L^2(\Lambda, \calA, \nu_W)$ onto $\cH$ by setting
$$
W(A)=I(\1_A)\;,\qquad A\in\Lambda\;.
$$
This simply follows by interpreting the left-hand side
of~(\ref{eq:caos-basic-isometric-identity}) as the scalar product
between $\1_A$ and $\1_B$ in $L^2(\Lambda, \calA, \nu_W)$ so that $I$
above can be defined as the unique isometric extension from
$L^2(\Lambda, \calA, \nu_W)$ to $\cH$ of the isometric mapping defined
by $\1_A\mapsto W(A)$ for $A\in\Lambda$.  This observation gives a
rigorous meaning to the integral in the Cramér representation
\eqref{eq:homes-x-spectral} where $\hat{X}$ is \caos\ (see
\cite[Section~2]{holmes-signalprocessing}).  Similarly, if $W$ is a
\cagos\ measure as in \Cref{def:cagos} and $\cH_0 = \cset^q$, the
mapping defined by $\1_A\aop\mapsto \aop W(A)$ for $A\in\Lambda$ and
$\aop\in\cset^{q\times q}$ is Gramian-isometric from a normal
pre-Hilbert module of matrix-valued functions onto $\cH$ (see
\cite{masani_recent_1966}). This observation is a key step to derive a
Cramér representation of the type \eqref{eq:homes-x-spectral} where
$(X_t)_{t \in \zset}$ is a multivariate time series and $\hat{X}$ is
\cagos\ In the infinite dimensional case, the Gramian-isometric
property of the mapping defined by $\1_A\aop\mapsto \aop W(A)$ for
$A\in\Lambda$ and $\aop\in \cL_b(\cH_0)$ can also be established. This
is done in \cite[Section~2.5]{these-tavakoli-2015} where the author
uses the completion of $L^2(\Lambda,\calA,\cL_b(\cH_0),\norm{\nu}_1)$
under an appropriate norm. These ideas are in fact very similar to the
ones of \cite{kallianpur1971spectral,
  mandrekar-square-integrability,yuichiro-multidim} with the exception
that the latter references provide a more general framework and lead
to a modular spectral domain which is an explicit set of
operator-valued functions defined on $\Lambda$.  We follow this
approach in the next section.

\subsection{The space $\mathsf{L}^2(\Lambda, \calA, \cO(\cH_0,\cG_0), \nu)$}\label{sec:square-integral-povm-complete}

As discussed in the previous sections, the role of \caos\ and \cagos\
measures in the spectral theory of weakly stationary processes relies
on their characterization by unitary or Gramian-unitary operators
between the (modular) time domain and the (modular) spectral
domain. This has been entirely studied in the case of univariate and
multivariate time series, see \cite{holmes-signalprocessing} and
\cite{masani_recent_1966}, respectively, and the references
therein. For time series valued in a general separable Hilbert space,
defining the modular spectral domain requires to exhibit a suitable
space of operator-valued
functions which are \emph{square-integrable} with respect to the trace-class \povm\
$\nu$. It was introduced in \cite{mandrekar-square-integrability} and
includes the space
$L^2(\Lambda, \calA, \cL_b(\cH_0,\cG_0), \norm{\nu}_1)$ but is in
general larger in the case where $\cH_0$ has infinite
dimension. The definition relies on the following notion of measurability which we slightly adapted
from \cite{mandrekar-square-integrability}, \cite[Section~3.4]{yuichiro-multidim}.
\begin{definition}[$\cO$-measurability] Given two Hilbert spaces $\cH$ and $\cG$, a function $\Phi : \Lambda \to \cO(\cH,\cG)$ is said to be $\cO$-measurable if it satisfies the two following conditions.
\begin{enumerate}[label=(\roman*)]
\item\label{itm:O-meas-i} For all $x \in \cH$, $\set{\lambda \in \Lambda}{x \in \domain(\Phi(\lambda))} \in \calA$.
\item\label{itm:O-meas-ii} There exists a sequence
  $(\Phi_n)_{n \in \nset}$ valued in $\simplemeasfunctionsetarg{\Lambda, \calA, \cH,\cG}$
  such that for all $\lambda \in \Lambda$ and  $x \in \domain(\Phi(\lambda))$,
  $\Phi_n(\lambda) x$ converges to  $\Phi(\lambda) x$ in $\cG$ as $n\to\infty$.
\end{enumerate}
We denote by $\weakmeasfunctionsetarg{\Lambda, \calA, \cH,\cG}$ the space of such
functions $\Phi$.  
\end{definition}
Square-integrability with respect to a trace-class \povm\ $\nu$ is then defined as follows. 
\begin{definition}\label{def:square-integrability}
  Let $(\Lambda, \cA)$ be a measurable space, $\cH_0, \cG_0,\cI_0$ be
  three separable Hilbert spaces and $\nu$ a trace-class \povm\ on
  $(\Lambda, \cA, \cH_0)$ with density $f$ with respect to its finite
  variation $ \norm{\nu}_1$. Then, we say that 
  $(\Phi,\Psi)\in\weakmeasfunctionsetarg{\Lambda,\calA,\cH_0,
    \cG_0}\times\weakmeasfunctionsetarg{\Lambda,\calA,\cH_0, \cI_0}$ is \emph{$\nu$-integrable} if the three following
  assertions hold.
  \begin{enumerate}[label=(\roman*)]
  \item\label{itm:range-in-dom} We have $\range(f^{1/2}) \subset \domain(\Phi)$ and $\range(f^{1/2}) \subset \domain(\Psi)$, $\norm{\nu}_1$-a.e.
  \item\label{itm:compo-in-S2} We have $\Phi f^{1/2}\in\cS_2(\cH_0, \cG_0)$ and $\Psi
    f^{1/2}\in\cS_2(\cH_0,
    \cI_0)$,  $\norm{\nu}_1$-a.e.
  \item\label{itm:inL1traceclass} We have $(\Phi f^{1/2})(\Psi f^{1/2})^\adjoint \in \cL^1(\Lambda, \calA, \cS_1(\cI_0,\cG_0), \norm{\nu}_1)$. 
  \end{enumerate}
  In this case,  we define 
   \begin{equation}\label{eq:integral-L2-unbounded}
  \int \Phi\rmd \nu \Psi^\adjoint := \int (\Phi f^{1/2})(\Psi f^{1/2})^\adjoint \, \rmd \norm{\nu}_1 \in \cS_1(\cI_0,\cG_0) \; .
  \end{equation}
  Moreover, we say that $\Phi\in\weakmeasfunctionsetarg{\Lambda,\calA,\cH_0,
    \cG_0}$ is \emph{square
  $\nu$-integrable} if  $(\Phi, \Phi)$ is $\nu$-integrable and we denote by
  $\mathscr{L}^2(\Lambda, \calA, \cO(\cH_0,\cG_0), \nu)$ the space of square $\nu$-integrable functions in $\weakmeasfunctionsetarg{\Lambda, \calA, \cH_0,
    \cG_0}$.
\end{definition}
\begin{remark}
  \label{rem-tcpovm}
  Let us briefly comment this definition.
  \begin{enumerate}[label=\arabic*)]
  \item The integral~(\ref{eq:integral-L2-unbounded}) of
  \Cref{def:square-integrability} can be seen as an extension of the
  integral of scalar-valued functions with respect to a trace-class \povm\ since, for
  a measurable scalar function $\phi:\Lambda\to\cset$ we can interpret
  the integral $\int \phi \, \rmd\nu$ as the one in~(\ref{eq:integral-L2-unbounded}) with  $\Phi:\lambda\mapsto \phi(\lambda)\Id_{\cH_0}$ and
  $\Psi\equiv\Id_{\cH_0}$. 
\item It is easy to show that for all
$\Phi,\Psi\in\mathscr{L}^2(\Lambda, \calA, \cO(\cH_0,\cG_0), \nu)$,
$(\Phi, \Psi)$ is $\nu$-integrable and thus $\int\Phi \rmd \nu
\Psi^\adjoint$ is well defined as above. 

\item In the special case where $\Phi$ and $\Psi$ are valued in
$\cL_b(\cH_0, \cG_0)$, $\cO$-measurability reduces to
simple measurability, \ref{itm:range-in-dom} and \ref{itm:compo-in-S2}
are always verified, \ref{itm:inL1traceclass} is equivalent to
$\Phi f \Psi^\adjoint \in \cL^1(\Lambda, \calA, \cS_1(\cG_0),
\norm{\nu}_1)$, in which case we have
$$
\int \Phi \rmd \nu \Psi^\adjoint = \int \Phi f \Psi^\adjoint \, \rmd \norm{\nu}_1 \; .
$$
In particular, since $\norm{\Phi f
  \Phi^\adjoint}_1\leq\norm{\Phi}^2\,\norm{f}_1=\norm{\Phi}^2$,
  $\norm{\nu}_1-$a.e., we get that
$$
\cL^2(\Lambda,\calA,\cL_b(\cH_0, \cG_0),\norm{\nu}_1)\subset \mathscr{L}^2(\Lambda, \calA, \cO(\cH_0,\cG_0), \nu)\;.
$$
Moreover, the mapping defined, for all $\Phi, \Psi \in \cL^2(\Lambda,\calA,\cL_b(\cH_0, \cG_0),\norm{\nu}_1)$, by $\gramian{\Phi}{\Psi}_\nu := \int \Phi \rmd\nu \Psi^\adjoint$ is a pseudo Gramian in the sense that it satisfies all assumptions of \Cref{def:prehilbertmodule} except  assumption~\ref{itm:gramian-definite}. In particular, the space $\mathscr{L}^2(\Lambda, \calA, \cO(\cH_0,\cG_0), \nu)$ is larger than the ones used in \cite{panaretos13Cramer} and \cite[Appendix~B.2.3]{vandelft2018} for filtering functional time series. 
\end{enumerate}
\end{remark}
The following theorem, which corresponds to \cite[Theorem~4.19]{mandrekar-square-integrability} and Theorem~11 in \cite[Section~3.4]{yuichiro-multidim}, shows that the same Gramian can be
used over the larger space 
$\mathscr{L}^2(\Lambda, \calA, \cO(\cH_0, \cG_0), \nu)$ and that it
makes this space a
normal Hilbert $\cL_b(\cG_0)$-module when quotiented by the set with
zero norm.
\begin{theorem}\label{thm:pseudo-gramian-L2-unbounded}
  Let $\cH_0, \cG_0$ be separable Hilbert spaces, $(\Lambda, \calA)$ a measurable space, $\nu$ a trace-class \povm\ on $(\Lambda, \calA, \cH_0)$ and $f = \frac{\rmd \nu}{\rmd \norm{\nu}_1}$. Then $\mathscr{L}^2(\Lambda, \calA, \cO(\cH_0, \cG_0), \nu)$ is an $\cL_b(\cG_0)$-module with module action
  $$
  \aop \bullet \Phi : \lambda \mapsto \aop \Phi(\lambda), \quad \aop
  \in \cL_b(\cG_0), \Phi \in \mathscr{L}^2(\Lambda, \calA, \cO(\cH_0,
  \cG_0), \nu) \;.
  $$
  Moreover, we can endow $\mathscr{L}^2(\Lambda, \calA, \cO(\cH_0, \cG_0), \nu)$
  with the pseudo-Gramian 
  \begin{equation}\label{eq:gramian-L2-unbounded}
    \gramian{\Phi}{\Psi}_{\nu} := \int \Phi \rmd \nu \Psi^\adjoint  \quad \Phi, \Psi \in \mathscr{L}^2(\Lambda, \calA, \cO(\cH_0, \cG_0), \nu)\;.
  \end{equation}
  Then, for all $\Phi \in \mathscr{L}^2(\Lambda, \calA, \cO(\cH_0, \cG_0), \nu)$,  we have
  $$
  \norm{\Phi}_\nu = \norm{\gramian{\Phi}{\Phi}_\nu}_1^{1/2}= 0\Longleftrightarrow\Phi f^{1/2} = 0 \quad \norm{\nu}_1\text{-a.e.} 
  $$
Let us denote the class of such $\Phi$'s by  $\{\norm{\cdot}_\nu=0\}$ and
the quotient space by
  $$
  \mathsf{L}^2(\Lambda, \calA, \cO(\cH_0, \cG_0), \nu) :=
  \quotient{\mathscr{L}^2(\Lambda, \calA, \cO(\cH_0, \cG_0), \nu)}
  {\{\norm{\cdot}_\nu=0\}} \; .
  $$
  Then 
  $\left(\mathsf{L}^2(\Lambda, \calA, \cO(\cH_0, \cG_0), \nu),\gramian{\cdot}{\cdot}_\nu\right)$ is a normal Hilbert
  $\cL_b(\cG_0)$-module.
\end{theorem}
\subsection{Integration with respect to a random \cagos\ measure}\label{sec:cagos-integral}

We now define the mapping which provides a representation of the
normal Hilbert $\cL_b$-module module generated by a random \cagos\
measure in the form a module of square integrable operator functions.
It is often seen as a stochastic integral because it linearly and
continuously maps a function to a random variable.  Let $\cH_0$ and
$\cG_0$ be two separable Hilbert spaces, $(\Lambda,\calA)$ be a
measurable space, and let $\nu$ be a trace-class \povm\ defined on
$(\Lambda,\calA,\cH_0)$.  Given an $\cH_0$-valued random \cagos\
measure $W$, we further set
\begin{equation}
  \label{eq:W-domain-nu-def}
\cH^{W,\cG_0} := \cspanarg[\cG]{\aop W(A) \; : \;  \aop \in \cL_b(\cH_0, \cG_0), \, A \in \calA}\;,
\end{equation}
which is a submodule of $\cG:=\cM(\Omega, \cF, \cG_0, \PP)$.  
As in Proposition~13 in \cite[Secion~3.4]{yuichiro-multidim} and \cite[Theorem~6.9]{mandrekar-square-integrability}, we now define the integral of an
$\cH_0\to\cG_0$ operator-valued
function with respect to a random \cagos\ measure $W$ as a Gramian-isometry
from  the normal Hilbert $\cL_b(\cG_0)$-module
$\mathsf{L}^2(\Lambda, \calA, \cO(\cH_0, \cG_0), \nu_W)$ to
$\cH^{W,\cG_0}$. A detailed proof can be found in
\Cref{sec:proofs-preliminaries}.
\begin{theorem}\label{thm:integral-cagos}
  Let $(\Lambda, \calA)$ be a measurable space and $(\Omega, \cF, \PP)$ a
  probability space. Let $\cH_0$ and $\cG_0$ be two separable Hilbert
  spaces. Let $W$ be an $\cH_0$-valued random \cagos\ measure on
  $(\Lambda,\calA,\Omega,\cF,\PP)$ with intensity operator
  measure $\nu_W$. Let
  $\cH^{W,\cG_0}$ be defined as in~(\ref{eq:W-domain-nu-def}).
  Then there
  exists a unique Gramian-isometry
  $$
  I_W^{\cG_0} : \mathsf{L}^2(\Lambda, \calA, \cO(\cH_0, \cG_0), \nu_W) \to   \cM(\Omega, \cF, \cG_0, \PP)
  $$
  such that, for all $A \in \calA$ and $\aop \in \cL_b(\cH_0, \cG_0)$,
  $$
  I_W^{\cG_0}(\1_A \aop) = \aop W(A) \quad  \PP\text{-a.s.} 
  $$
  Moreover,  $ \mathsf{L}^2(\Lambda, \calA, \cO(\cH_0, \cG_0), \nu_W)$ and
  $\cH^{W,\cG_0}$ are Gramian-isometrically isomorphic.
\end{theorem}
We can now define the integral of an operator-valued function with
respect to $W$.
\begin{definition}[Integral with respect to  a random \cagos\ measure]\label{def:integral-cagos}
Under the assumptions of \Cref{thm:integral-cagos}, we use an
integral sign to denote $I_W^{\cG_0}(\Phi)$ for
$\Phi\in \mathsf{L}^2(\Lambda, \calA, \cO(\cH_0, \cG_0), \nu_W)$. Namely, we
write
\begin{equation}\label{eq:integral-cagos}
  \int \Phi \, \rmd W = \int \Phi(\lambda) \, W(\rmd \lambda) := I_W^{\cG_0}(\Phi)\;.
\end{equation}
\end{definition}
The following remark will be useful.
\begin{remark}
  \label{rem:scalar-int-cagos}
  In the setting of \Cref{def:integral-cagos}, take
  $\Phi=\phi\,\Id_{\cH_0}$ with $\phi:\Lambda\to\cset$. Then, we have 
  $\Phi\in \mathsf{L}^2(\Lambda, \calA, \cO(\cH_0), \nu_W)$ if and
  only if $\phi\in L^2(\Lambda, \calA,\norm{\nu_W}_1)$. We will omit
  $\Id_{\cH_0}$ in the notation of the integral, writing $\int
  \phi\;\rmd W$ for  $\int
  \phi\Id_{\cH_0}\;\rmd W$.
\end{remark}
\section{Modular spectral domain of a weakly stationary  process and applications}\label{sec:spectral-analysis}

\subsection{The Gramian-Cramér representation and general Bochner theorems}
\label{sec:gram-cram-bochn}
We now have all the tools to derive a spectral theory for
Hilbert-valued weakly stationary processes following
\cite[Section~4.2]{yuichiro-multidim}. Let $(\Omega, \cF, \PP)$ be a
probability space, $\cH_0$ be a separable Hilbert space and
$(\lcag,+)$ be a locally compact Abelian (\lca) group (for example
$\zset$ or $\rset$), whose null element is denoted by 0. This means
that $\lcag$ is an Abelian topological group which is locally compact,
Hausdorff for its topology. Recall that the dual group $\hat\lcag$
denotes the set of continuous characters of $\lcag$ i.e. the set of
continuous functions $\chi : \lcag \to \unitcircle$ satisfying
$s,t \in \lcag$, $\chi(s+t) = \chi(s) \chi(t)$, where $\unitcircle$
denotes the complex unit circle. In particular,
$\hat{\zset} = \tore := \rset/(2\pi\zset)$ and 
$\hat{\rset} = \rset$. Details about \lca\ groups can be found in
\cite{rudin1990fourier}. Throughout this section we are interested in
the spectral properties of a centered process valued in a separable
Hilbert space and assumed to be weakly stationary in the following
sense.
\begin{definition}[Hilbert-valued weakly stationary processes]\label{def:functional-weakstat}
   Let $(\Omega, \cF, \PP)$ be a probability space, $\cH_0$ be a separable Hilbert
   space and $(\lcag,+)$ be an \lca\ group. Then a process $X := (X_t)_{t \in \lcag}$ is
   said to be an $\cH_0$-valued weakly stationary process if
\begin{enumerate}[label=(\roman{enumi}),resume=def-functional-weakstat]
\item\label{item:def:functional-weakstat1} For all $t \in \lcag$, $X_t \in L^2(\Omega, \cF, \cH_0, \PP)$.
\item\label{item:def:functional-weakstat2} For all $t \in \lcag$, $\PE[X_t] = \PE[X_0]$. We say that $X$ is centered if $\PE[X_0] = 0$.
\item\label{item:def:functional-weakstat3}  For all $t,h \in \lcag$, $\cov{X_{t+h}}{X_{t}} = \cov{X_h}{X_0}$.
\item\label{item:cov-op-continuity} The \emph{autocovariance operator function}
  $\Gamma_X : h \mapsto \cov{X_h}{X_0}$ satisfies the following
  continuity condition: for all
  $\aop \in \cL_b(\cH_0)$, $h \mapsto \tr(\aop \Gamma_X(h))$ is
  continuous on $\lcag$.
\end{enumerate}
\end{definition}
In the case of time series, $\lcag=\zset$, all mappings on $\lcag$ are
continuous and Condition~\ref{item:cov-op-continuity} can be discarded
in this definition.  It is less trivial to show that, for any \lca\
group $\lcag$, we get an equivalent definition if we
replace~\ref{item:cov-op-continuity} by just saying that $\Gamma_X$ is
continuous in \wot\ This interesting fact is explained in the
following remark in a more detailed fashion.
\begin{remark}
  \label{rem:cont-cond-def-weak-stat}
For any
  $x,y\in\cH_0$, taking $\aop=xy^\adjoint$ we have
  $\tr(\aop \Gamma_X(h))=\pscal{\Gamma_X(h)x}{y}_{\cH_0}$. Hence
  Condition~\ref{item:cov-op-continuity} of
  \Cref{def:functional-weakstat} implies the following one.
\begin{enumerate}[label=(\roman{enumi}'),resume=def-functional-weakstat]\addtocounter{enumi}{-1}
\item\label{item:cov-op-continuity-prim}   The autocovariance operator function
  $\Gamma_X : h \mapsto \cov{X_h}{X_0}$ is continuous  in \wot
\end{enumerate}
It is easy to find $\lcag$, $\cH_0$ and a mapping $f:\lcag\to\cS_1(\cH_0)$ which is
continuous in \wot\ but such that $h \mapsto \tr(f(h))$ is not
continuous hence does not satisfy the continuity condition imposed on
$\Gamma_X$ in~\ref{item:cov-op-continuity}.
However, it turns out that if $\Gamma_X$ is the
autocovariance operator function $h\mapsto \cov{X_h}{X_0}$ with $X$
satisfying Conditions~\ref{item:def:functional-weakstat1}
and~\ref{item:def:functional-weakstat3}, then
Conditions~\ref{item:cov-op-continuity}
and~\ref{item:cov-op-continuity-prim} become equivalent.  The reason
behind this surprising fact will be made clear later in Point
\ref{item:rem:wot-vs-wc1} of \Cref{rem:wot-vs-wc1}. In other words, we
can replace~\ref{item:cov-op-continuity}
by~\ref{item:cov-op-continuity-prim} without altering
\Cref{def:functional-weakstat}.
\end{remark}
As in the univariate case, the notion of weak stationarity is related
to an isometric property of the lag operators, but here the
covariance stationarity expressed in
Condition~\ref{item:def:functional-weakstat3} translates into a
Gramian-isometric property rather than a scalar isometric property.
Namely, let $X:=(X_t)_{t \in \lcag}$ satisfy
Conditions~\ref{item:def:functional-weakstat1}
and~\ref{item:def:functional-weakstat2} and take it centered so that
each $X_t$ belongs to the normal Hilbert module
$\cM(\Omega, \cF, \cH_0, \PP)$ as defined in
\Cref{exple:functional-rv}.  For all $h \in \lcag$, define the lag
operator of lag $h\in\lcag$ as the mapping
$U_h^X :X_t\mapsto X_{t + h}$ defined for all $t \in \lcag$. Then
Condition~\ref{item:def:functional-weakstat3} is equivalent to saying
that for all $h\in\lcag$, the mapping $U_h^X$ is Gramian-isometric on
$\set{X_t}{t\in\lcag}$ for
the Gramian structure inherited from $\cM(\Omega, \cF, \cH_0, \PP)$.
Thus, if this condition holds,  by
\Cref{prop:gramian-isometric-extension}, for any lag $h\in\lcag$, there exists a unique
Gramian-unitary operator extending $U_h^X$ on the \emph{modular time domain} $\cH^X$ of
$X$ defined as
the submodule of $\cH$ generated by the $X_t$'s, that is,
$$
\cH^X := \cspanarg[\cH]{\aop X_t \; : \; \aop \in \cL_b(\cH_0),  \, t \in \lcag} \; ,
$$
which is the generalization of (\ref{eq:ts-cHX}) to a general \lca\
group $\lcag$. In fact it is convenient to introduce a slightly more
general definition of the modular time domain where the output Hilbert
space $\cG_0$ may be taken different form $\cH_0$.
\begin{definition}[$\cG_0$-valued modular time domain]
Let $(\lcag,+)$ be an \lca\ group, and  $\cH_0$ and $\cG_0$ be two
separable Hilbert spaces. Let 
   $X:=(X_t)_{t \in \lcag}$ be a collection of variables in
$\cM(\Omega, \cF, \cH_0, \PP)$ as defined in
\Cref{exple:functional-rv}. 
The $\cG_0$-valued \emph{modular time domain} of $X$ is defined by
  \begin{equation}
    \label{eq:time-domain-def}
\cH^{X,\cG_0} := \cspanarg[\cM(\Omega, \cF, \cG_0, \PP)]{\aop X_t \; : \;  \aop \in \cL_b(\cH_0, \cG_0), \, t \in \lcag}\;, 
\end{equation}
which is a submodule of $\cM(\Omega, \cF, \cG_0, \PP)$.
\end{definition}
We now extend the (scalar) Cramér representation theorem by means of an
integral with respect to a \cagos\ measure.
\begin{theorem}[Gramian-Cramér representation theorem]\label{thm:spectral-analysis-hilbert}
  Let $\cH_0$ be a separable Hilbert space, $(\Omega, \cF, \PP)$ be a
  probability space and $(\lcag,+)$ be an \lca\, group.  Let
  $X:=(X_t)_{t \in \lcag}$ be a centered weakly stationary
  $\cH_0$-valued process as in \Cref{def:functional-weakstat}. Then 
  there exists a unique regular $\cH_0$-valued random \cagos\ measure
  $\hat{X}$ on $(\hat{\lcag}, \borel(\hat{\lcag}), \Omega,\cF,\PP)$
  such that
  \begin{equation}\label{eq:spectralrep-functional}
    X_t = \int \chi(t) \; \hat{X}(\rmd \chi)\quad\text{for all}\quad t \in \lcag\;.
  \end{equation}
\end{theorem}
This result is partly stated in Theorem~2 in
\cite[Section~4.2]{yuichiro-multidim}. Here we add the uniqueness of
$\hat{X}$, which
appears to be a new result in this general setting. We provide a 
detailed proof in \Cref{sec:proofs-spectral-rep}. In fact Theorem~2
 in \cite[Section~4.2]{yuichiro-multidim} contains a converse
 statement, which we now state separately as a lemma whose detailed proof can be found in \Cref{sec:proofs-spectral-rep}.
\begin{lemma}
\label{lem:cramer-easy}  Let $(\lcag,+)$ be an \lca\ group, $\cH_0$ a separable Hilbert space
  and $W$ be an $\cH_0$-valued random \cagos\ measure on
  $(\hat{\lcag},\borel(\hat{\lcag}),\Omega,\cF,\PP)$ with intensity
  operator measure $\nu$. Define, for all $t\in\lcag$,
  $$
  X_t=\int \chi(t)\;W(\rmd\chi) \;.
  $$
  Then $X=(X_t)_{t\in\lcag}$ is a centered $\cH_0$-valued weakly
  stationary process with autocovariance operator function $\Gamma$
  defined by
    \begin{equation}\label{eq:traceclass-bochner-general}
    \Gamma(h) = \int \chi(h) \, \nu(\rmd \chi)\quad\text{for all  $h \in \lcag$.}
  \end{equation}
\end{lemma}
With \Cref{thm:spectral-analysis-hilbert} at our disposal, we can now
define the \emph{Gramian-Cramér  representation} and the \emph{spectral
  operator measure} of $X$. 
\begin{definition}[Gramian-Cramér  representation and spectral operator measure]
  \label{def:spectral-analysis-hilbert}
  Under the setting of \Cref{thm:spectral-analysis-hilbert}, the regular
  \cagos\ measure $\hat{X}$ is called the \emph{(Gramian) Cramér
    representation} of $X$ and its intensity operator measure is
  called the \emph{spectral operator measure} of $X$. It is a regular
  trace-class \povm\ on $(\hat{\lcag}, \borel(\hat{\lcag}), \cH_0)$.
\end{definition}
By \Cref{lem:cramer-easy}, we see that the autocovariance operator function and the
spectral operator measure of $X$ are related to each other through
the identity~(\ref{eq:traceclass-bochner-general}). As already hinted in the introduction,
using the tools introduced in \Cref{sec:cagos-integral}, we can more
generally interpret the Cramér representation of
\Cref{thm:spectral-analysis-hilbert} as establishing a
Gramian-isometric mapping onto the modular time domain of $X$,
starting from its \emph{modular spectral domain} which we now
introduce.
\begin{definition}[$\cG_0$-valued spectral time domain]
  \label{def:modular-scptral-domain}
  Let $\cH_0$ and $\cG_0$ be two
separable Hilbert spaces and 
   $X:=(X_t)_{t \in \lcag}$ be a centered weakly stationary process
   valued in $\cH_0$ as in \Cref{def:functional-weakstat}. 
The
  $\cG_0$-valued \emph{modular spectral domain} of $X$  is the normal
  Hilbert $\cL_b(\cG_0)$-module defined  by
  \begin{equation}
    \label{eq:spectral-domain-def}
  \widehat{\cH}^{X,\cG_0} :=   \mathsf{L}^2(\hat{\lcag}, \borel(\hat{\lcag}), \cO(\cH_0, \cG_0), \nu_X)\;,      
\end{equation}
where $\nu_X$ is the spectral operator measure of $X$ introduced in \Cref{def:spectral-analysis-hilbert}.
\end{definition}
We can now state that the modular time and spectral domain  are
Gramian-isometrically isomorphic, whose proof can be found in \Cref{sec:proofs-spectral-rep}.
\begin{theorem}[Kolmogorov isomorphism theorem]
\label{thm:kolmo-isomorphism-thm}Under the setting of \Cref{thm:spectral-analysis-hilbert}, for any separable Hilbert space $\cG_0$,  the mapping
$I_{\hat{X}}^{\cG_0} : \Phi\mapsto\int \Phi\;\rmd\hat{X}$ is a Gramian-unitary operator from
$\widehat{\cH}^{X,\cG_0}$ to $\cH^{X,\cG_0}$
 and we have
$\cH^{X,\cG_0} =\cH^{\hat{X},\cG_0}$. Thus, the
$\cG_0$-valued \emph{modular time domain} $\cH^{X,\cG_0}$ and the
$\cG_0$-valued \emph{modular spectral domain} $\widehat{\cH}^{X,\cG_0}$ 
are \emph{Gramian-isometrically isomorphic}. 
\end{theorem}
Relation~(\ref{eq:traceclass-bochner-general}) is at the core of the
general Bochner theorem, which we now discuss. Recall that the \emph{standard}
(univariate) Bochner theorem can be stated as follows (see
\cite[Theorem~1.4.3]{rudin1990fourier} for existence and
\cite[Theorem~1.3.6]{rudin1990fourier} for uniqueness).
\begin{theorem}[Bochner Theorem]
  \label{them:bochner-stadard} Let $(\lcag,+)$ be an \lca\ group and $\gamma : \lcag \to \cset$. Then the two following statements are equivalent:
\begin{enumerate}[label=(\roman*)]
  \item\label{item:bochner-univariate1} $\gamma$ is continuous and
    hermitian non-negative definite, that is, for all $n \in \nset$, $t_1, \cdots, t_n \in \lcag$ and $a_1, \cdots, a_n \in \cset$,
  $$
  \sum_{i, j = 1}^n  a_i \overline{a_j} \gamma(t_i-t_j) \geq 0.
  $$
 \item\label{item:bochner-univariate2} There exists a regular finite non-negative measure $\nu$ on $(\hat{\lcag}, \borel(\hat{\lcag}))$ such that
  \begin{equation}\label{eq:bochner}
  \gamma(h) = \int \chi(h) \, \nu(\rmd \chi), \quad h \in \lcag.
\end{equation}
\end{enumerate}
Moreover, if Assertion~\ref{item:bochner-univariate2} holds, $\nu$ is
the unique  regular non-negative measure satisfying \eqref{eq:bochner}. 
\end{theorem}
There are various other ways to extend
Condition~\ref{item:bochner-univariate1} of \Cref{them:bochner-stadard}
when replacing $\cset$ by a Hilbert space $\cH_0$.
\begin{definition}\label{def:positive-definite-operator-function}
  Let $\cH_0$ be a Hilbert space and $(\lcag,+)$ an  \lca\ group.  A function $\Gamma : \lcag \to \cL_b(\cH_0)$ is said to be
  \begin{enumerate}
  \item\label{item:proper-auto-cov} a \emph{proper autocovariance}
    operator function if $\cH_0$ is separable and there exists a
    $\cH_0$-valued weakly
    stationary process with autocovariance
    operator function $\Gamma$; 
    \item\label{item:positive-definite-operator-function-pd} \emph{positive definite} if for all $n \in \nset^*$, $t_1, \cdots, t_n \in \lcag$ and $\aop_1, \cdots, \aop_n \in \cL_b(\cH_0)$,
      \begin{equation*}
      \sum_{i, j = 1}^n \aop_i \Gamma(t_i-t_j)\aop_j^\adjoint \succeq 0 \; ;
    \end{equation*}
  \item\label{item:positive-definite-operator-function-pt} of \emph{positive-type} if for all $n \in \nset^*$, $t_1, \cdots, t_n \in \lcag$ and $x_1, \cdots, x_n \in \cH_0$,
      \begin{equation*}
      \sum_{i, j = 1}^n \pscal{\Gamma(t_i-t_j) x_j}{x_i}_{\cH_0} \geq 0 \; ;
    \end{equation*}
  \item\label{item:positive-definite-operator-function-hnnd} \emph{hermitian non-negative definite} if for all $n \in \nset^*$, $t_1, \cdots, t_n \in \lcag$ and $a_1, \cdots, a_n \in \cset$,
  \begin{equation*}
  \sum_{i, j = 1}^n a_i \overline{a_j} \Gamma(t_i-t_j) \succeq 0.
\end{equation*}
  Equivalently, $\Gamma$ is hermitian non-negative definite if and only if for all $x \in \cH_0$, $t \mapsto \pscal{\Gamma(t) x}{x}_{\cH_0}$ is hermitian non-negative definite. 
\end{enumerate}
\end{definition}
It is straightforward to show that the definitions in \Cref{def:positive-definite-operator-function} are given in an increasing order of generality in the sense that \ref{item:proper-auto-cov} $\Rightarrow$ \ref{item:positive-definite-operator-function-pd} $\Rightarrow$ \ref{item:positive-definite-operator-function-pt} $\Rightarrow$ \ref{item:positive-definite-operator-function-hnnd}.
In the univariate case, for a continuous $\gamma:\lcag\to\cset$ all
these definitions are trivially equivalent to
Assertion~\ref{item:bochner-univariate1} in
\Cref{them:bochner-stadard}. A natural question for a general Hilbert
space $\cH_0$ is which definition should be used to extend the Bochner
theorem. A first answer is the following corollary whose proof can be found in \Cref{sec:proofs-spectral-rep}
\begin{corollary}
\label{cor:general-bochn}  Let $(\lcag,+)$ be an \lca\ group, $\cH_0$ a separable Hilbert space and
  $\Gamma : \lcag \to \cL_b(\cH_0)$. Then the following
  assertions are equivalent.
  \begin{enumerate}[label=(\roman{enumi})]
  \item\label{item:cor:bochner-gen-proper} The function $\Gamma$ is a proper autocovariance operator function.
  \item\label{item:cor:bochner-gen-tcpovm} There exists a regular
    trace-class \povm\ $\nu$ on
    $(\hat{\lcag},\borel(\hat{\lcag}),\cH_0)$ such
    that~(\ref{eq:traceclass-bochner-general}) holds.
\end{enumerate}
\end{corollary}
This result extends Bochner's theorem from the point of view of
$\cH_0$-valued weakly stationary processes so that $\Gamma$ in
\Cref{cor:general-bochn}\ref{item:cor:bochner-gen-proper} is valued in
$\cS_1(\cH_0)$ and for all $\aop \in \cL_b(\cH_0)$,
$h \mapsto \tr(\aop \Gamma(h))$ is continuous.  It turns our that
other extensions can be obtained using a purely operator theory point
of view with the more general positiveness conditions of
\Cref{def:positive-definite-operator-function}. In the following theorem,
$\cH$ is not necessarily separable,  $\Gamma$ is not necessarily $\cS_1(\cH)$-valued (and therefore the resulting \povm\ may not be trace-class) and its continuity condition can be relaxed to continuity for the \wot\ This result is essentially the Naimark's
moment theorem of \cite{berberian1966naui}. We refer to it as the
\emph{general} Bochner theorem (or general Herglotz theorem for
$\lcag=\zset$). 
\begin{theorem}[General Bochner Theorem]\label{thm:bochnerop}
  Let $(\lcag,+)$ be an \lca\ group, $\cH$ a Hilbert space and
  $\Gamma : \lcag \to \cL_b(\cH)$.
  Then the following
  assertions are equivalent.
  \begin{enumerate}[label=(\roman{enumi})]
  \item\label{itm:pos-def} $\Gamma$ is  continuous  in \wot\ and positive definite.
  \item\label{itm:pt} $\Gamma$ is continuous  in \wot\ and of positive type.
  \item\label{itm:hnnd} $\Gamma$ is continuous  in \wot\ and hermitian non-negative definite.
  \item\label{itm:fourier} There exists a regular \povm\ $\nu$ on
    $(\hat{\lcag}, \borel(\hat{\lcag}), \cH)$ such
    that~(\ref{eq:traceclass-bochner-general}) holds.
  \end{enumerate}
  Moreover, if Assertion~\ref{itm:fourier} holds, $\nu$ is the unique
  regular \povm\ satisfying~(\ref{eq:traceclass-bochner-general}).
\end{theorem}
It is important to note that there is a subtle difference
between Assertion~\ref{item:cor:bochner-gen-tcpovm}
of \Cref{cor:general-bochn} and Assertion~\ref{itm:fourier}
of \Cref{thm:bochnerop}, namely, the latter assertion is weaker since $\nu$
is not supposed to be trace-class. In particular, we cannot rely on the
Radon-Nikodym derivative as $\nu$ is not trace-class. The proof of \Cref{thm:bochnerop} is discussed in \Cref{sec:gram-cram-bochn}. An immediate consequence of \Cref{cor:general-bochn} and \Cref{thm:bochnerop} is the following
result whose proof can be found in \Cref{sec:gram-cram-bochn}. 
\begin{corollary}
\label{cor:general-bochn-add}  Let $(\lcag,+)$ be an \lca\ group, $\cH_0$ a separable Hilbert space and
  $\Gamma : \lcag \to \cL_b(\cH_0)$. Then the following
  assertions are equivalent.
  \begin{enumerate}[label=(\roman{enumi})]
  \item\label{item:cor:general-bochn-add1} The function $\Gamma$ is a proper autocovariance operator function.
  \item\label{item:cor:general-bochn-add2} Any of
    the Assertions~\ref{itm:pos-def}--\ref{itm:hnnd} in
    \Cref{thm:bochnerop} holds and $\Gamma(0)\in\cS_1(\cH_0)$. 
  \end{enumerate}
\end{corollary}

\begin{remark}\label{rem:wot-vs-wc1}
  Let us briefly comment on the equivalence established in
  \Cref{cor:general-bochn-add}. 
  \begin{enumerate}[label=\arabic*)]
  \item \label{item:rem:wot-vs-wc1}In Condition~\ref{item:cov-op-continuity} of
    \Cref{def:functional-weakstat}, we required a 
    condition on
    $\Gamma$ which is stronger than continuity in \wot\ However in Assertion~\ref{item:cor:general-bochn-add2} of
    \Cref{cor:general-bochn-add}, the continuity of $\Gamma$ is only
    needed in the \wot\ This means that we can replace the 
    continuity Condition~\ref{item:cov-op-continuity} in \Cref{def:functional-weakstat} by continuity
    in \wot\ as in~\Cref{rem:cont-cond-def-weak-stat}~\ref{item:cov-op-continuity-prim} without changing the overall definition of a weakly
    stationary process.
  \item The previous remark is related to a fact established
    in   Proposition~3 in  \cite[Section~4.2]{yuichiro-multidim}, which states
    the equivalence  between  being \emph{scalar stationary} and
    being \emph{operator stationary}. The latter definition is the
    same as our \Cref{def:functional-weakstat}, and the former one
    amounts to replace Condition~\ref{item:cov-op-continuity} in 
    \Cref{def:functional-weakstat} by assuming that for all
    $x\in\cH_0$, $x^\adjoint \Gamma x:h\mapsto x^\adjoint \Gamma(h) x$
    is continuous and hermitian non-negative definite. But this
    amounts to says that $\Gamma$ itself is continuous in the \wot\
    and hermitian non-negative definite. Since
    $\Gamma(0)\in\cS_1(\cH_0)$ is a consequence of
    Assertion~\ref{item:def:functional-weakstat1} in
    \Cref{def:functional-weakstat}, \Cref{cor:general-bochn-add}
    indeed implies the equivalence established by  Proposition~3
    in  \cite[Section~4.2]{yuichiro-multidim}. 
  \end{enumerate}
\end{remark}

\subsection{Composition and inversion of filters}
\label{sec:pointw-comp-oper}

With the construction of the spectral theory for weakly stationary
processes of \Cref{sec:gram-cram-bochn}, the study of linear filters
for such processes is easily derived. Indeed, we are now able to give the most general definition of
linear filtering, characterize the spectral structure of the filtered
process and provide results on compositions and inversion of linear
filters. Then, in the next section, we will provide a general statement of harmonic principal
component analysis for weakly stationary processes valued in a
separable Hilbert space.

Let $\cH_0$ and $\cG_0$ be two separable Hilbert spaces and  $\Phi\in \weakmeasfunctionsetarg{\Lambda, \calA, \cH_0, \cG_0}$. Let $X=(X_t)_{t\in\lcag}$ be an $\cH_0$-valued weakly stationary stochastic process such that 
\begin{equation}
  \label{eq:filtering-condition-H0-G0}
  \Phi\in\widehat{\cH}^{X,\cG_0}\;,
\end{equation}
where $\widehat{\cH}^{X,\cG_0}$ denotes the
modular spectral domain of \Cref{def:modular-scptral-domain}. Then we can define a $\cG_0$-valued weakly stationary stochastic process $Y=(Y_t)_{t\in\lcag}$ by 
\begin{equation}\label{eq:filtering-Y}
Y_t = \int \,\chi(t)\,\Phi(\chi)\;\hat{X}(\rmd\chi)\;,\quad \text{for all } t\in\lcag\; . 
\end{equation}
We say that $\Phi$ is the \emph{transfer operator function} of the
filter.  For convenience we write, in the time domain,
\begin{equation}
  \label{eq:time-domains-notation-transfer}
X\in\processtransfer{\Phi}(\Omega,\cF,\PP)\quad\text{and}\quad
Y=\filterprocess{\Phi}(X)\;,
\end{equation}
for \eqref{eq:filtering-condition-H0-G0} and \eqref{eq:filtering-Y}
respectively.  Since $\widehat{\cH}^{X,\cG_0}$ is
not easy to described, it may appear difficult to check the condition
$X\in\processtransfer{\Phi}(\Omega,\cF,\PP)$ for a given
$\Phi$. However a special case of interest is quite easy to
characterize, namely, when
$\Phi\in\simplemeasfunctionsetarg{\hat{\lcag}, \borel(\hat{\lcag}),
  \cH_0, \cG_0}$, as shown by the following result.
\begin{proposition}
  \label{prop:carac-processtransfer}
Let $(\lcag,+)$ be an \lca\ group, and  $\cH_0$ and $\cG_0$ be two
separable Hilbert spaces. Let $(\Omega,\cF,\PP)$ be a
  probability space and
  $\Phi\in\simplemeasfunctionsetarg{\hat{\lcag}, \borel(\hat{\lcag}),
  \cH_0, \cG_0}$. Let  $X=(X_t)_{t\in\lcag}$ be a $\cH_0$-valued centered weakly stationary
  processes admitting $g_X$ as  a spectral operator density 
  with respect to a $\sigma$-finite non-negative measure $\mu$ on
  $(\hat{\lcag}, \borel(\hat{\lcag}))$. Then the mapping $\norm{\Phi g_X \Phi^\adjoint}_1$ is
  measurable from $(\hat{\lcag}, \borel(\hat{\lcag}))$ to
  $(\rset,\borel(\rset))$ and we have
  $X\in\processtransfer{\Phi}(\Omega,\cF,\PP)$ if and only if
  $$
  \int \norm{\Phi g_X \Phi^\adjoint}_1\;\rmd\mu<\infty\;.
  $$
\end{proposition}
Many examples in the time series literature rely on a \emph{time domain}
description of the filtering obtained as follows.
\begin{example}[Convolutional filtering in the discrete case]
  \label{ex:conv-discrete}
Let $\cH_0$ and $\cG_0$ be two separable Hilbert spaces.
 Let $X=(X_t)_{t\in\zset}$ be an $\cH_0$-valued weakly stationary stochastic process
 defined on $(\Omega,\cF,\PP)$. 
  Let $\Phi=(\Phi_k)_{k\in\zset}$ be a sequence valued in
  $\cL_b(\cH_0,\cG_0)$ such that $\sum_k\norm{\Phi_k}<\infty$. Define the process $Y=(Y_t)_{t\in\zset}$ by the \emph{time
    domain convolutional filtering}
  $$
  Y_t=\sum_{k\in\zset} \Phi_k\,X_{t-k}\;,\quad t\in\zset\;,
  $$
  which converges absolutely in $\cM(\Omega, \cF, \cG_0, \PP)$.
  Then, defining $\hat\Phi:\tore\to\cL_b(\cH_0,\cG_0)$ by
  $$
  \hat\Phi(\lambda)=\sum_{k\in\zset} \Phi_k\,\rme^{\rmi\lambda k}\;,
  $$
  which absolutely converges in $\cL_b(\cH_0,\cG_0)$, uniformly in
  $\lambda\in\tore$. Thus,
  $\hat\Phi\in L^{\infty}(\tore,\btore,\cL_b(\cH_0,\cG_0))$, which
  implies
  $\hat\Phi g_X
  \hat\Phi^\adjoint\in L^{1}(\tore,\btore,\cS_1(\cG_0))$. Thus, by
  \Cref{prop:carac-processtransfer}, we have
  $X\in\processtransfer{\Phi}(\Omega,\cF,\PP)$ and
  $Y=\filterprocess{\hat\Phi}(X)$.
\end{example}
This example can be easily extended to processes indexed by any \lca\ group
$\lcag$ by using the Fubini-type theorem
\Cref{prop:fubini-bochner-cagos}, see \Cref{ex:conv-allG}.

The following result deals with the composition and inversion of
general filters. Its proof can be found
in \Cref{sec:proofs-filtering}. See also  \Cref{sec:hilbert-modules}, where we introduced the symbols
$\gramianisomorphic,\gramianisometricallyembedded$.
\begin{proposition}[Composition and inversion of filters on weakly
  stationary time series]\label{prop:compo-inversion-filters-ts}
  Let $\cH_0$ and $\cG_0$ be two separable Hilbert spaces and pick a
  transfer operator function
  $\Phi \in \weakmeasfunctionsetarg{\hat{\lcag}, \borel(\hat{\lcag}),
    \cH_0, \cG_0}$. Let $X$ be a centered weakly stationary $\cH_0$-valued
  process defined on $(\Omega,\cF,\PP)$ with spectral operator measure
  $\nu_X$. Suppose that
  $X\in\processtransfer{\Phi}(\Omega,\cF,\PP)$ and set
  $Y=\filterprocess{\Phi}(X)$, as defined in~(\ref{eq:time-domains-notation-transfer}).  Then the three
  following assertions hold.
\begin{enumerate}[label=(\roman*)]
\item For any separable Hilbert space  $\cI_0$, we have $\cH^{Y,\cI_0} \gramianisometricallyembedded
  \cH^{X,\cI_0}$.
\item  For any
    separable Hilbert space  $\cI_0$ and all $\Psi \in
    \weakmeasfunctionsetarg{\hat{\lcag}, \borel(\hat{\lcag}), \cG_0, \cI_0}$, we have
    $X\in \processtransfer{\Psi \Phi}(\Omega,\cF,\PP)$ if and only if
    $\filterprocess{\Phi}(X) \in \processtransfer{\Psi}(\Omega,\cF,\PP)$, and in this
    case, we have
  \begin{equation}\label{eq:compo-filt-process}
  \filterprocess{\Psi} \circ \filterprocess{\Phi}(X) = \filterprocess{\Psi \Phi}(X).
  \end{equation}
\item Suppose that $\Phi$ is injective $\norm{\nu_X}_1$-a.e. Then
  $X = \filterprocess{\Phi^{-1}}\circ\filterprocess{\Phi}(X)$, where
  we define
  $\Phi^{-1}(\lambda):= \left(\Phi(\lambda)_{|\domain(\Phi(\lambda))
      \to \range(\Phi(\lambda))}\right)^{-1}$ with domain
  $\range(\Phi(\lambda))$ for all
  $\lambda\in\{\Phi\text{ is injective}\}$ and $\Phi^{-1}(\lambda)=0$
  otherwise. Moreover,
  Assertion~\ref{item:cor:inclusion-filtered-space} above holds with
  $\gramianisometricallyembedded$ replaced by $\gramianisomorphic$.
\end{enumerate}
\end{proposition}

\subsection{Cramér-Karhunen-Loève decomposition}
\label{sec:cram-karh-loeve}
Let $\cH_0$ be a separable Hilbert space with (possibly infinite)
dimension $N$ and $X = (X_t)_{t \in \lcag}$ be a centered,
$\cH_0$-valued weakly-stationary process defined on a probability
space $(\Omega, \cF, \PP)$ with Cramér representation $\hat{X}$ and
spectral operator measure $\nu_X$. 

The Cramér-Karhunen-Loève decomposition amounts to give a rigorous
meaning to the formula
\begin{equation}
  \label{eq:CKL-first}
      \hat{X}(\rmd \chi) = \sum_{0\leq n< N} \phi_n(\chi) \otimes \phi_n(\chi)\, \hat{X}(\rmd \chi) \; ,
\end{equation}
where, for all $\chi\in\hat\lcag$, $(\phi_n(\chi))_{0\leq n<N}$ is
an orthonormal sequence in $\cH_0$ chosen in such a way that the summands
in~(\ref{eq:CKL-first}) are uncorrelated and where, for all $x\in \cH$ and $y\in\cG$, we denote by $x \otimes y$ the trace-class operator from $\cG$ onto $\cH$ defined by
$(x\otimes y) z = \pscal{z}{y}_{\cG} x$ for all $z \in \cG$. Such a decomposition provides a
way to derive the harmonic principal component analysis of the process
$X$, which is an approximation of $X$ by a finite rank linear
filtering. In recent works, the functional Cramér-Karhunen-Loève
decomposition is achieved under additional assumptions on $\nu_X$ such
as having a continuous density with respect to the
Lebesgue measure (in \cite{these-tavakoli-2015}) or at most
finitely many atoms (in \cite{van_delft_note_herglots_2020}). In fact,
thanks to the Radon-Nikodym property of trace-class \povm's,
there is no need for such additional
assumptions. Instead, we rely on the following lemma, whose proof can
be found in \Cref{sec:proofs-crefs-karh}.
\begin{lemma}[Eigendecomposition of a trace-class \povm]
  \label{lem:eigendecomp-tc-povm} Let $\cH_0$ be a separable Hilbert
  space with dimension $N\in\{1,\dots,+\infty\}$.  Let $\nu$ be a
  trace-class \povm\ on $(\Lambda,\cA,\cH_0)$ and $\mu$ a
  $\sigma$-finite dominating measure of $\nu$, \text{e.g.} its
  variation norm $\norm{\nu}_1$. Then there exist sequences
  $(\sigma_n)_{0\leq n<N}$ and $(\phi_n)_{0\leq n<N}$ of
  $(\Lambda,\cA)\to(\rset_+,\borel(\rset_+))$ and
  $(\Lambda,\cA)\to(\cH_0,\borel(\cH_0))$ measurable functions,
  respectively, such that the following assertions hold.
  \begin{enumerate}[label=(\roman*),series=eigendecomp-density-tc-povm]
  \item\label{item:lem:eigendecomp-tc-povmi}   For all $\lambda\in\Lambda$,
    $(\sigma_n(\lambda))_{0\leq n<N}$ is non-increasing and
    $\displaystyle\sum_{0\leq n<N} \sigma_n(\lambda)<\infty$.
  \item\label{item:lem:eigendecomp-tc-povmii}    For all $\lambda\in\Lambda$, 
    $(\phi_n(\lambda))_{0\leq n<N}$ is orthonormal.
  \item\label{item:lem:eigendecomp-tc-povmiii} The trace-class \povm\ $\nu$ admits the density
    $$
    f : \lambda\mapsto \sum_{0\leq n<N} \sigma_n(\lambda) \,\phi_n(\lambda) \otimes \phi_n(\lambda) \;,
    $$
    with respect to $\mu$, where the convergence holds absolutely in
    $\cS_1$ for each $\lambda\in\Lambda$.
  \end{enumerate}
  Moreover, using the notations $\phi_n^\adjoint : \lambda \mapsto \phi_n(\lambda)^\adjoint$ and $\phi_n \otimes \phi_n : \lambda \mapsto \phi_n(\lambda) \otimes \phi_n(\lambda)$, we have the following properties.
  \begin{enumerate}[resume*=eigendecomp-density-tc-povm]
  \item\label{item:lem:eigendecomp-tc-povm1} The sequence $(\phi_n^\adjoint)_{0 \leq n < N}$ is orthogonal in
    $\mathsf{L}^2(\Lambda, \calA, \cO(\cH_0,\cset), \nu)$.
  \item\label{item:lem:eigendecomp-tc-povm2}  The sequence $(\phi_n\otimes\phi_n)_{0 \leq n < N}$ is Gramian-orthogonal in
    $\mathsf{L}^2(\Lambda, \calA, \cO(\cH_0), \nu)$.
  \item\label{item:lem:eigendecomp-tc-povm3} The $\cL_b(\cH_0)$-valued
    mapping 
    $\sum_{0\leq n<N}\phi_n\otimes\phi_n$ is equal to the mapping
    $\lambda\mapsto\Id_{\cH_0}$ in
    $\mathsf{L}^2(\Lambda, \calA, \cO(\cH_0), \nu)$.
  \end{enumerate}
\end{lemma}
Assertion~\ref{item:lem:eigendecomp-tc-povm3} may be misleading
at first sight, so the following comment may be worth noting. 
\begin{remark}
  By Assertions
  \ref{item:lem:eigendecomp-tc-povmi}-\ref{item:lem:eigendecomp-tc-povmiii},
  for all $\lambda \in \Lambda$,
  $\sum_{0 \leq n < N} \phi_n(\lambda) \otimes \phi_n(\lambda)$ is the
  orthogonal projection onto the closure of the range of
  $f(\lambda)$. Thus, Assertion~\ref{item:lem:eigendecomp-tc-povm3}
  says that this projection is equal to $\Id_{\cH_0}$ in
  $\mathsf{L}^2(\Lambda, \calA, \cO(\cH_0),\nu)$. It is not equivalent
  to saying that  $\sum_{0 \leq n < N} \phi_n \otimes \phi_n=\Id_{\cH_0}$, $\norm{\nu}_1$-a.e. since it may happen that the range of
  $f(\lambda)$ is dense in $\cH_0$ for none of the $\lambda$'s, in
  which case we have  Assertion~\ref{item:lem:eigendecomp-tc-povm3} at
  the same time as $\{\sum_{0 \leq n < N} \phi_n \otimes \phi_n=\Id_{\cH_0}\}=\emptyset$. 
\end{remark}

We then get the following general formulation of a harmonic principal
components analysis for $\cH_0$-valued weakly-stationary processes, whose proof can be found in \Cref{sec:proofs-crefs-karh}.
We also refer to \Cref{rem:harmonic-pca-filtering} in the same
section, where we explain how to correctly interpret~(\ref{eq:CKL-first}).

\begin{proposition}[Harmonic functional principal components
  analysis]\label{prop:hfpca}
Let $\cH_0$ be a separable Hilbert space and $X = (X_t)_{t \in \lcag}$
be a centered, $\cH_0$-valued weakly-stationary process defined on a
probability space $(\Omega, \cF, \PP)$ with spectral operator measure
$\nu_X$. Let  $(\sigma_n)_{0 \leq n < N}$ and
  $(\phi_n)_{0 \leq n < N}$ be given as in \Cref{lem:eigendecomp-tc-povm}
  for some dominating measure $\mu$ of $\nu_X$, for instance $\mu=\norm{\nu_X}_1$.
Let  $q : \hat{\lcag} \to \nset^*$ be a measurable function. Then for
all $t \in \lcag$,
$$
 \min \set{\PE[\norm{X_t - \left[\filterprocess{\Theta}(X)\right]_t}_{\cH_0}^2] }{\Theta \in \mathsf{L}^2(\hat{\lcag}, \borel(\hat{\lcag}), \cO(\cH_0), \nu_X),  \, \rank(\Theta) \leq q}
 $$
 is equal to
 $$
 \int_{\hat{\lcag}} \sum_{q(\chi) \wedge N \leq n < N} \sigma_n(\chi)\, \mu(\rmd \chi) \; ,
 $$
 and the minimum is achieved for
 $$
 \Theta : \chi \mapsto \sum_{0 \leq n < q(\chi) \wedge N} \phi_n(\chi) \otimes \phi_n(\chi) \;.
 $$
\end{proposition}

\section{Postponed proofs}\label{sec:proofs}
\subsection{Proofs of \Cref{sec:preliminaries}}
\label{sec:proofs-prel-results}
We start with a useful lemma about measurability of compact and Schatten operator-valued functions. 
\begin{lemma}\label{lem:meas-schatten}
Let $(\Lambda,\cA)$ be a measurable space and $\cH_0,\cG_0$  two separable Hilbert spaces. Let $\cE = \cK(\cH_0, \cG_0)$ or $\cS_p(\cH_0, \cG_0)$ where $p \in \{1,2\}$ and. Then a function $\Phi : \Lambda \to \cE$ is measurable if and only if it is simply measurable.  
\end{lemma}
\begin{proof}
  The \emph{only if} part is straightforward and we only show that,
  if $\Phi$ is simply measurable then it is measurable. Since the space $\cE$ is
  separable, Pettis's measurability theorem gives that  it is enough to
  show that for all $f \in \cE^*$, $f \circ \Phi$ is a measurable
  complex-valued function. By \cite[Theorems~19.1, 18.14,
  19.2]{conwaycourseope},  we get that $\cK(\cH_0, \cG_0)^*$,
  $\cS_1(\cH_0, \cG_0)^*$ and $\cS_2(\cH_0, \cG_0)^*$ are respectively
  isometrically isomorphic to $\cS_1(\cH_0, \cG_0)$, $\cL_b(\cH_0, \cG_0)$ and
  $\cS_2(\cH_0, \cG_0)$ and the duality relation can be defined on
  $\cE \times \cE^*$ as $(\aop, \bop) \mapsto \tr(\bop^\adjoint
  \aop)$. This means that we only have to show measurability of the
  complex-valued functions $\lambda \mapsto \tr(\aop^\adjoint \Phi(\lambda))$ for all
  $\aop \in \cE^*$. Let $(\phi_k)_{k \in \nset}, (\psi_k)_{k \in \nset}$ be
  Hilbert basis of $\cH_0$ and $\cG_0$ respectively, then
  $\tr(\aop^\adjoint \Phi(\lambda)) = \sum_{k \in \nset} \pscal{\Phi(\lambda)
    \phi_k}{\aop \psi_k}_{\cG_0}$ which defines a measurable function of $\lambda$ by
  simple measurability of $\Phi$.
\end{proof}
We now provide the proof of \Cref{lem:traceclasspovm} 
\begin{proof}[\bf Proof of \Cref{lem:traceclasspovm}]
  The first point comes from the fact that for all $A \in \calA$,
  $\nu(\Lambda) - \nu(A)$ is a positive operator. Now, if $\nu$ is trace-class, then
  \eqref{eq:sigma-add} is easily verified for the norm
  $\norm{\cdot}_1$ using the fact that $\norm{\cdot}_1 = \tr(\cdot)$
  for positive operators. Finally, by definition of $\norm{\nu}_1$,
  regularity of $\norm{\nu}_1$ is equivalent to regularity of $\nu$ as
  an $\cS_1(\cH_0)$-valued measure which clearly implies regularity of
  $\nu_x = x^\adjoint \nu(\cdot) x$ for all $x \in \cH_0$. Suppose now that for all
  $x \in \cH_0$, $\nu_x$ is regular, then let
  $(e_k)_{k \in \nset}$ be a Hilbert basis of $\cH_0$, and define for
  all $n \in \nset$, the non-negative measure
  $\mu_n := \sum_{k = 0}^n \nu_{e_k}$ such that for all
  $A \in \calA$,
  $\norm{\nu}_1(A) = \lim_{n \to +\infty} \mu_n(A) = \sup_{n \in
    \nset} \mu_n(A)$. Then, by Vitali-Hahn-Sakh-Nikodym's theorem (see
  \cite{BrooksVHStheorem}), the sequence $(\mu_n)_{n \in \nset}$ is
  uniformly countably additive which implies regularity of
  $\norm{\nu}_1$ by Lemma 23 in  \cite[Chapter~VI, Section~2]{diestel1977vector}.
\end{proof}

We also provide the following useful properties on the density of a trace-class \povm\ with respect to a dominating measure. 
\begin{lemma}\label{thm:radon-nikodym-povm}
  Les $\nu$ be a trace-class \povm\ on $(\Lambda,\cA,\cH_0)$ and $\mu$ a $\sigma$-finite measure such that $\norm{\nu}_1 \ll \mu$. Les $g = \frac{\rmd \nu}{\rmd \mu}$. Then the following assertions hold.
  \begin{enumerate}[label=(\alph*)]
  \item\label{itm:density-positive} For $\mu$-almost every $\lambda
    \in \Lambda$, $g(\lambda)\in\cS_1^+(\cH_0)$.
  \item\label{itm:density-sqaure-root} The mapping
  $g^{1/2}:\lambda\mapsto g(\lambda)^{1/2}$ belongs to $L^2(\Lambda, \cA, \cS_2(\cH_0), \mu)$.
  \item\label{itm:density-trace} The density of $\norm{\nu}_1$ with
    respect to $\mu$ is $\norm{g}_1$. In particular, $g = \frac{\rmd
      \nu}{\rmd \norm{\nu}_1} \norm{g}_1$ $\mu$-a.e. and if $\mu = \norm{\nu}_1$, then $\norm{g}_1= 1$ $\mu$-a.e.
   \item\label{itm:density-trace-integral} Let
     $f:\Lambda\to\cset$ be measurable. Then $f\in
  L^1(\Lambda,\cA,\norm{\nu}_1)$ if and only if $\lambda\mapsto f(\lambda)\,g(\lambda)\in
  L^1(\Lambda,\cA,\cS_1(\cH_0),\mu)$, and we have
  $\displaystyle
  \int f(\lambda) \, \nu(\rmd \lambda)=\int f(\lambda) \,g(\lambda) \;\mu(\rmd \lambda)$.
  \end{enumerate}
\end{lemma}
\begin{proof}
  For all $x \in \cH_0$ and
  $A \in \calA$,
  $$
  \int_A \pscal{g(\lambda) x}{x}_{\cH_0} \, \mu(\rmd\lambda) = \pscal{\nu(A) x}{x}_{\cH_0} \geq 0\;,
  $$
  and there exists a set $A_x \in \calA$ with $\mu(A_x^c) = 0$
  and $\pscal{g(\lambda)x}{x}_{\cH_0} \geq 0$ for all
  $\lambda \in A_x$. Taking $(x_n)_{n \in \nset}$ a dense countable
  subset of $\cH_0$ we get that $g\in\cS_1^+(\cH_0)$ on
  $A=\bigcap_{n \in \nset} A_{x_n}$ thus proving
  Assertion~\ref{itm:density-positive}.
  For Assertion~\ref{itm:density-sqaure-root}, we get have $\g^{1/2} \in \functionsetarg{\Lambda, \calA,\cS_2(\cH_0)}$ by Lemma~2 in \cite[Section~3.4]{yuichiro-multidim} and \Cref{lem:meas-schatten} and $g^{1/2} \in L^2(\Lambda, \cA, \cS_2(\cH_0), \mu)$ then
  follows from the identity   $\norm{g^{1/2}(\lambda)}_2^2 = \norm{g(\lambda)}_1$.  Moreover, taking the trace in
  \eqref{eq:density-povm} gives for all $A \in \calA$,
  $$
  \norm{\nu}_1(A) = \int_A \norm{g}_1 \, \rmd \mu
  $$
  which gives Assertion
  \ref{itm:density-trace}. Finally, Assertion~\ref{itm:density-trace-integral} is 
  easy to get by extending the case $f=\1_A$ for $A\in\cA$ to simple
  functions and then using the density of simple functions.  
\end{proof}

\subsection{Proofs of \Cref{sec:more-preliminaries}}
\label{sec:proofs-preliminaries}
We start by exhibiting the relation between the spaces $\mathsf{L}^2(\Lambda,\calA, \cO(\cH_0, \cG_0), \nu)$ and $L^2(\Lambda, \calA,\cS_2(\cH_0, \cG_0), \mu)$ where $\norm{\nu}_1 \ll \mu$. It is easy to show that this last space is a normal Hilbert $\cL_b(\cG_0)$-module with module action defined, for $\aop\in\cL_b(\cG_0)$ and
  $\Phi\in L^2(\Lambda, \calA, \cS_2(\cH_0, \cG_0), \mu)$, by
  $\aop\bullet\Phi:\lambda\mapsto\aop\Phi(\lambda)$ and Gramian defined, for 
  $\Phi,\Psi\in L^2(\Lambda, \calA, \cS_2(\cH_0, \cG_0), \mu)$, by
\begin{equation}\label{eq:gramian-L2S2}
\gramian{\Phi}{\Psi}_{L^2(\Lambda, \calA, \cS_2(\cH_0, \cG_0), \mu)}:=\int\Phi\Psi^\adjoint\;\rmd\mu \; . 
\end{equation}
The following proposition provides an easy way to verify that a function belongs in $\mathsf{L}^2(\Lambda,\calA, \cO(\cH_0, \cG_0), \nu)$. 
 \begin{proposition}\label{prop:charac-L2O}
  Let $(\Lambda, \cA)$ be a measurable space, $\cH_0, \cG_0,\cI_0$ be
  three separable Hilbert spaces and $\nu$ a trace-class \povm\ on
  $(\Lambda, \cA, \cH_0)$. Let
  $\mu$ be a $\sigma$-finite non-negative measure dominating $\norm{\nu}_1$ and set
  $g = \frac{\rmd \nu}{\rmd \mu}$.   
  Then  the following assertions hold.
  \begin{enumerate}[label=(\alph*)]
  \item \label{item:prop:charac-L2O-1} For all $\Phi
    \in\weakmeasfunctionsetarg{\Lambda,\calA,\cH_0, \cG_0}$, we have $\Phi \in
    \mathscr{L}^2(\Lambda, \calA, \cO(\cH_0, \cG_0), \nu)$ if and only if
      \begin{equation*}
    \begin{cases}
      \range(g^{1/2}) \subset \domain(\Phi)$ and $\Phi g^{1/2} \in\cS_2(\cH_0, \cG_0)$, $\mu \text{-a.e.} \\
      \displaystyle\int \norm{\Phi g^{1/2}}_2^2\;\rmd\mu<\infty \;.
    \end{cases}
  \end{equation*}
  \item \label{item:prop:charac-L2O-2}   If $\Phi, \Psi \in \mathscr{L}^2(\Lambda, \calA, \cO(\cH_0, \cG_0), \nu)$, then $(\Phi, \Psi)$ is $\nu$-integrable and
  \begin{equation}\label{eq:square-integral-unbounded-2}
    \int \Phi \rmd \nu \Psi^\adjoint  = \gramian{\Phi g^{1/2}}{\Psi g^{1/2}}_{L^2(\Lambda, \calA, \cS_2(\cH_0, \cG_0), \mu)}\;,
  \end{equation}
where the latter Gramian comes from \eqref{eq:gramian-L2S2}. Hence the mapping $\Phi \mapsto \Phi
    g^{1/2}$ is Gramian-isometric from $\mathsf{L}^2(\Lambda,
    \calA, \cO(\cH_0, \cG_0), \nu)$ to $L^2(\Lambda, \calA,
    \cS_2(\cH_0, \cG_0), \mu)$.
  \end{enumerate}
\end{proposition}
\begin{proof}
 Let $f = \frac{\rmd \nu}{\rmd \norm{\nu}_1}$. Using that $\norm{\nu}_1(\{g = 0\}) = \int_{\{g = 0\}} \norm{g}_1 \,
  \rmd \mu =  0$ and $g = f \norm{g}_1$ $\mu$-a.e. by uniqueness of
  the density, we get that
  \begin{equation}
    \label{eq:proof-prop:charac-L2O:1}
  \norm{g}_1>0\quad\norm{\nu}_1\text{-a.e.}\quad\text{and}\quad g = f \norm{g}_1\quad\mu\text{-a.e.}  
  \end{equation}
  (and thus also $\norm{\nu}_1$-a.e. since $\norm{\nu}_1 \ll
  \mu$). From this observation, we easily get  that Assertions~\ref{itm:range-in-dom}, \ref{itm:compo-in-S2} and \ref{itm:inL1traceclass} of \Cref{def:square-integrability} are respectively equivalent to
  \begin{enumerate}[label=(\roman*')]
  \item\label{itm:range-in-dom-mu} We have $\range(g^{1/2}) \subset \domain(\Phi)$ and $\range(g^{1/2}) \subset \domain(\Psi)$, $\mu$-a.e.
  \item\label{itm:compo-in-S2-mu}  We have $\Phi g^{1/2}\in\cS_2(\cH_0, \cG_0)$ and $\Psi
    g^{1/2}\in\cS_2(\cH_0,
    \cI_0)$,  $\mu$-a.e.
  \item\label{itm:inL1traceclass-mu} $(\Phi g^{1/2})(\Psi g^{1/2})^\adjoint \in \cL^1(\Lambda, \calA, \cS_1(\cG_0,\cI_0), \mu)$. 
  \end{enumerate}
  We also easily get  that 
  \begin{equation}\label{eq:integral-L2-unbounded-mu}
    \int \Phi\rmd \nu \Psi^\adjoint = \int (\Phi g^{1/2})(\Psi g^{1/2})^\adjoint \, \rmd \mu \; .
  \end{equation}
  Let us for instance detail the proof of the equivalence
  between \ref{itm:range-in-dom-mu} and
  \ref{itm:range-in-dom} of \Cref{def:square-integrability}. The
  left-hand side of~(\ref{eq:proof-prop:charac-L2O:1}) gives that
  \begin{equation}
    \label{eq:proof-prop:charac-L2O:2}
  \norm{\nu}_1 \left( \left\lbrace\range(f^{1/2}) \not\subset
      \domain(\Phi) \right\rbrace \right)= \norm{\nu}_1 \left( \left\lbrace\range(f^{1/2}) \not\subset \domain(\Phi) \right\rbrace \cap \{g \neq 0\} \right)\;,
\end{equation}
and its right-hand side yields
  \begin{align}
    \label{eq:proof-prop:charac-L2O:3}
   \mu \left( \left\lbrace\range(f^{1/2}) \not\subset \domain(\Phi) \right\rbrace \cap \{g \neq 0\} \right) 
    &= \mu \left( \left\lbrace\range(g^{1/2}) \not\subset \domain(\Phi) \right\rbrace \cap \{g \neq 0\} \right) \\
\nonumber    &= \mu \left( \left\lbrace\range(g^{1/2}) \not\subset \domain(\Phi) \right\rbrace \right) \;,
  \end{align}
  since $\left\lbrace\range(g^{1/2}) \not\subset \domain(\Phi) \right\rbrace \cap \{g = 0\}=\emptyset$.
  To get \ref{itm:range-in-dom-mu} $\Leftrightarrow$
  \ref{itm:range-in-dom}, we note that
  $$
  \norm{\nu}_1 \left( \left\lbrace\range(f^{1/2}) \not\subset
      \domain(\Phi) \right\rbrace \cap \{g \neq 0\} \right) =
  \int_{\left\lbrace\range(f^{1/2}) \not\subset \domain(\Phi)
    \right\rbrace\cap \{g \neq 0\}} \norm{g}_1 \, \rmd \mu\;,
  $$
  and thus the right-hand side of~(\ref{eq:proof-prop:charac-L2O:2}) is
  zero if and only if the left-hand side of~(\ref{eq:proof-prop:charac-L2O:3}) is.
  Now, Assertions~\ref{item:prop:charac-L2O-1} and~\ref{item:prop:charac-L2O-2} come easily using the definition of
  $\mathscr{L}^2(\Lambda, \calA, \cO(\cH_0, \cG_0), \nu)$. Note that measurability of $\Phi g^{1/2}$ and
  $(\Phi g^{1/2})(\Phi g^{1/2})^\adjoint$ are ensured by $\cO$-measurability of
  $\Phi$, simple measurability of $g$ and \Cref{lem:meas-schatten}.
\end{proof}
We can now derive  \Cref{thm:pseudo-gramian-L2-unbounded}.
\begin{proof}[\bf Proof of \Cref{thm:pseudo-gramian-L2-unbounded}]
  All theses results are easily derived from \Cref{prop:charac-L2O} and the module nature of $L^2(\Lambda,\calA, \cS_2(\cH_0, \cG_0), \mu)$. The only difficulty lies in showing the completeness of $\mathsf{L}^2(\Lambda, \cA, \cO(\cH_0, \cG_0), \nu)$, which is detailed in the proof of Theorem~11 in \cite[Section~3.4]{yuichiro-multidim}. 
\end{proof}

We now provide a useful result about dense subsets of $\mathsf{L}^2(\Lambda, \cA, \cO(\cH_0, \cG_0), \nu)$.
\begin{theorem}\label{thm:l2op-density}
  Let $\cH_0, \cG_0$ be two separable Hilbert spaces,
  $(\Lambda, \calA)$ a measurable space, and $\nu$ a trace-class
  \povm\ on $(\Lambda, \calA, \cH_0)$. Then the space
  $L^2(\Lambda, \calA, \cL_b(\cH_0, \cG_0), \norm{\nu}_1)$ is dense in
  $\mathsf{L}^2(\Lambda, \calA, \cO(\cH_0,\cG_0), \nu)$ and 
  the following assertions hold.
  \begin{enumerate}[label=(\roman*)]
  \item\label{item:thm:l2op-density:ass1}   The space
  $\lspan{\1_A\, \aop \; : \; A \in \calA, \aop \in
    \cL_b(\cH_0, \cG_0)}$ of simple $\cL_b(\cH_0, \cG_0)$-valued
  functions is dense in
  $\mathsf{L}^2(\Lambda, \calA, \cO(\cH_0,\cG_0), \nu)$.
\item\label{item:thm:l2op-density:ass2} For any subset $E\subset L^2(\Lambda, \calA, \norm{\nu}_1)$
  which is linearly dense in $L^2(\Lambda, \calA, \norm{\nu}_1)$,
  the space
  $\lspan{h\, \aop \; : \; h\in E, \aop \in
    \cL_b(\cH_0, \cG_0)}$  is dense in
  $\mathsf{L}^2(\Lambda, \calA, \cO(\cH_0,\cG_0), \nu)$.
\end{enumerate}
\end{theorem}
\begin{proof}
  In the first two steps of the proof of Theorem~12 in
  \cite[Section~3.4]{yuichiro-multidim} (see also \cite[Theorem~4.22]{mandrekar-square-integrability}), it is shown that, if 
  $\Phi \in \mathsf{L}^2(\Lambda, \calA, \cO(\cH_0,\cG_0), \nu)$ and
  $\epsilon > 0$, there exists $\Psi \in L^2(\Lambda, \calA, \cL_b(\cH_0,\cG_0), \norm{\nu}_1) \subset
  \mathsf{L}^2(\Lambda, \calA, \cO(\cH_0,\cG_0), \nu)$ such that
  $\norm{\Phi - \Psi}_{\nu} <
  \epsilon$. This implies that $L^2(\Lambda, \calA, \cL_b(\cH_0, \cG_0), \norm{\nu}_1)$ is dense in
  $\mathsf{L}^2(\Lambda, \calA, \cO(\cH_0,\cG_0), \nu)$. Then Assertion~\ref{item:thm:l2op-density:ass1} 
  follows using  the usual
  density of simple
  functions and the fact that, for all $\Phi\in L^2(\Lambda, \calA, \cL_b(\cH_0, \cG_0),\norm{\nu}_1)$,
      \begin{multline}
      \label{eq:Lb-normal-pre-hilbert-module-cont-embedding}
      \norm{\Phi}^2_{\nu}=\tr
      \int\Phi \,f\, \Phi^\adjoint\;\rmd\norm{\nu}_1=
      \int\tr\left(\Phi \,f\,
        \Phi^\adjoint\right)\;\rmd\norm{\nu}_1\\
      \leq       \int\norm{\Phi}^2_{\cL_b(\cH_0, \cG_0)} \;\rmd\norm{\nu}_1=
\norm{\Phi}^2_{L^2(\Lambda, \calA, \cL_b(\cH_0, \cG_0),
    \norm{\nu}_1)}\;,
\end{multline}
where we used again that $\norm{f}_{1}=1$, $\norm{\nu}_1$-a.e. Assertion~\ref{item:thm:l2op-density:ass2} then follows by approximating, for any
  $A\in\cA$ and $\aop\in\cL_b(\cH_0, \cG_0)$ the function $\1_A\aop$
  by  $g\aop$ with $g\in\lspan E$ arbitrarily close to $\1_A$ in
  $L^2(\Lambda, \calA, \norm{\nu}_1)$.  
\end{proof}

With this in mind, we can prove \Cref{thm:integral-cagos}. 
\begin{proof}[\bf Proof of \Cref{thm:integral-cagos}]
  We set $\cH = \cM(\Omega, \cF, \cH_0, \PP)$ and
  $\cG = \cM(\Omega, \cF, \cG_0, \PP)$.
  For all $A,B \in \calA$ and $\aop,\bop \in \cL_b(\cH_0, \cG_0)$, we have,
  by \Cref{thm:pseudo-gramian-L2-unbounded},
  \begin{align*}
    \gramian{\1_A \aop}{\1_B \bop}_{\nu_W}
    &= \aop \nu_W(A \cap B) \bop^\adjoint \\
    &= \aop \Cov{W(A)}{W(B)} \bop^\adjoint \\
    &= \Cov{\aop W(A)}{\bop W(B)} \\
    &= \gramian{\aop W(A)}{\bop W(B)}_{\cG} \; .
  \end{align*}
  Then \Cref{prop:gramian-isometric-extension}, applied to $J = \calA \times \cL_b(\cH_0, \cG_0)$ with $v_{(A, \aop)} = \1_A \aop$ and $w_{(A, \aop)} = \aop W(A)$, gives that there exists a unique Gramian-isometric operator 
    \begin{equation}\label{eq:def:integral-cagos-extension}
  I_W^{\cG_0} : \cspanarg[\mathsf{L}^2(\Lambda, \calA, \cO(\cH_0, \cG_0), \nu_W)]{\1_A \bop \aop  \; : \; A \in \calA, \aop \in \cL_b(\cH_0, \cG_0), \bop \in \cL_b(\cG_0)} \to \cG
\end{equation}
such that for all $A \in \calA, \aop \in \cL_b(\cH_0, \cG_0)$, $I_W^\cG(\1_A \aop) = \aop W(A)$ and, in addition,
  \begin{equation}\label{eq:rangeIWG}
    \range(I_W^{\cG_0}) = \cspanarg[\cG]{\bop \aop W(A)  \; : \; A \in \calA, \aop \in \cL_b(\cH_0, \cG_0), \bop \in \cL_b(\cG_0)} \; .
  \end{equation}
  Now, note that
  \begin{equation}\label{eq:compo-bounded}
    \cL_b(\cH_0, \cG_0) = \set{\bop \aop}{\aop \in \cL_b(\cH_0, \cG_0), \bop \in \cL_b(\cG_0)} \; .
  \end{equation}
  This gives that
  $$
  \lspan{\1_A \bop \aop  \; : \; A \in \calA, \aop \in \cL_b(\cH_0, \cG_0), \bop \in \cL_b(\cG_0)} = \lspan{\1_A \aop \; : \; A \in \calA, \aop \in \cL_b(\cH_0, \cG_0)}\;.
  $$
  Therefore, by \Cref{thm:l2op-density}, the domain of 
  $I_W^{\cG_0}$ in~(\ref{eq:def:integral-cagos-extension}) is the whole
  space $\mathsf{L}^2(\Lambda, \calA, \cO(\cH_0, \cG_0), \nu_W)$. 
  Finally, \eqref{eq:compo-bounded} with \eqref{eq:rangeIWG} yields
  $$
  \range(I_W^{\cG_0})
  = \cspanarg[\cG]{\aop W(A)  \; : \; A \in \calA, \aop \in \cL_b(\cH_0, \cG_0)}
  = \cH^{W,\cG_0} \;,
  $$
  which concludes the proof.
\end{proof}
\subsection{Proofs of \Cref{sec:gram-cram-bochn}}
\label{sec:proofs-spectral-rep}
Let us start with the proof of the Gramian-Cramér representation
theorem, as a consequence of the Stone theorem. Our proof is mainly a more detailed version of the proof of Theorem~2 in \cite[Section~4.2]{yuichiro-multidim}. However, for completeness, we also prove the uniqueness of $\hat{X}$ and the converse statement (\Cref{lem:cramer-easy}) whose proofs are not  provided in \cite{yuichiro-multidim}.
The usual Stone
theorem (see \eg\ \cite[Chapter~IX]{conway1994course}) says that any continuous
isomorphism $h\mapsto U_h$ from an \lca\ group $\lcag$ to the set of unitary
operators from a Hilbert space $\cH$ onto itself can be represented as an integral of this
mapping, that is,
$$
U_h=\int \chi(h)\;\xi(\rmd \chi)\;,
$$
where $\xi$ is a \povm\ defined on the dual set of characters $\hat\lcag$
endowed with its Borel $\sigma$-field and valued in the set of
orthogonal projections on $\cH$.
This classical 
theorem has a counterpart in the case where $\cH$ is an
$\cL_b(\cH_0)$-normal Hilbert module and each $U_h$ is not only unitary
but also Gramian-unitary, in which case $\xi$ is valued in the set of
orthogonal projections   on $\cH$ whose ranges  are closed
submodules.  See \cite[Section~2.5]{yuichiro-multidim} for details.  It turns out that such \povm's are related to \cagos\ measure by the following lemma.
\begin{lemma}\label{lem:gramian-projection-povm-cagos}
  Let $\cH_0$ be a separable Hilbert space, $\cH$ an
  $\cL_b(\cH_0)$-normal Hilbert module and $(\Lambda, \cA)$ a
  measurable space. Let $\xi$ be a \povm\ on $(\Lambda, \cA, \cH)$
  valued in the set of orthogonal projections on $\cH$ whose ranges
  are closed submodules. Then for all $x_0 \in \cH$, the mapping
  $\xi x_0 : A \mapsto \xi(A) x_0$ is a \cagos\ measure on
  $(\Lambda, \cA, \cH)$ which is regular if $\xi$ is
  regular.
  \end{lemma}
\begin{proof}
  Using the fact that $\xi$ is a \povm\ on $(\Lambda, \cA, \cH)$ and
  \cite[Proposition~1]{Berberian1966NotesOS}, it is straightforward to
  see that $\xi x_0$ is an $\cH$-valued measure. Moreover, since $\xi$
  is valued in the set of orthogonal projections on $\cH$ whose ranges
  are closed submodules, we get that for all disjoint
  $A,B \in \borel(\lcag)$
  $$
  \gramian{\xi(A)x_0}{\xi(B) x_0}_{\cH} =  \gramian{\xi(B) \xi(A)x_0}{x_0}_{\cH} = \gramian{\xi(B \cap A)x_0}{x_0}_{\cH}  = 0 \; ,
  $$
  where the first equality is justified in
  \cite[P.~23]{yuichiro-multidim} and the second one by
  \cite[Theorem~3]{Berberian1966NotesOS}. This proves that $\xi x_0$
  is a \cagos\ measure on $(\Lambda,\cA,\cH)$. In the following, we
  denote by $\nu$ its intensity 
  operator measure. Then,  for all $A \in \cA$, we have 
  $$
  \norm{\nu(A)}_1=\tr{\gramian{\xi(A) x_0}{\xi(A) x_0}_{\cH}} = \pscal{\xi(A) x_0}{x_0}_{\cH}\;,  
  $$
  where the last equality comes from the fact that $\xi(A)$ is an orthogonal projection on $\cH$. 
  Now, if $\xi$ is regular, then the measure $A \mapsto
  \pscal{\xi(A) x_0}{x_0}$ is regular and so is $\norm{\nu}_1$ by the
  previous display. This implies that $\xi x_0$ is regular and the
  proof is concluded.
\end{proof}

The proof of uniqueness in \Cref{thm:spectral-analysis-hilbert} requires the following result which is a kind of Fubini theorem for
interchanging a Bochner integral with a \cagos\ 
integral.

\begin{proposition}\label{prop:fubini-bochner-cagos}
  Let $(\Lambda, \calA)$ be a measurable space and $\cH_0$, $\cG_0$ two
  separable Hilbert spaces. Let $W$ be an $\cH_0$-valued random \cagos\ measure on
  $(\Lambda,\calA,\Omega,\cF,\PP)$ with intensity operator
  measure $\nu_W$. Let $\mu$ be a non-negative
  measure on a measurable space  $(\Lambda',\calA')$. Suppose that
  $\Phi$ is measurable from $\Lambda\times\Lambda'$
  to $\cL_b(\cH_0,\cG_0)$ and 
  satisfies 
  \begin{align}
    \label{eq:fubini-cagos-bichner-cond1}
 \int\left(\int  \norm{\Phi(\lambda,\lambda')}_{\cL_b(\cH_0,\cG_0)}\;
    \mu(\rmd\lambda')\right)^2\norm{\nu_W}_1(\rmd\lambda)<\infty\;,\\
        \label{eq:fubini-cagos-bichner-cond2}
 \int\left(\int  \norm{\Phi(\lambda,\lambda')}_{\cL_b(\cH_0,\cG_0)}^2\;
  \norm{\nu_W}_1(\rmd\lambda)\right)^{1/2}\mu(\rmd\lambda')<\infty\;.    
  \end{align}
  Then we have
    \begin{equation}
    \label{eq:fubini-cagos-bichner-formula}
  \int\left(\int
    \Phi(\lambda,\lambda')\;\mu(\rmd\lambda')\right)\;
  W(\rmd\lambda)=
    \int\left(\int
    \Phi(\lambda,\lambda')\;W(\rmd\lambda)\right)\;\mu(\rmd\lambda')\;,
\end{equation}
where integrals with respect to $W$ are as
in \Cref{def:integral-cagos}, in the left-hand side the innermost
integral is understood as a Bochner integral on
$L^2(\Lambda',\cA',\cL_b(\cH_0,\cG_0),\mu)$
and in the right-hand side, the outermost integral is understood as a
Bochner integral on $L^2(\Lambda',\cA',\cM(\Omega, \cF, \cG_0, \PP),\mu)$.
\end{proposition}
\begin{proof}
  Conditions~(\ref{eq:fubini-cagos-bichner-cond1})
  and~(\ref{eq:fubini-cagos-bichner-cond2}) ensure that
  $\Phi(\lambda,\cdot)\in
  L^1(\Lambda',\calA',\cL_b(\cH_0,\cG_0),\mu)$ for
  $\norm{\nu_W}_1$-a.e.
  $\lambda\in\Lambda$ and that
  $\Phi(\cdot,\lambda')\in
  L^2(\Lambda,\calA,\cL_b(\cH_0,\cG_0),\norm{\nu}_1)$ for
  $\mu$-a.e.
  $\lambda'\in\Lambda'$, showing that the innermost integrals in both
  sides of~(\ref{eq:fubini-cagos-bichner-formula}) are well defined
  for adequate sets of $\lambda$ and $\lambda'$, respectively.
  Let $E_1$ and $E_2$ denote the sets of functions $\Phi$  measurable from $\Lambda\times\Lambda'$
  to $\cL_b(\cH_0,\cG_0)$  and
  satisfying~(\ref{eq:fubini-cagos-bichner-cond1})
  and~(\ref{eq:fubini-cagos-bichner-cond2}), respectively.  We denote by
  $\norm{\Phi}_{E_1}$ the square root of the left-hand side
  of~(\ref{eq:fubini-cagos-bichner-cond1}) and by  $\norm{\Phi}_{E_2}$
  the left-hand side of~(\ref{eq:fubini-cagos-bichner-cond2}), which make
  $E_1$ and $E_2$ Banach spaces. Then, for all $\Phi\in E:= E_1\cap
  E_2$, concerning the left-hand side of~(\ref{eq:fubini-cagos-bichner-formula}), we have
$$
\norm{\int
  \Phi(\cdot,\lambda')\;\mu(\rmd\lambda')}^2_{\nu_W}\leq
\int\norm{\int
  \Phi(\cdot,\lambda')\;\mu(\rmd\lambda')}^2_{\cL_b(\cH_0,\cG_0)}\;\rmd\norm{\nu_W}_1\leq
\norm{\Phi}_{E_1}^2\;,
$$
as for the right-hand side, we have,  setting $\cH:=\cM(\Omega, \cF, \cG_0, \PP)$,
$$
\int\norm{\int
  \Phi(\lambda,\cdot)\;W(\rmd\lambda)}_\cH\;\rmd \mu=
\int\norm{\Phi(\cdot,\lambda')}_{\nu_W}\;\mu(\rmd \lambda')\leq
\norm{\Phi}_{E_2}\;,
$$
These two
inequalities show that both sides
of~(\ref{eq:fubini-cagos-bichner-formula}) seen as functions of $\Phi$
are linear continuous from $E$ endowed with the norm $\norm{\cdot}_{E}=\norm{\cdot}_{E_1}+\norm{\cdot}_{E_2}$ to $\cM(\Omega, \cF, \cG_0,
\PP)$. Since they coincide for $\Phi(\lambda, \lambda') =
\1_A(\lambda) \1_B(\lambda') \aop$ with $A\in\cA$,
$B\in\cA'$ and $\aop \in \cL_b(\cH_0,\cG_0)$, this concludes the proof. 
\end{proof}
An interesting straightforward application of
\Cref{prop:fubini-bochner-cagos} is the following extension of \Cref{ex:conv-discrete}.
\begin{example}[Convolutional filtering]
  \label{ex:conv-allG}
Let $\cH_0$ and $\cG_0$ be two separable Hilbert spaces.
 Let $X=(X_t)_{t\in\lcag}$ be an $\cH_0$-valued weakly stationary stochastic process
 defined on $(\Omega,\cF,\PP)$. 
  Let $\eta$ be the Haar measure on $\lcag$ (see
  \cite[Chapter~1]{rudin1990fourier}) and $\Phi\in
  L^1(\lcag,\borel(\lcag),\cL_b(\cH_0,\cG_0),\eta)$. Define the process $Y=(Y_t)_{t\in\lcag}$ by the \emph{time
    domain convolutional filtering}
  $$
  Y_t=\int \Phi(s)\,X_{t-s}\;\eta(\rmd s)\;,\quad t\in\lcag\;,
  $$
  where the integral is a Bochner integral on $L^1(\lcag,\borel(\lcag),\cM(\Omega, \cF, \cG_0, \PP),\eta)$.
  Then, defining $\hat\Phi:\hat\lcag\to\cL_b(\cH_0,\cG_0)$ by the
  following Bochner integral on $L^1(\lcag,\borel(\lcag),\cL_b(\cH_0,\cG_0),\eta)$,
  $$
  \hat\Phi(\chi)=\int \Phi(s)\;\overline{\chi(s)}\;\eta(\rmd
  s)\;,\quad\chi\in\hat\lcag\;.
  $$.
  It is then immediate to show
  that~(\ref{eq:fubini-cagos-bichner-cond1})
  and~(\ref{eq:fubini-cagos-bichner-cond2}) hold with
  $\Phi:(\chi,s)\mapsto\Phi(s)\;\overline{\chi(s)}$,  $\mu=\eta$ and
  $W=\hat X$, and \Cref{prop:fubini-bochner-cagos} yields
  $Y=\filterprocess{\hat\Phi}(X)$.
\end{example}
We can now provide a detailed proof of \Cref{thm:spectral-analysis-hilbert}. 
\begin{proof}[\bf Proof of \Cref{thm:spectral-analysis-hilbert}]
  Suppose that $X$ is weakly stationary as in
  \Cref{def:functional-weakstat}. Then the collection of lag operators
  $(U_h^X)_{h\in\lcag}$ 
  satisfies the assumptions of the generalized Stone's theorem stated
  as Proposition~4 in \cite[Section~2.5]{yuichiro-multidim}. This gives that there
  exists a regular \povm\ $\xi^X$ on
  $(\hat{\lcag}, \borel(\hat{\lcag}),\cH^X)$ valued in the set of
  orthogonal projections whose ranges are closed submodules of $\cH^X$ such
  that, for all $h \in \lcag$,
  \begin{equation}\label{eq:stone-shift-module}
    U^X_h = \int \chi(h) \, \xi^X(\rmd \chi)\;,
  \end{equation}
  where the integral is as in Theorem~9 in \cite{Berberian1966NotesOS}.
  Then, by \Cref{lem:gramian-projection-povm-cagos}, the mapping
  \begin{equation}\label{eq:def:hatX-stone}
    \hat{X}: \fundef{\borel(\hat{\lcag}) & \to & \cH^X  \\ A & \mapsto & \xi^X(A) X_0}
\end{equation}
  is a regular \cagos\ measure on $(\hat\lcag, \borel(\hat\lcag),
  \cH^X)$ and we denote by $\nu_X$ its intensity operator
  measure. Since $\cH^X$ is a submodule of $\cM(\Omega, \cF, \cH_0,
  \PP)$, $\hat{X}$ is also a regular $\cH_0$-valued random \cagos\
  measure on $(\Omega, \cF, \PP)$, see \Cref{def:cagos}. Relation
  \eqref{eq:spectralrep-functional} then follows by
  applying~(\ref{eq:stone-shift-module}) and the fact that, for all $t
  \in \lcag$,  $U^h_t X_0 = X_t$ and, for all
  $\phi : \Lambda \to \cset$ measurable and
  bounded,
    \begin{equation}\label{eq:integral-cagos-projection}
    \int \phi \, \rmd \hat X = \left(\int \phi \, \rmd \xi^X \right) X_0 \; ,
  \end{equation}
  where the integral in the left-hand side is defined as in
  \Cref{def:integral-cagos} (see also \Cref{rem:scalar-int-cagos}) and
  the integral in the right-hand side as in  Theorem~9 in \cite{Berberian1966NotesOS}, for the \povm\ $\xi^X$. Relation~(\ref{eq:integral-cagos-projection}) obviously
  holds if $\phi=\1_A$ with $A\in\cA$ and also for $\phi$ simple by
  linearity. Now, for a general measurable and bounded $\phi : \Lambda \to \cset$,
  we can find a sequence $(\phi_n)_{n\in\nset}$ of
  simple functions such that $|\phi_n|\leq|\phi|$ for all $n\in\nset$ and
  $\phi_n(\lambda)\to\phi(\lambda)$ as $n\to\infty$ for all
  $\lambda \in\Lambda$. Then, by dominated convergence, $\phi_n$
  converges to $\phi$ in $L^2(\Lambda, \cA, \norm{\nu}_1)$ and
  therefore $\phi_n \Id$ converges to $\phi \Id$ in
  $\mathsf{L}^2(\Lambda, \cA, \cO(\cH_0), \nu)$. Thus
  $\int \phi_n \, \rmd \hat X \to \int \phi \, \rmd \hat X$ in $\cH^X$
  by the isometric property of the integral of
  \Cref{def:integral-cagos}.  To get
  \eqref{eq:integral-cagos-projection}, it now suffices to show that,
  for all $Y\in\cH^X$,
  $\pscal{\left(\int \phi_n \, \rmd \xi\right)X_0}{Y}_{\cH^X}
  \to\pscal{\left(\int \phi \, \rmd \xi\right)X_0}{Y}_{\cH^X}$. This
  follows from the  polarization formula,   Theorem~9 in \cite{Berberian1966NotesOS}
  and dominated convergence.

To show uniqueness, suppose there exists another regular
$\cH_0$-valued random \cagos\ measure $W$ on
$(\hat{\lcag}, \borel(\hat{\lcag}), \Omega,\cF,\PP)$ satisfying the
same identity as \eqref{eq:spectralrep-functional} with $\hat{X}$
replaced by $W$. Then, we get
  \begin{equation}\label{eq:W-hatX-equal-on-characters}
    \int \chi(t) \; \hat{X}(\rmd \chi)= \int \chi(t) \; W(\rmd \chi)\quad\text{for all}\quad t \in \lcag\;.
  \end{equation}
Let $\eta$ denote the Haar measure
on $\lcag$ and denote by $\cC_c(\lcag)$ the space of compactly
supported functions from $\lcag$ to $\cset$. Then, by
\cite[Theorem~1.2.4]{rudin1990fourier} and
\cite[Section~E.8]{rudin1990fourier}, the space
$$
E=\set{\hat\phi:\chi\mapsto\int\phi(t)\overline{\chi(t)}\;\eta(\rmd t)}{\phi\in L^1(\lcag,\borel(\lcag), \eta)}
$$
is dense in
$L^2(\hat\lcag,\borel(\hat\lcag),\norm{\nu_W}_1+\norm{\nu_X}_1)$.  We
can thus find, for any $A\in\borel(\hat\lcag)$,
$(\phi_n)_{n\in\nset}\in\cC_c(\lcag)^\nset$ such that,  defining $\hat\phi_n$ as
above, $\hat\phi_n\to\1_A$ both in
$L^2(\hat\lcag,\borel(\hat\lcag),\norm{\nu_W}_1)$ and in
$L^2(\hat\lcag,\borel(\hat\lcag),\norm{\nu_X}_1)$. Then by
\Cref{prop:fubini-bochner-cagos}, we have, for all $n\in\nset$,
\begin{multline*}
\int \hat\phi_n(\chi)\;W(\rmd\chi)
=\int\left(\int\chi(-t) \;W(\rmd\chi)\right)\phi_n(t)\;\eta(\rmd t)\\
=\int\left(\int\chi(-t)\;\hat X(\rmd\chi)\right)\phi_n(t)\;\eta(\rmd t)
=\int \hat\phi_n(\chi)\;\hat X(\rmd\chi)\;,
\end{multline*}
where we have used~(\ref{eq:W-hatX-equal-on-characters}) in the second
equality. Letting $n\to\infty$, we get $W(A)=\hat X(A)$,
  thus proving the uniqueness of $\hat{X}$. 
\end{proof}

We now prove \Cref{lem:cramer-easy}.

\begin{proof}[\bf Proof of \Cref{lem:cramer-easy}]
  By \Cref{def:integral-cagos}, $X=(X_t)_{t\in\lcag}$ is a centered
  $\cH_0$-valued process
  satisfying~\ref{item:def:functional-weakstat1}
  and~\ref{item:def:functional-weakstat2} in
  \Cref{def:functional-weakstat}. Using the Gramian-isometric property
  of integration with respect to $W$, we get for all $t,h \in \lcag$,
  $\Cov{X_{t+h}}{X_t} = \int \chi(t+h) \overline{\chi(t)} \nu_X(\rmd
  \chi) = \int \chi(h) \nu_X(\rmd \chi)$ which gives
  \ref{item:def:functional-weakstat3} in
  \Cref{def:functional-weakstat} with autocovariance operator
  function $\Gamma$ given
  by~(\ref{eq:traceclass-bochner-general}). Finally, for all
  $\aop \in \cL_b(\cH_0)$, for all $h \in \lcag$, denoting by $f$ the
  density of $\nu$ with respect to $\norm{\nu}_1$, we have
  $$
  \aop \Gamma(h)=\aop  \, \int \chi(h) \,f(\chi)\; \norm{\nu}_1(\rmd
  \chi)=\int \chi(h) \,\aop  \, f(\chi)\; \norm{\nu}_1(\rmd
  \chi)\;,
  $$
  Since the integrand in the last integral has trace-class norm upper
  bounded by $\norm{\aop}_{\cL_b(\cH_0)}$ and $\norm{\nu}_1$ is
  finite we get that $h\mapsto \aop \Gamma(h)$ is continuous from
  $\hat\lcag$ to $\cS_1(\cH_0)$ by dominated convergence. The
  continuity of $h \mapsto \tr(\aop \Gamma(h))$ follows, thus showing
  the last point of \Cref{def:functional-weakstat}.
\end{proof}

We can now prove  the Kolmogorov isomorphism theorem.
\begin{proof}[\bf Proof of \Cref{thm:kolmo-isomorphism-thm}]
By  \Cref{thm:integral-cagos} and~(\ref{eq:spectral-domain-def}), $I_{\hat{X}}^{\cG_0}$ is a
Gramian-unitary operator from $\widehat{\cH}^{X,\cG_0}$ to
$\cH^{\hat{X},\cG_0}$. Thus to conclude, we only need to show that $\cH^{X,\cG_0}
=\cH^{\hat{X},\cG_0}$. By~(\ref{eq:spectralrep-functional}), we have
for all $\aop\in\cL_b(\cH_0,\cG_0)$ and $t\in\lcag$, $\aop X_t=
I_{\hat{X}}^{\cG_0}(\aop\rme_t)\in\cH^{\hat{X},\cG_0}$,  where $\rme_t
: \chi \mapsto \chi(t)$. Thus, by~(\ref{eq:time-domain-def}), we get that $\cH^{X,\cG_0}
\subset\cH^{\hat{X},\cG_0}$. The definition of $\hat{X}$
in~(\ref{eq:def:hatX-stone}) gives the converse inclusion, which
achieves the proof. 
\end{proof}

The proof of \Cref{cor:general-bochn} relies on the following result providing a way to build a \cagos\ measure $W$ from its intensity measure. 
\begin{theorem}
  \label{thm:gaussian-cagos}
    Let $\cH_0$ be a separable Hilbert space and $(\Lambda, \calA)$ be
    a measurable space.
  Let $\nu$ be a trace-class \povm\ on  $(\Lambda, \calA,\cH_0)$. Then there
  exist a probability space $(\Omega,\cF,\PP)$ and an $\cH_0$-valued
  random \cagos\ $W$ on $(\Lambda, \cA, \Omega, \cF, \PP)$ with
  intensity operator measure $\nu$ such that
  the process $(\pscal{W(A)}{x})_{A\in\cA,x\in\cH_0}$ is a (complex)
  Gaussian process. 
\end{theorem}
\begin{proof}
  Define $\gamma:(\cH_0\times\cA)^2\to\cset$ by of
  $$
  \gamma((x,A);(y,B))= x^\adjoint\nu(A\cap B) y = \gramian {x^\adjoint\1_A}{y^\adjoint\1_B}_{\nu}\;,
  $$
  where we used the Gramian~\eqref{eq:gramian-L2-unbounded} with $\cG_0=\cset$. Then it is easy to see $\gamma$ is hermitian
  non-negative definite in the sense that for all $n\geq1$, 
  $x_1,\dots,x_n\in\cH_0$, $A_1,\dots,A_n\in\cA$ and $a_1,\dots,a_n\in\cset$,
  $$
  \sum_{i,j=1}^na_i\overline{a_j}\gamma((x_i,A_i);(a_j,A_j))\geq0\;.
  $$
  Let $(Z_{x,A})_{(x,A)\in\cH_0\times\cA}$ be the centered
  circularly symmetric Gaussian process complex with covariance
  $\gamma$.  Let $(\phi_n)_{0\leq n< N}$ be a Hilbert basis of $\cH_0$,
  with $N = \dim \cH_0 \in\{1,2,\dots,\infty\}$. It is
  straightforward to show that for all $A\in\cA$,
  $$
  W(A):=\sum_{0\leq n<N}Z_{\phi_n,A}\,\phi_n
  $$
  is well defined in
  $\cH=\cM(\Omega, \cF, \cH_0, \PP)$ and that the so defined $W$ is a
  random \cagos\ with intensity operator measure 
  $\nu$.
\end{proof}

We can now prove \Cref{cor:general-bochn}.

\begin{proof}[\bf Proof of \Cref{cor:general-bochn}]
  The implication
  \ref{item:cor:bochner-gen-proper}$\Rightarrow$\ref{item:cor:bochner-gen-tcpovm}
  follows from \Cref{thm:spectral-analysis-hilbert}. Now suppose
  that~\ref{item:cor:bochner-gen-tcpovm} holds. Let $W$ be the
  Gaussian  \cagos\ measure with intensity operator measure $\nu$
  obtained in \Cref{thm:gaussian-cagos}. Then
  Assertion~\ref{item:cor:bochner-gen-proper} follows from \Cref{lem:cramer-easy}.  
\end{proof}

The proof of \Cref{thm:bochnerop} is mainly contained in \cite{berberian1966naui}. 
\begin{proof}[\bf Proof of \Cref{thm:bochnerop}]
  The equivalence between \ref{itm:pos-def} and \ref{itm:pt} is
  straightforward: to show that
  \ref{itm:pos-def}$\Rightarrow$\ref{itm:pt}, take an arbitrary
  $x\in\cH_0$ with unit norm and set $\aop_i= x\,x_i^\adjoint$  for
  $i=1,\dots,n$. To show that
  \ref{itm:pt}$\Rightarrow$\ref{itm:pos-def}, take, for any
  $x\in\cH_0$, $x_i=\aop_i^\adjoint x$ for
  $i=1,\dots,n$. 
  The equivalence between \ref{itm:pt}, \ref{itm:hnnd} and
  \ref{itm:fourier} is given by \cite[Theorem~3]{berberian1966naui}.
  The lastly stated fact that $\nu$ is uniquely determined by
  \eqref{eq:traceclass-bochner-general} is a consequence of the uniqueness stated
  in the univariate Bochner theorem (recalled in
  \Cref{them:bochner-stadard}) applied to
  $\nu_x:A\mapsto x^\adjoint\nu(A)\,x$ for all $x\in\cH_0$.

  Note that the proof of the implication  \ref{itm:fourier} $\Rightarrow$ \ref{itm:pt} provided in \cite[Theorem~3]{berberian1966naui} uses dilation theory. We hereafter propose an alternative and more elementary proof of this implication.  Suppose that \ref{itm:fourier} holds. The continuity of $\Gamma$
    in \wot\ follows immediately by dominated convergence and we now
    prove that it is of positive type as in \Cref{def:positive-definite-operator-function}.  Take some arbitrary
    $n \in \nset^*$, and $x_1,\cdots, x_n \in \cH_0$. Let us define
    the $\cset^{n \times n}$-valued measure $\mu$ on on
    $(\hat{\lcag}, \borel(\hat{\lcag}))$ by
  $$
  \mu (A )= \begin{bmatrix} \pscal{\nu(A) x_1}{x_1}_{\cH_0} & \cdots & \pscal{\nu(A) x_n}{x_1}_{\cH_0} \\
    \vdots & \ddots & \vdots \\
    \pscal{\nu(A) x_1}{x_n}_{\cH_0} & \cdots & \pscal{\nu(A) x_n}{x_n}_{\cH_0}
  \end{bmatrix}\;.
  $$
  Then, by the Cauchy-Schwartz inequality, for all $i,j \in
  \iseg{1,n}$,  the $\cset$-valued measure
  $\mu_{i,j} : A \mapsto [\mu(A)]_{i,j}$ admits a density $f_{i,j}$
  with respect to the non-negative finite measure
  $\norm{\mu}_1 : A \mapsto \norm{\mu(A)}_1 = \tr(\mu(A))$ and   the matrix-valued function
  $f : \chi \mapsto (f_{i,j}(\chi))_{1 \leq i,j \leq n}$ is
  $\norm{\mu}_1$-a.e. hermitian, non-negative semi-definite since,
  for all $a\in\cset^n$ and $A\in \borel(\hat{\lcag})$,
  $$
  \int_A a^\adjoint f(\chi)a \; \norm{\mu}_1(\rmd\chi)=a^\adjoint\mu(A)a=\left(\sum_{i=1}^na_ix_i\right)^\adjoint\nu(A)\left(\sum_{i=1}^na_ix_i\right)\geq0\;.
  $$
 Then, for
  all $t_1, \cdots, t_n \in \lcag$, we have 
\begin{align*}
  \sum_{i,j = 1}^n \pscal{\Gamma(t_i-t_j) x_i}{x_j}_{\cH_0}
  = \sum_{i,j = 1}^n \int \chi(t_i) \overline{\chi(t_j)} \, \mu_{i,j}(\rmd \chi) 
  &= \sum_{i,j = 1}^n \int \chi(t_i) \overline{\chi(t_j)} f_{i,j}(\chi) \, \norm{\mu}_1(\rmd \chi) \\
  &= \int \underbrace{\left(\sum_{i,j = 1}^n  \chi(t_i) \overline{\chi(t_j)} f_{i,j}(\chi)\right)}_{\geq 0 \; \norm{\mu}_1\text{-a.e.}} \, \norm{\mu}_1(\rmd \chi) \\
  &\geq 0 \; .
\end{align*}
The first line follows from \ref{itm:fourier}, the definition of
$\mu_{i,j}$ above and the definition of the integral as given by  Theorem~9 in \cite{Berberian1966NotesOS}.
The second line follows from the definition of $f_{i,j}$
  and the third line from the above property of the matrix-valued
  function $f$. Hence we have shown \ref{itm:pt} and the proof of the
  implication is concluded. 
\end{proof}

We conclude this section with the proof of \Cref{cor:general-bochn-add}

\begin{proof}[\bf Proof of \Cref{cor:general-bochn-add}]
  By definition of the autocovariance operator function of a weakly
  stationary process, it is straightforward to see that 
  Assertion~\ref{item:cor:general-bochn-add1} implies
  Assertion~\ref{item:cor:general-bochn-add2}. Now, suppose that
  Assertion~\ref{item:cor:general-bochn-add2} holds. By
  \Cref{cor:general-bochn}, we only need to prove
  Assertion~\ref{item:cor:bochner-gen-tcpovm} of
  \Cref{cor:general-bochn}, which is what we almost get in Assertion~\ref{itm:fourier} of
  \Cref{thm:bochnerop}, except that we have to prove that, additionally, $\nu$ is
  trace-class. Applying~(\ref{eq:traceclass-bochner-general}) with
  $h=0$, we get that $\nu(\hat\lcag)=\Gamma(0)$, which is assumed to
  be in $\cS_1(\cH_0)$ in the present
  Assertion~\ref{item:cor:general-bochn-add2}. Thus by \Cref{lem:traceclasspovm},
  $\nu$ is indeed trace-class and the proof is concluded.
\end{proof}

\subsection{Filtering random \cagos\
  measures and proofs of \Cref{sec:pointw-comp-oper}}
\label{sec:proofs-filtering}
We start with the proof of  \Cref{prop:carac-processtransfer}.
\begin{proof}[\bf Proof of \Cref{prop:carac-processtransfer}]
  Recall that $X\in\processtransfer{\Phi}(\Omega,\cF,\PP)$ if and only
  if $X\in\widehat{\cH}^X$.  By \Cref{def:modular-scptral-domain} and
  applying \Cref{prop:charac-L2O}\ref{item:prop:charac-L2O-1} with
  $\Lambda=\hat{\lcag}$, $\cA=\borel(\hat{\lcag})$ and $\nu=\nu_X$, we
  have $X\in\widehat{\cH}^X$ if and only if
  $\Phi g_X^{1/2}\in\cL^2(\tore,\btore,\cS_2(\cH_0),\mu)$.  Since
  $g_X\in L^1(\tore,\btore,\cS_1^+(\cH_0),\mu)$ and
  $\Phi\in\simplemeasfunctionsetarg{\tore, \borel(\tore), \cH_0}$, we
  easily get that $\lambda\mapsto\Phi(\lambda) g_X^{1/2}(\lambda)$ is
  simply measurable on $(\tore,\btore)$ and valued in
  $\cS_2(\cH_0)$. By \Cref{lem:meas-schatten}, we get that
  $\Phi g_X^{1/2}$ is measurable from $(\tore,\btore)$ to
  $(\cS_2(\cH_0),\borel(\cS_2(\cH_0)))$.  Now, for all
  $\lambda\in\tore$, since $\Phi(\lambda)\in\cL_b(\cH_0)$ and
  $g_X(\lambda)\in\cS_1^+(\cH_0)$, we have
  $\norm{\Phi(\lambda) g_X^{1/2}(\lambda)}_2^2=\norm{\Phi(\lambda)
    g_X(\lambda)\Phi^\adjoint(\lambda)}_1$. The result follows.
\end{proof}
Having a clear description of the modular
spectral domain in \Cref{sec:proofs-spectral-rep},
\Cref{prop:compo-inversion-filters-ts} can be seen as a particular
instance of the composition and inversion of operator-valued functions
filtering a general random \cagos\ measure, which we now thoroughly
examine. We first state a straightforward result, whose proof is
omitted.
\begin{proposition}
   \label{cor:filt-rand-cagos} Let $(\Lambda, \calA)$ be a measurable space, $\cH_0$, $\cG_0$ two
  separable Hilbert spaces. Let $W$ be an $\cH_0$-valued random \cagos\ measure on
  $(\Lambda,\calA,\Omega,\cF,\PP)$ with intensity operator
  measure $\nu_W$. Let
  $\Phi \in \mathscr{L}^2(\Lambda, \calA, \cO(\cH_0,\cG_0),
  \nu_W)$. Then the mapping
  \begin{equation}
    \label{eq:filt-W-cagos}
V : A\mapsto   \int_A \Phi \, \rmd W= I_W^{\cG_0}(\1_A\Phi)    
  \end{equation}
is a $\cG_0$-valued random \cagos\ measure on
$(\Lambda,\calA,\Omega,\cF,\PP)$ with intensity operator measure
$$
\Phi \nu_W \Phi^\adjoint : A \mapsto \int_A \Phi \rmd \nu_W
\Phi^\adjoint\;,
$$
which is a well defined trace-class \povm
\end{proposition}
The \cagos\ $V$ defined by~(\ref{eq:filt-W-cagos}) is said to admit the
density $\Phi$ with respect to $W$, and we write $\rmd V = \Phi\rmd W$
(or, equivalently, $V(\rmd \lambda)=\Phi(\lambda)W(\rmd \lambda)$). In
the following definition, based on \Cref{cor:filt-rand-cagos}, we use
a signal processing terminology where $\Lambda$ is seen as a set of
frequencies and $\Phi$ is seen as a transfer operator function acting
on the (random) input frequency distribution $W$.  
\begin{definition}[Filter $\filterspecrep{\Phi}(W)$ acting on a random \cagos\ measure in $\specreptransfer{\Phi}$]\label{def:density-cagos}
Let $(\Lambda, \calA)$ be a measurable space, $\cH_0$, $\cG_0$ two
separable Hilbert spaces. For a given transfer operator function
  $\Phi\in \weakmeasfunctionsetarg{\Lambda, \calA, \cH_0, \cG_0}$, we
  denote by $\specreptransfer{\Phi}(\Omega,\cF,\PP)$ the set of 
 $\cH_0$-valued random \cagos\ measures on
  $(\Lambda,\calA,\Omega,\cF,\PP)$ whose intensity operator
  measures $\nu_W$ satisfy  $\Phi \in \mathscr{L}^2(\Lambda, \calA, \cO(\cH_0,\cG_0),
  \nu_W)$. Then, for any $W\in\specreptransfer{\Phi}(\Omega,\cF,\PP)$, we say that the
  random $\cG_0$-valued \cagos\ measure $V$
  defined by~(\ref{eq:filt-W-cagos}) is the output of the filter with
  transfer operator function $\Phi$ applied to the input \cagos\ measure $W$, and we denote
  $V=\filterspecrep{\Phi}(W)$. 
\end{definition}

Now, consider the filtering, using \Cref{def:density-cagos},
$
V=\filterspecrep{\Phi}(W)
$
for a random \cagos\ measure $W$ and a transfer function
$\Phi\in \mathscr{L}^2(\Lambda, \calA, \cO(\cH_0,\cG_0), \nu_W)$. The
goal of this section is, given another separable Hilbert space
$\cI_0$, to characterize the transfer functions $\Psi$ valued in
$\cO(\cG_0,\cI_0)$ which can be used to filter the \cagos\ measure
$V$. Taking $W$ to be the Cramér representation $\hat X$ of a weakly
stationary process $X$, we will get the already stated
\Cref{prop:compo-inversion-filters-ts} on the composition of linear filters
as a byproduct.

According to \Cref{cor:filt-rand-cagos}, $\Psi$ must be
square-integrable with respect to $\nu_V = \Phi \nu_W \Phi^\adjoint$ and
this turns out to be equivalent to checking that $\Psi\Phi$ is square
integrable with respect to $\nu_W$ as stated in the following
theorem. We recall that $\Psi\Phi$ is the pointwise composition, that
is, $\Psi\Phi: \lambda \mapsto\Psi(\lambda)\circ\Phi(\lambda)$ and is
defined whenever the image of $\Phi(\lambda)$ is included in the
domain of $\Psi(\lambda)$.
We first need the following lemma, which will be used in the proof of \Cref{thm:l2-compo-characterization}. 
\begin{lemma}\label{lem:compo-s2}
  Let $\cH_0$, $\cG_0$, $\cI_0$ be separable Hilbert spaces and $\aop \in
  \cO(\cG_0, \cI_0)$, $\bop \in \cK(\cH_0, \cG_0)$. The following assertions hold. 
  \begin{enumerate}[label=(\roman*)]
  \item\label{item:lem:compo-s2:1} $\range(\abs{\bop^\adjoint}) = \range(\bop)$.
  \item\label{item:lem:compo-s2:2} If $\range(\bop) \subset \domain(\aop)$, then $(\aop \bop) (\aop \bop)^\adjoint = (\aop \abs{\bop^\adjoint})(\aop \abs{\bop^\adjoint})^\adjoint$.
  \item\label{item:lem:compo-s2:3} If $\range(\bop) \subset \domain(\aop)$, then $\aop \bop \in \cS_2(\cH_0, \cI_0)$ if and only if $\aop \abs{\bop^\adjoint} \in \cS_2(\cG_0, \cI_0)$. In this case $\norm{\aop\bop}_2 = \norm{\aop \abs{\bop^\adjoint}}_2$. 
  \end{enumerate}
\end{lemma}
\begin{proof}
For convenience, we only consider the case where the spaces have
infinite dimensions. The singular values decomposition of $\bop$
yields for two orthonormal sequences
$(\psi_n)_{n\in\nset}\in\cG_0^\nset$ and
$(\phi_n)_{n\in\nset}\in\cH_0^\nset$, 
$$
\bop = \sum_{n \in \nset}  \sigma_n \psi_n \otimes \phi_n
\quad \text{and} \quad
\abs{\bop^\adjoint} = \sum_{n \in \nset} \sigma_n \psi_n \otimes \psi_n \; .
$$
\noindent\textbf{Proof of~\ref{item:lem:compo-s2:1}.}   We have $\range(\bop) =
\set{\sum_{n \in \nset} \sigma_n x_n \psi_n}{(x_n)_{n \in \nset} \in
  \ell^2(\nset)} = \range(\abs{\bop^\adjoint})$.\\
\noindent\textbf{Proof of~\ref{item:lem:compo-s2:2}.}   
 By the first point both compositions $\aop\bop$ and $\aop
 \abs{\bop^\adjoint}$ make sense. Consider the polar decomposition of
 $\bop^\adjoint$ : $\bop^\adjoint = U \abs{\bop^\adjoint}$, with $U=\sum_{n \in \nset}  \phi_n \otimes \psi_n$.
 Then $\bop = \abs{\bop^\adjoint} U^\adjoint$ and 
    $$
    (\aop\bop)(\aop\bop)^\adjoint = \left(\aop
      \abs{\bop^\adjoint}\right) U^\adjoint U \left(\aop
      \abs{\bop^\adjoint}\right)^\adjoint
    = \left(\aop \abs{\bop^\adjoint}\right) \left(\aop \abs{\bop^\adjoint}\right)^\adjoint\;,
    $$
    where we used that $\abs{\bop^\adjoint}U^\adjoint U =
    \abs{\bop^\adjoint}$.\\
    \noindent\textbf{Proof of~\ref{item:lem:compo-s2:3}.}
    We have that  $\aop \bop \in \cS_2(\cH_0, \cI_0)$ if and only if $(\aop\bop)(\aop\bop)^\adjoint \in \cS_1(\cI_0)$, which is equivalent to $\aop \abs{\bop^\adjoint} \in \cS_2(\cG_0,\cI_0)$ by the previous point.
  \end{proof}
We can now derive the main result of this section.
\begin{theorem}\label{thm:l2-compo-characterization}
  Let $(\Lambda, \calA)$ be a measurable space, $\cH_0$, $\cG_0$,
  $\cI_0$ separable Hilbert spaces and $\nu$ a trace-class \povm\ on
  $(\Lambda, \calA, \cH_0)$. Let $\Phi \in \mathscr{L}^2(\Lambda,
  \calA, \cO(\cH_0,\cG_0), \nu)$ and $\Psi \in
  \weakmeasfunctionsetarg{\Lambda, \calA, \cG_0, \cI_0}$.
  Define
  $\Phi \nu \Phi^\adjoint:A\mapsto\int_A\Phi\rmd\nu\Phi^\adjoint=\gramian{\1_A\Phi}{\1_A\Phi}_\nu$,
  which is a trace-class \povm\ on
  $(\Lambda, \calA, \cG_0)$.
  Then
  \begin{equation}
    \label{eq:1thm:l2-compo-characterization}
  \Psi \in \mathscr{L}^2(\Lambda, \calA, \cO(\cG_0,\cI_0), \Phi \nu
  \Phi^\adjoint) \Leftrightarrow \Psi \Phi \in \mathscr{L}^2(\Lambda, \calA,
  \cO(\cH_0,\cI_0), \nu)     \;.
  \end{equation}
  Moreover, the following assertions hold.
  \begin{enumerate}[label=(\alph*)]
  \item\label{item:1ass:thm:l2-compo-characterization} For all $\Psi, \Theta \in \mathscr{L}^2(\Lambda, \calA, \cO(\cG_0,\cI_0), \Phi \nu \Phi^\adjoint)$, we have
    $(\Psi \Phi) \nu (\Theta \Phi)^\adjoint = \Psi (\Phi \nu \Phi^\adjoint)
    \Theta^\adjoint$. 
  \item\label{item:2ass:thm:l2-compo-characterization} The mapping $\Psi
    \mapsto \Psi \Phi$ is a well defined Gramian-isometric operator from
    $\mathsf{L}^2(\Lambda, \calA, \cO(\cG_0, \cI_0), \Phi \nu \Phi^\adjoint)$
    to $\mathsf{L}^2(\Lambda, \calA, \cO(\cH_0, \cI_0), \nu)$. 
  \item\label{item:3ass:thm:l2-compo-characterization} If moreover 
    $\Phi$ is injective $\norm{\nu}_1$-a.e., then $\Phi^{-1} \in \mathscr{L}^2(\Lambda, \calA, \cO(\cG_0, \cH_0), \Phi \nu \Phi^\adjoint)$,  where we define
  $\Phi^{-1}(\lambda):= \left(\Phi(\lambda)_{|\domain(\Phi(\lambda)) \to
      \range(\Phi(\lambda))}\right)^{-1}$ with domain $\range(\Phi(\lambda))$
  for all $\lambda\in\{\Phi\text{ is injective}\}$ and $\Phi^{-1}(\lambda)=0$
  otherwise.
  \end{enumerate}
\end{theorem}
\begin{proof}
  Let $\mu$ be a dominating measure for $\norm{\nu}_1$ and
  $g = \frac{\rmd \nu}{\rmd \mu}$, then, by definition of
  $\Phi \nu \Phi^\adjoint$, $\mu$ also dominates
  $\norm{\Phi \nu \Phi^\adjoint}_1$ and
  $\frac{\rmd \Phi \nu \Phi^\adjoint}{\rmd \mu} = (\Phi g^{1/2})(\Phi
  g^{1/2})^\adjoint $. Hence,
  $\left( \frac{\rmd \Phi \nu \Phi^\adjoint}{\rmd \mu} \right)^{1/2} =
  \abs{(\Phi g^{1/2})^\adjoint}$ and we get, by
  \Cref{prop:charac-L2O},
  \begin{align*}
    \Psi \in \mathscr{L}^2(\Lambda, \calA, \cO(\cH_0,\cI_0), \Phi \nu \Phi^\adjoint)
    &\Leftrightarrow
      \begin{cases} \range \abs{(\Phi g^{1/2})^\adjoint }  \subset \domain(\Psi) \; \; \mu\text{-a.e.} \\
        \Psi \abs{(\Phi g^{1/2})^\adjoint}  \in \cL^2(\Lambda, \calA, \cS_2(\cG_0, \cI_0), \mu)
      \end{cases} \\
    &\Leftrightarrow
      \begin{cases} \range g^{1/2} \subset \domain(\Psi \Phi) \; \; \mu\text{-a.e.} \\
        \Psi \Phi g^{1/2} \in \cL^2(\Lambda, \calA, \cS_2(\cH_0, \cI_0), \mu)
      \end{cases} \\
    &\Leftrightarrow \Psi \Phi \in \mathscr{L}^2(\Lambda, \calA, \cO(\cH_0,\cI_0), \nu) \;,
  \end{align*}
  where the second equivalence comes from \Cref{lem:compo-s2} and the
  fact that for all $\lambda \in \Lambda$,
  $\domain(\Psi(\lambda) \Phi(\lambda))$ is the preimage of
  $\domain(\Psi(\lambda))$ by $\Phi(\lambda)$ which gives that
  $\range(g^{1/2}(\lambda)) \subset \domain(\Psi(\lambda) \Phi(\lambda))$ if and
  only if $\range(\Phi(\lambda) g^{1/2}(\lambda))
  \subset\domain(\Psi(\lambda))$. 

  We now prove
  Assertion~\ref{item:1ass:thm:l2-compo-characterization}. For
  $\Psi, \Theta \in \mathscr{L}^2(\Lambda, \calA, \cO(\cG_0, \cI_0),
  \Phi \nu \Phi^\adjoint)$ and $A \in \calA$, we have
  \begin{align*}
    (\Psi \Phi) \nu (\Theta \Phi)^\adjoint (A)
    = \int_A \left(\Psi \Phi g^{1/2}\right)\left(\Theta \Phi g^{1/2}\right)^\adjoint \, \rmd \mu 
    &= \int_A \left(\Psi \abs{(\Phi g^{1/2})^\adjoint}\right)\left(\Theta \abs{(\Phi g^{1/2})^\adjoint}\right)^\adjoint \, \rmd \mu \\
   &=  \Psi (\Phi \nu \Phi^\adjoint) \Theta^\adjoint (A) \;,
  \end{align*}
  where  the second equality holds by \Cref{lem:compo-s2}. 
  Assertion~\ref{item:1ass:thm:l2-compo-characterization} follows as
  well as 
  Assertion~\ref{item:2ass:thm:l2-compo-characterization} by taking
  $A = \Lambda$. Finally, to show
  Assertion~\ref{item:3ass:thm:l2-compo-characterization}, suppose
  that $\Phi$ is injective $\norm{\nu}_1$-a.e. then
  $\Phi^{-1} \Phi : \lambda \mapsto \Id_{\cH_0} \indi{\Phi(\lambda)
    \text{ is injective}}$ is in
  $\mathscr{L}^2(\Lambda, \calA, \cO(\cH_0), \nu)$ which gives that
  $\Phi^{-1} \in \mathscr{L}^2(\Lambda, \calA, \cO(\cG_0, \cH_0), \Phi
  \nu \Phi^\adjoint)$ by
  Assertion~\ref{item:1ass:thm:l2-compo-characterization}.
\end{proof}
We deduce the following corollary on the composition and inversion for random \cagos\ measures. 
\begin{corollary}[Composition and inversion of filters on random \cagos\ measures]\label{prop:compo-inversion-filters-cagos}
  Let $(\Lambda, \calA)$ be a measurable space, $\cH_0$, $\cG_0$ two
separable Hilbert spaces, and
  $\Phi\in \weakmeasfunctionsetarg{\Lambda, \calA, \cH_0,
    \cG_0}$. Let $W\in\specreptransfer{\Phi}(\Omega,\cF,\PP)$ with intensity operator
  measure $\nu_W$. Then three
  following assertions hold.
  \begin{enumerate}[label=(\roman*)]
  \item \label{item:cor:inclusion-filtered-space} For any
    separable Hilbert space  $\cI_0$, we have $\cH^{\filterspecrep{\Phi}(W), \cI_0} \gramianisometricallyembedded \cH^{W, \cI_0}$.
  \item \label{item:cor:composition-filters-hilbert} For any
    separable Hilbert space  $\cI_0$ and all $\Psi \in
    \weakmeasfunctionsetarg{\Lambda, \calA, \cG_0, \cI_0}$, we have
    $W \in \specreptransfer{\Psi \Phi}(\Omega,\cF,\PP)$ if and only if
    $\filterspecrep{\Phi}(W) \in \specreptransfer{\Psi}(\Omega,\cF,\PP)$, and in this
    case, we have
  \begin{equation}\label{eq:compo-filt-cagos}
  \filterspecrep{\Psi} \circ \filterspecrep{\Phi}(W) = \filterspecrep{\Psi \Phi}(W).
  \end{equation}
\item \label{item:cor:inversion-filters-hilbert} Suppose that $\Phi$
  is injective $\norm{\nu_W}_1$-a.e. Then
  $W = \filterprocess{\Phi^{-1}}\circ\filterprocess{\Phi}(W)$, where
  $\Phi^{-1}$ is defined as in
  Assertion~\ref{item:3ass:thm:l2-compo-characterization} of
  \Cref{thm:l2-compo-characterization}. Moreover,
  Assertion~\ref{item:cor:inclusion-filtered-space} above holds with
 $\gramianisometricallyembedded$ replaced by
  $\gramianisomorphic$.
  \end{enumerate}
\end{corollary}
\begin{proof}
{\noindent\bf Proof of
  Assertion~\ref{item:cor:inclusion-filtered-space}.} This follows from Assertion~\ref{item:2ass:thm:l2-compo-characterization} of \Cref{thm:l2-compo-characterization} and \Cref{thm:integral-cagos}. \\
{\noindent\bf Proof of Assertion~\ref{item:cor:composition-filters-hilbert}.}
 If $W \in \specreptransfer{\Phi}(\Omega,\cF,\PP)$, then the equivalence between
  $W \in \specreptransfer{\Psi \Phi}(\Omega,\cF,\PP)$ and
  $\filterspecrep{\Phi}(W) \in \specreptransfer{\Psi}(\Omega,\cF,\PP)$ is just another
  formulation of the equivalence
  \eqref{eq:1thm:l2-compo-characterization} with $\nu=\nu_W$. Suppose
  that it holds and
  set  $V := \filterspecrep{\Phi}(W)$ so that
  $\nu_V=\Phi\nu\Phi^\adjoint$ and
  \eqref{eq:compo-filt-cagos} means that,
  for all $\Psi \in \cL^2(\Lambda, \calA, \cO(\cG_0, \cI_0), \nu_V)$
  and $A \in \calA$,
  $\int_A \Psi \, \rmd V = \int_A \Psi \Phi \, \rmd W$.
  Replacing $\Psi$ by $\Psi\1_A$, it is sufficient to show this
  identity with $A=\Lambda$. Using that the integral with respect to a
  random \cagos\ measure is Gramian-isometric and
  Assertion~\ref{item:2ass:thm:l2-compo-characterization} of
  \Cref{thm:l2-compo-characterization}, the mappings  $\Psi\mapsto\int
  \Psi \, \rmd V$ and $\Psi\mapsto\int \Psi \Phi \, \rmd W$ are 
  Gramian-isometric from $\mathsf{L}^2(\Lambda, \calA, \cO(\cG_0,
  \cI_0), \Phi \nu_W \Phi^\adjoint)$ to $\cM(\Omega, \cF, \cI_0, \PP)$. Hence by \Cref{thm:l2op-density},
  they coincide on the whole space if they coincide on all
  $\Psi=\1_A\aop$ for $A\in\cA$ and $\aop\in\cL_b(\cG_0, \cI_0)$. To
  conclude the proof of
  Assertion~\ref{item:cor:composition-filters-hilbert}, it is thus
  enough to prove that,  for all $A\in\cA$ and $\aop\in\cL_b(\cG_0,
  \cI_0)$,
  $
    \int_A \aop \, \rmd V = \int_A \aop \Phi \, \rmd W \;.
  $
  This identity follows from the definition of $V$ and the fact that
  on both sides the operator $\aop$ can be moved in front of the
  integrals. This latter fact directly follows from the definition of the
  integral for the left-hand side and for the right-hand side when
  $\Phi=\1_B$ for some $B\in\cA$, which extends to all $\Phi$ by
  observing that $\Phi\mapsto \int \aop \Phi \, \rmd W$ and
  $\Phi\mapsto  \aop\int \Phi \, \rmd W$ are continuous on  $\mathsf{L}^2(\Lambda, \calA, \cO(\cH_0, \cG_0), \nu_W)$.
  
  {\noindent\bf Proof of
    Assertion~\ref{item:cor:inversion-filters-hilbert}.} Continuing
  with the setting of the proof of the previous point, we now suppose that $\Phi$ is injective $\norm{\nu_W}_1$-a.e.
  Assertions~\ref{item:3ass:thm:l2-compo-characterization} and
  \ref{item:1ass:thm:l2-compo-characterization} of
  \Cref{thm:l2-compo-characterization} give that
  $\Phi^{-1} \in \mathscr{L}^2(\Lambda, \calA, \cO(\cG_0, \cH_0),
  \nu_V)$ (\textit{i.e.} $V \in \specreptransfer{\Phi^{-1}}(\Omega,\cF,\PP)$) and
  $\Phi^{-1} \nu_V \left(\Phi^{-1}\right)^\adjoint = \nu_W$. Hence,
  writing Relation \eqref{eq:compo-filt-cagos} with
  $\Psi = \Phi^{-1}$, we get
  $\filterspecrep{\Phi^{-1}}(V) = \filterspecrep{\Phi^{-1} \Phi}(W) =
  W$. Moreover, reversing the roles of $W$ and $V$ in
  assertion~\ref{item:cor:inclusion-filtered-space} gives the
  embedding $\cH^{W, \cI_0} \gramianisometricallyembedded
  \cH^{\filterspecrep{\Phi}(W), \cI_0}$ which, with
  Assertion~\ref{item:cor:inclusion-filtered-space}, allow us to
  conclude that $\cH^{W, \cI_0}\gramianisomorphic
  \cH^{\filterspecrep{\Phi}(W), \cI_0}$. 
\end{proof}
We conclude this section with the proof of
\Cref{prop:compo-inversion-filters-ts}. 
\begin{proof}[\bf Proof of \Cref{prop:compo-inversion-filters-ts}]
  Using the Gramian-unitary operator between the
  modular time domain and the modular spectral domain, this result is
  a direct application of \Cref{prop:compo-inversion-filters-cagos}
  with $\Lambda=\hat\lcag$ and $\cA=\borel(\hat{\lcag})$ and
  $W=\hat X$.
\end{proof}
\subsection{Proofs of \Cref{sec:cram-karh-loeve}}
\label{sec:proofs-crefs-karh}

The goal of this section is to provide  a proof of
\Cref{lem:eigendecomp-tc-povm}. Before that, let us recall essential facts
about the diagonalization of compact positive operators.
Let $\cH_0$ be a separable Hilbert space of dimension $N \in \{1, \cdots, +\infty\}$, $(\Lambda, \cA)$ be a
measurable space and
$\Phi \in \simplemeasfunctionsetarg{\Lambda, \cA, \cH_0}$ such that
for all $\lambda \in \Lambda$, $\Phi(\lambda) \in \cS_1^+(\cH_0)$. Then, in this case, for any $\lambda \in \Lambda$, $\Phi(\lambda)$
admits the eigendecomposition 
\begin{equation}\label{eq:eigendecomposition}
\Phi(\lambda) = \sum_{0 \leq n < N} \sigma_n(\lambda) \phi_n(\lambda) \otimes \phi_n(\lambda) \;,
\end{equation}
where the series converges in operator norm and the family
$(\phi_n(\lambda))_{0 \leq n  < N}$ is orthonormal. Moreover, we have
$$
\tr(\Phi(\lambda))= \sum_{0 \leq n < N} \sigma_n(\lambda)  <
+\infty\;.
$$
The following theorem shows that such a decomposition can be
constructed in a way which makes the eigenvalues and eigenvectors
measurable as functions of $\lambda$. We recall that the  weak topology on
$\cH_0$ is defined as the smallest topology which makes the functions
$\set{x^\adjoint}{x \in \cH_0}$ continuous.

\begin{theorem}\label{thm:measurable-eigendecomp}
  Let $\cH_0$ be a separable Hilbert space and $(\Lambda, \cA)$ be a
  measurable space.
  Let $\Phi \in \simplemeasfunctionsetarg{\Lambda, \cA, \cH_0}$ such
  that for all $\lambda \in \Lambda$,
  $\Phi(\lambda) \in \cS_1^+(\cH_0)$. Then the pairs
  $\set{(\sigma_n, \phi_n)}{0 \leq n < N}$ in
  \eqref{eq:eigendecomposition} can be taken so that for all
  $0 \leq n < N$, $\sigma_n$ is measurable from $(\Lambda, \cA)$ to
  $(\rset^+, \borel(\rset^+))$ and $\phi_n$ is measurable from
  $(\Lambda,\cA)$ to $(\cH_0, \borel(\cH_0))$.
\end{theorem}
\begin{proof}
  The construction of the eigenvalues and eigenvectors is done
  iteratively using the Measurable Maximum Theorem
  (see \cite[Theorem~18.19]{hitchiker-infinite-dimension}) on
  $\Lambda\times\bar B_{0,1}$, where $\bar B_{0,1}$ denotes the closed
  unit ball of $\cH_0$, which is compact metrizable for the weak
  topology by the Banach-Alaoglu theorem and \cite[Theorem~6.32]{hitchiker-infinite-dimension}. As in
  \cite[Definition~17.1]{hitchiker-infinite-dimension}, a
  \emph{correspondence} $\varphi$ from $\Lambda$ to $\bar B_{0,1}$,
  denoted by $\varphi : \Lambda \twoheadrightarrow \bar B_{0,1}$, is a
  mapping which assigns each element of $\Lambda$ to a subset of
  $\bar B_{0,1}$.
  
  \noindent\underline{\bf Construction of $(\sigma_1,\phi_1)$ :} Define 
  $$
  f : \fundef{\Lambda \times \bar B_{0,1} & \to & \rset_+ \\ (\lambda, x) & \mapsto & \pscal{\Phi(\lambda)x}{x}_{\cH_0}}\;.
  $$
  Then, for all $x\in\bar B_{0,1}$, the mapping $\lambda\mapsto f(\lambda,x)$ is measurable. Moreover, for all $\lambda\in\Lambda$, considering the eigendecomposition of $\Phi(\lambda)$, the mapping $x\mapsto f(\lambda,x)$ can be written as the uniform limit of continuous functions (for the weak topology on $\bar B_{0,1}$) hence it is continuous for the weak topology on $\bar B_{0,1}$. Finally, the correspondence
  $$
  \varphi : \fundef{\Lambda \twoheadrightarrow \bar B_{0,1} \\ \lambda \mapsto \bar B_{0,1}}
  $$
  is weakly measurable (in the sense of
  \cite[Definition~18.1]{hitchiker-infinite-dimension}) with nonempty
  compact values (for the weak topology).  Therefore the Measurable
  Maximum Theorem \cite[Theorem~18.19]{hitchiker-infinite-dimension}
  gives that
  $m : \lambda \mapsto \max_{x \in \bar B_{0,1}} f(\lambda,x)$ is
  measurable and that there exists a function
  $g : \Lambda \to \bar B_{0,1}$ such that for all
  $\lambda \in \Lambda$,
  $g(\lambda) \in \argmax_{x \in \bar B_{0,1}} f(\lambda,x)$ and $g$
  is measurable from $\Lambda$ to $\bar B_{0,1}$ endowed with the
  Borel $\sigma$-field $\borel_w(\cH_0)$ generated by the weak topology.  This implies
  the usual measurability with respect to the $\sigma$-field $\borel(\cH_0)$ generated by the norm topology  because, for all $x\in\cH_0$, the mapping $y\mapsto\norm{x-y}_{\cH_0}$ is measurable  from  $(\cH_0,\borel_w(\cH_0))$ to $(\rset_+,\borel(\rset_+))$. 
  We set $\sigma_0 = m$ and $\phi_0 = g$. Then, from the definitions
  of $f, m$ and $g$, that $\sigma_0(\lambda)$ is the largest
  eigenvalue of $\Phi(\lambda)$ and that $\phi_0(\lambda)$ is an
  eigenvector with eigenvalue $\sigma_0(\lambda)$.
  
  \noindent\underline{\bf Construction of $(\sigma_n, \phi_n)$ :}
  Assume we have constructed $n$ measurable functions
  $\sigma_0, \cdots, \sigma_{n-1}$ and $\phi_0, \cdots, \phi_{n-1}$
  satisfying for all $\lambda \in \Lambda$,
  $\sigma_0(\lambda) \geq \cdots \geq \sigma_{n-1}(\lambda)$, and
  $(\phi_0(\lambda), \cdots, \phi_{n-1}(\lambda))$ is an orthonormal
  family where for all $0 \leq i \leq n-1$,
  $\phi_i(\lambda) \in \ker(\Phi(\lambda) -
  \sigma_i(\lambda)\Id_{\cH_0})$. Then, as in the initialization step, the function
  $$
  f : \fundef{\Lambda \times \bar B_{0,1} & \to & \rset_+ \\ (\lambda, x) & \mapsto & \pscal{\Phi(\lambda) x}{x}_{\cH_0} - \sum_{i=1}^{n-1} \sigma_i(\lambda) \abs{\pscal{x}{\phi_i(\lambda)}_{\cH_0}}^2}\;.
  $$
  is measurable in $\lambda$ and continuous in $x$ (for the weak topology) and the correspondence 
  $$
  \varphi : \fundef{\Lambda & \twoheadrightarrow & \bar B_{0,1} \\
  \lambda & \mapsto & \bar B_{0,1} \cap \lspan{\phi_0(\lambda), \cdots, \phi_{n-1}(\lambda)}^\perp} 
$$
is weakly measurable (in the sense of
\cite[Definition~18.1]{hitchiker-infinite-dimension}) because of
\cite[Corollary~18.8 and Lemma~18.2]{hitchiker-infinite-dimension})
and the fact that
$\varphi(\lambda) = \set{x \in \bar B_{0,1}}{\sum_{i=0}^{n-1}
  \abs{\pscal{x}{\phi_i(\lambda)}_{\cH_0}}^2 = 0}$ and has nonempty
compact values (because $\varphi(\lambda)$ is a closed subset of
$\bar B_{0,1}$ for the weak topology hence is compact for this
topology). Hence, as previously, the Measurable Maximum Theorem gives that
$m : \lambda \mapsto \max_{x \in \varphi(\lambda)} f(\lambda,x)$ is
measurable and that there exists a measurable function
$g : \Lambda \to \cH_0$ such that for all $\lambda \in \Lambda$,
$g(\lambda) \in \argmax_{x \in \varphi(\lambda)} f(\lambda,x) $. We
set $\sigma_n = m$ and $\phi_n = g$. Then, from the definitions of
$f, m$ and $g$, we get that $\sigma_n(\lambda) \leq \sigma_{n-1}(\lambda)$ is
the $(n+1)$-th largest eigenvalue of $\Phi(\lambda)$ (because it is
the largest eigenvalue of
$\Phi(\lambda) - \sum_{i=0}^{n-1} \sigma_i(\lambda) \phi_i(\lambda)
\otimes \phi_i(\lambda)$) and that $\phi_n(\lambda)$ is an eigenvector
with eigenvalue $\sigma_n(\lambda)$ and is orthogonal to
$\phi_0, \cdots, \phi_{n-1}$.
\end{proof}
We can now prove \Cref{lem:eigendecomp-tc-povm}.
\begin{proof}[\bf Proof of \Cref{lem:eigendecomp-tc-povm}]
  We provide a proof in the case where $N=\infty$ as the finite
  dimensional case is easier.  Let
  $f\in L^1(\Lambda,\cA, \cS_1^+(\cH_0), \mu)$ be the density of $\nu$
  with respect to $\mu$. We assume without loss of generality that
  $f(\lambda) \in \cS_1(\cH_0)^+$ for all $\lambda \in \hat\lcag$
  (rather than for $\mu$-almost every $\lambda$). Using
  \Cref{thm:measurable-eigendecomp} we can write
\begin{equation}\label{eq:eigendecomposition-specdens}
  f(\lambda) = \sum_{n = 0}^{+\infty} \sigma_n(\lambda) \phi_n(\lambda) \otimes \phi_n(\lambda) \;,
\end{equation}
where $(\sigma_n(\lambda))_{n\in\nset}$ is non-decreasing and
converges to zero and $(\phi_n(\lambda))_{n\in\nset}$
satisfies~\ref{item:lem:eigendecomp-tc-povmii}. Moreover, for all
$\lambda\in\Lambda$,
$\sum_n\sigma_n(\lambda)=\norm{f(\lambda)}_1<\infty$, and we get
Assertions~\ref{item:lem:eigendecomp-tc-povmi}
and~\ref{item:lem:eigendecomp-tc-povmiii}.

It only remains to
prove~\ref{item:lem:eigendecomp-tc-povm1}--\ref{item:lem:eigendecomp-tc-povm3},
which we now proceed to
do. By (\ref{eq:eigendecomposition-specdens}) and the previously proved assertions, we get that for all $n \in \nset$ and all $\lambda \in \Lambda$, $\norm{\phi_n^\adjoint f^{1/2}(\lambda)}_2^2 =  \sigma_n(\lambda) \leq \norm{f(\lambda)}_1$. Hence $\phi_n^\adjoint f^{1/2} \in L^2(\Lambda, \cA, \cS_2(\cH_0, \cset), \nu)$ and \Cref{prop:charac-L2O} gives that $\phi_n^\adjoint \in \mathsf{L}^2(\Lambda, \cA, \cO(\cH_0, \cset), \nu)$ and  for all $n,p\in\nset$,
$$
\pscal{\phi_n^\adjoint}{\phi_p^\adjoint}_{\nu}
=\int \phi_n^\adjoint f \phi_p\;\rmd\mu
= \begin{cases}
  0 & \text{ if $n\neq p$,}\\
  \int\sigma_n\;\rmd\mu & \text{otherwise.} \; , 
\end{cases}
$$
where the last equality comes from (\ref{eq:eigendecomposition-specdens}) and the previously proved assertions.

Similarly, for all $n \in \nset$, $\phi_n \otimes \phi_n f^{1/2} \in L^2(\Lambda, \cA, \cS_2(\cH_0), \nu)$, hence by \Cref{prop:charac-L2O}, we have $\phi_n \otimes \phi_n \in \mathsf{L}^2(\Lambda, \cA, \cO(\cH_0), \nu)$ and for all $n,p\in\nset$,
\begin{align*}
\gramian{\phi_n\otimes\phi_n}{\phi_p\otimes\phi_p}_{\nu}
=\int (\phi_n\otimes\phi_n) f (\phi_p\otimes\phi_p)\;\rmd\mu
=
\begin{cases}
  0 & \text{ if $n\neq p$,}\\
  \int\sigma_n\,(\phi_n\otimes\phi_n)\;\rmd\mu & \text{otherwise,}
\end{cases}
\end{align*}
which proves Assertion~\ref{item:lem:eigendecomp-tc-povm2}. Now
observe that, for all $\lambda\in\Lambda$,
$
\left(\sum_{n=0}^\infty\phi_n(\lambda)\otimes\phi_n(\lambda)\right)f(\lambda)=
f(\lambda)\left(\sum_{n=0}^\infty\phi_n(\lambda)\otimes\phi_n(\lambda)\right)= f(\lambda)
$.
This yields
$
\norm{\sum_{n=0}^\infty\phi_n\otimes\phi_n-\Id_{\cH_0}}_{\nu} =0
$,
and thus Assertion~\ref{item:lem:eigendecomp-tc-povm3} holds, which
concludes the proof.
\end{proof}
A first consequence of \Cref{lem:eigendecomp-tc-povm} is the following.
\begin{remark}
  \label{rem:harmonic-pca-filtering}
Applying \Cref{lem:eigendecomp-tc-povm} to the trace-class \povm\
$\nu_X$, we deduce that
\begin{equation}
  \label{eq:CKL-second}
\hat{X}= \filterspecrep{\left(\sum_{0\leq n < N} \phi_n\otimes\phi_n\right)}(\hat{X}) = \sum_{0 \leq n < N} \filterspecrep{\phi_n\otimes\phi_n}(\hat{X})\;,
\end{equation}
where $(\filterspecrep{\phi_n\otimes\phi_n}(\hat{X}))_{0 \leq n < N}$
are uncorrelated random \cagos's on
$(\hat\lcag,\borel(\hat\lcag),\cH_0)$. In other words,
(\ref{eq:CKL-first}) holds both with $\hat{X}$ in the sum sign or out
of it in the right-hand side. Moreover, for all $n\in\nset$,
 $\filterspecrep{\phi_n\otimes\phi_n}(\hat{X})=\filterspecrep{\phi_n}\circ\filterspecrep{\phi_n^\adjoint}(\hat{X})$
and, by~\ref{item:lem:eigendecomp-tc-povm1} of
\Cref{lem:eigendecomp-tc-povm},
$(\filterspecrep{\phi_n^\adjoint}(\hat{X}))_{0 \leq n < N}$ is a
sequence of uncorrelated $\cset$-valued
\caos\ measures. Hence, interpreting~(\ref{eq:CKL-second}) in the time domain,
we get a decomposition of the process
$X=(X_t)_{t\in\lcag}$ based on a collection of the uncorrelated univariate processes
$(\filterprocess{\phi_n^\adjoint}(X))_{0 \leq n < N}$. 
\end{remark}
To conclude, we prove \Cref{prop:hfpca}. 
\begin{proof}[\bf Proof of \Cref{prop:hfpca}]
  Let
  $
  f_X(\chi)=\sum_{0\leq n<N} \sigma_n(\chi) \,\phi_n(\chi) \otimes \phi_n(\chi)
  $
  denote the density of $\nu_X$ with respect to $\mu$ as given by
  \Cref{lem:eigendecomp-tc-povm}. We have, for all $t\in\lcag$ and $\Theta \in
  \mathsf{L}^2(\hat{\lcag}, \borel(\hat{\lcag}), \cO(\cH_0), \nu_X)$,
  $
  \left[\filterprocess{\Theta}(X)\right]_t = \int \chi(t)\,\Theta(\chi)\;\hat{X}(\rmd\chi)\;,
  $
  and thus by isometric isomorphism between the spectral domain and
  the time domain,
  $$
  \PE[\norm{X_t- [\filterprocess{\Theta}(X)]_t}^2] = \int \norm{(\Id_{\cH_0} - \Theta(\chi))f_X^{1/2}(\chi)}_{2}^2 \, \mu(\rmd \chi)\;.
  $$
  The result is then obtained by observing that, for each
  $\chi\in\hat\lcag$,  the norm in the integral is minimal under the
  constraint  $\rank(\Theta(\chi)) \leq q(\chi)$ for $\Theta(\chi) =
  \sum_{0 \leq n < q(\chi) \wedge N} \phi_n(\chi) \otimes \phi_n(\chi)$.  
\end{proof}

\bibliographystyle{plainnat}
\bibliography{biblio.bib}

\end{document}